\documentclass[11pt,reqno]{amsart}
\usepackage{graphicx,verbatim,lineno,titletoc}
\usepackage{amssymb,mathrsfs}
\usepackage{amsmath}
\usepackage{adjustbox}
\usepackage{calc}
\usepackage{array}
\usepackage[font=small]{caption}
\usepackage[T1]{fontenc}
\usepackage{fourier}
\usepackage[english]{babel}
\usepackage{appendix}
\usepackage{enumerate}
\usepackage[english]{babel}
\usepackage{multirow}
\usepackage{float}
\usepackage[all]{xy}
\usepackage{enumitem}
\usepackage[numbers]{natbib}
\usepackage{tikz}
\usetikzlibrary{shapes.multipart,patterns,arrows}
\usepackage{hyperref}
\hypersetup{colorlinks,linkcolor={red},citecolor={olive},urlcolor={red}}
\usepackage[top=1.4in, left=1.4in, right=1.4in, bottom=1.3in]{geometry}
\usepackage{mathtools}
\mathtoolsset{showonlyrefs}

\newcolumntype{M}[1]{>{\centering\arraybackslash}m{#1}}
\newcolumntype{N}{@{}m{0pt}@{}}

\pdfoutput=1

\newtheorem{theorem}{Theorem}[section]
\newtheorem{prop}[theorem]{Proposition}

\newtheorem{lemma}[theorem]{Lemma}
\newtheorem{corollary}[theorem]{Corollary}

\newtheorem{remark}[theorem]{Remark}

\def\liminf{\mathop{\rm lim\,inf}\limits}
\def\limsup{\mathop{\rm lim\,sup}\limits}
\def\one{\mathbf{1}}
\def\N{\mathbb{N}}
\def\Z{\mathbb{Z}}

\def\R{\mathbb{R}}

\def\EE{\mathbb{E}}
\def\PP{\mathbb{P}}
\def\E{\mathcal{E}}
\def\HH{\mathcal{H}}
\def\I{\mathcal{I}}
\def\T{\mathcal{T}}
\def\L{\mathcal{L}}
\def\m{\mathtt{m}}
\def\r{\mathtt{r}}
\def\h{\mathtt{h}}

\def\F{\mathfrak{F}}
\def\eps{\varepsilon}
\def\ess{\mathfrak{S}}
\def\LL{\mathfrak{L}}

\DeclareMathOperator{\Var}{\textnormal{Var}}
\newcommand{\st}{\hspace{.05em}:\hspace{.05em}}

\newtheorem{innercustomthm}{Theorem}
\newenvironment{customthm}[1]
{\renewcommand\theinnercustomthm{#1}\innercustomthm}
{\endinnercustomthm}

\theoremstyle{definition}

\DeclareRobustCommand{\cev}[1]{\reflectbox{\ensuremath{\vec{\reflectbox{\ensuremath{#1}}}}}}

\allowdisplaybreaks

\begin{document}

\title[Phase transition in a soliton cellular automaton]{Double jump phase transition in a soliton cellular automaton}

\author{Lionel Levine}
\address{Lionel Levine, Department of Mathematics, Cornell University, Ithaca, NY 14853.}
\email{\texttt{levine@math.cornell.edu}}

\author{Hanbaek Lyu}
\address{Hanbaek Lyu, Department of Mathematics, University of California, Los Angeles, CA 90095.}
\email{\texttt{colourgraph@gmail.com}}

\author{John Pike}
\address{John Pike, Department of Mathematics, Bridgewater State University, Bridgewater, MA 02324.}
\email{\texttt{john.pike@bridgew.edu}}

\date{\today}


\begin{abstract}
In this paper, we consider the soliton cellular automaton introduced in \cite{takahashi1990soliton} with a random initial 
configuration. We give multiple constructions of a Young diagram describing various statistics of the system in terms of 
familiar objects like birth-and-death chains and Galton-Watson forests. Using these ideas, we establish limit theorems 
showing that if the first $n$ boxes are occupied independently with probability $p\in(0,1)$, then the number of solitons is of 
order $n$ for all $p$, and the length of the longest soliton is of order $\log n$  for $p<1/2$, order $\sqrt{n}$ for $p=1/2$, 
and order $n$ for $p>1/2$. Additionally, we uncover a condensation phenomenon in the supercritical regime: 
For each fixed $j\geq 1$, the top $j$ soliton lengths have the same order as the longest for $p\leq 1/2$, whereas all but the 
longest have order $\log n$ for $p>1/2$. As an application, we obtain scaling limits for the lengths of the $k^{\text{th}}$ 
longest increasing and decreasing subsequences in a random stack-sortable permutation of length $n$ in terms of random walks  
and Brownian excursions. 
\end{abstract}

\maketitle

\section{Introduction}
\label{Introduction}

In 1990, Takahashi and Satsuma proposed a $1+1$ dimensional cellular automaton of filter type called the 
\emph{soliton cellular automaton}, also known as the \emph{box-ball system} \cite{nagai1999soliton,takahashi1990soliton}. 
It is defined as a discrete-time dynamical system $\left(X_{s}\right)_{s\ge 0}$ whose states are binary sequences 
$X_{s}:\N\rightarrow\{0,1\}$ with finitely many $1$'s. 
We may think of the states as configurations of balls in boxes where box $k$ contains a ball at stage $s$ if $X_{s}(k)=1$ and is empty 
if $X_{s}(k)=0$. The update rule $X_{s}\mapsto X_{s+1}$ is defined as follows: At the beginning of stage $s$, each ball has been moved 
a total of $s$ times. To reach stage $s+1$, successively move the leftmost ball which has been moved a total of $s$ times to the first empty 
box on its right, continuing until all balls have been moved. Alternatively, at each stage $s\geq 0$ a `carrier' starts at the origin and sweeps 
rightward to infinity. Each time she encounters an occupied box, she \emph{pushes} the ball to the top of her stack. Each time she encounters 
an empty box and her stack is nonempty, she \emph{pops} any ball from her stack into the box. In keeping with this picture, we will refer 
to the stages of the box-ball system as \emph{sweeps} henceforth.

As a concrete example, the system initially having balls in boxes $2,3,5,6,7,11$  evolves through sweep $s=3$ as
\[
 \setlength\arraycolsep{3pt}
\begin{array}{*{2}{r|lllllllllllllllllllllllll@{\ }}}
            s=0 & & 0 & 1 & 1 & 0 & 1 & 1 & 1 & 0 & 0 & 0 & 1 & 0 & 0 & 0 & 0 & 0 & 0 & 0 & 0 & 0 & 0 & 0 & \ldots \\[\jot]
            1 & & 0 & 0 & 0 & 1 & 0 & 0 & 0 & 1 & 1 & 1 & 0 & 1 & 1 & 0 & 0 & 0 & 0 & 0 & 0 & 0 & 0 & 0 & \ldots \\[\jot]
            2 & & 0 & 0 & 0 & 0 & 1 & 0 & 0 & 0 & 0 & 0 & 1 & 0 & 0 & 1 & 1 & 1 & 1 & 0 & 0 & 0 & 0 & 0 & \ldots \\[\jot]
           3 & & 0 & 0 & 0 & 0 & 0 & 1 & 0 & 0 & 0 & 0 & 0 & 1 & 0 & 0 & 0 & 0 & 0 & 1 & 1 & 1 & 1 & 0 & \ldots \\
\end{array}
\]
\vspace{4pt}

In this model, a (non-interacting) \emph{soliton of length $k$} is defined to be a string of $k$ consecutive $1$'s followed 
by $k$ consecutive $0$'s. During one sweep, such a soliton travels to the right at speed $k$. The physical interpretation is that 
of a traveling wave with velocity equal to its wavelength. 
If a $k$-soliton precedes a $j$-soliton with $j<k$, then the two will eventually collide, resulting in interference. The subsequent states of 
the system depend on the congruence class of their initial distance modulo their relative speed, $k-j$, but solitons are never created or destroyed in 
the course of these interactions. The case of three or more interacting solitons can be described similarly \cite{takahashi1990soliton}. It is easy to see 
that since we have finitely many balls initially, after some finite time the system consists of non-interacting solitons whose lengths are nondecreasing 
from left to right. We will call such a configuration \emph{stable}. This final macrostate of the system can be encoded in the Young diagram 
$\Lambda(X_{0})$ having $j^{\text{th}}$ column equal in length to the $j^{\text{th}}$ longest soliton.  

In this paper, we start the soliton cellular automaton from a random initial configuration and study the limiting shape of the 
resulting Young diagram. We have two parameters, $n\in\N$ and $p\in (0,1)$. Let $X^{n,p}$ 
be a random coloring of $\N$ so that each site in $[1,n]$ is $1$ with probability $p$ and $0$ with probability $1-p$, 
independently of all others, and all sites in $(n,\infty)$ are $0$. Let $\Lambda^{n,p}=\Lambda(X^{n,p})$ be the corresponding random 
Young diagram and denote the lengths of its $i^{\text{th}}$ row and $j^{\text{th}}$ column by $\rho_{i}(n)$ and $\lambda_{j}(n)$, respectively. 
(Thus $\lambda_{j}(n)$ gives the length of the $j^{\text{th}}$ longest soliton and $\rho_{i}(n)$ the number of solitons of length at 
least $i$.) We are going to show that each fixed row has order $n$ for all values of $p$, but the column lengths vary drastically according to whether 
$p$ is less than, equal to, or greater than $1/2$. The asymptotics of the rows and columns of $\Lambda^{n,p}$ are summarized in the following table, 
for which Theorem~\ref{theorem:rows} proves the $\rho$ entries, and Theorem~\ref{theorem:columns} proves the $\lambda$ entries.
For the precise meaning of the Landau notation employed, see Subsection \ref{subsection:Landau}.

\begin{table}[h]
	\centering
	\includegraphics[width=0.8 \linewidth]{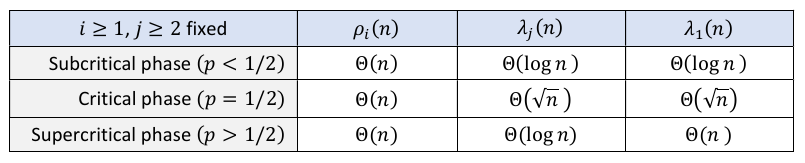}
	\vspace{0.3cm}
	\caption{Double jump phase transition for the order of the longest $j$ solitons ($j$ fixed as $n \to \infty$) in the random box-ball system. 
	All entries are up to constant factors that do not depend on $n$. In the sub- and supercritical phases the $\lambda_j$ are concentrated, 
	and the constant factor depends only on $p$ (and not on $j$).  In the critical phase the $\lambda_j$ are not concentrated, and the constant 
	factor depends on $j$. 
	}
	\label{fig:table}
\end{table}

Erd\H{o}s and R\'{e}nyi coined the term \emph{double jump} to describe the emergence of a giant component in the sparse random graph with $n$ 
vertices, each pair independently joined by an edge with probability $c/n$, where $c>0$ is a parameter. The analogy between random graph components 
and box-ball solitons becomes apparent if we take $c=p/(1-p)$. Then with high probability, all connected components of the Erd\H{o}s-R\'{e}nyi graph are 
of size $O(\log n)$ for $p<1/2$; components of size $\Theta(n^{2/3})$ emerge at $p=1/2$; and for $p>1/2$, the largest component is of size $\Theta(n)$ 
while all the rest have size $O(\log{n})$ \cite{ErdRen}. Except for the exponent $2/3$ (which becomes $1/2$) this is exactly the behavior of the soliton 
lengths in the box-ball system as summarized in the last two columns of Table~\ref{fig:table}.

\subsection{Related work}
\label{subsection:Related work}

There have been some exciting recent developments involving the box-ball system with a bi-infinite random initial configuration. A central question is to 
understand the invariant measures on $\{0,1\}^{\Z}$ under the box-ball dynamics. Ferrari, Nguyen, Rolla, and Wang \cite{ferrari2018soliton} showed that 
the Bernoulli product measure with density $p<1/2$ is invariant and provided a recipe for constructing additional invariant measures based on a soliton 
decomposition of box-ball configurations. Croydon, Kato, Sasada, and Tsujimoto \cite{croydon2018dynamics} found sufficient conditions for invariance using 
Pitman's $2M-X$ transformation and considered extending the box-ball system from $\Z$ to $\R$. See the references for more details.

\subsection{Notation}
\label{subsection:Landau}

We adopt the notation $\R^{+}=[0,\infty)$, $\N=\{1,2,3,\ldots\}$, and $\N_{0}=\N\cup\{0\}$ throughout. 
We employ the Landau notation $O(\cdot),\, \Omega(\cdot),\, \Theta(\cdot)$ in the sense of stochastic boundedness. That is, given 
$\{a_{n}\}_{n=1}^{\infty}\subseteq\R^{+}$ and a sequence $\{W_{n}\}_{n=1}^{\infty}$ of nonnegative random variables, we say that 
$W_{n}=O(a_{n})$ if for every $\eps>0$, there is a $C\in(0,\infty)$ such that $\PP\{W_{n}>Ca_{n}\}<\eps$ for all $n$. We say that 
$W_{n}=\Omega(a_{n})$ if for every $\eps>0$, there is a $c\in(0,\infty)$ such that $\PP\{W_{n}<ca_{n}\}<\eps$ for all $n$, and we say 
$W_{n}=\Theta(a_{n})$ if $W_{n}=O(a_{n})$ and $W_{n}=\Omega(a_{n})$. The constants $c,C$ may depend on $p$ and $\eps$ but not $n$.

\subsection{Main results} 
\label{subsection:Statement}

Fix $p\in(0,1)$, and let $\xi_{1},\xi_{2},\ldots$ be a sequence of i.i.d. random variables
with law $\PP\{\xi_{1}=1\}=p$ and $\PP\{\xi_{1}=-1\}=1-p$. Define $X^{p}\in\{0,1\}^{\N}$ by
\[
X^{p}(k) =\one\{\xi_{k}=1\}, 
\]
and for each $n\in\N$, set $X^{n,p}=X^{p}\one_{[1,n]}$. The interpretation is that $X^{n,p}$ corresponds to an arrangement of balls 
in boxes where boxes $1,\ldots,n$ are each occupied independently with probability $p$, and boxes $n+1,n+2,\ldots$ are empty. 

For each fixed $n\geq 1$ and $p\in (0,1)$, we consider the box-ball system $(X_{s})_{s\geq 0}$ with the random initial configuration 
$X_{0}=X^{n,p}$. Recall that the soliton lengths are denoted by $\lambda_{1}(n)\geq\lambda_{2}(n)\geq\cdots$. 
This information can be summarized by the Young diagram $\Lambda^{n,p}$ whose $j^{\text{th}}$ column has length $\lambda_{j}(n)$. 
The length of its $i^{\text{th}}$ row, $\rho_{i}(n)$, equals the number of solitons in the system having length at least $i$. 
In particular, $\rho_{1}(n)$ gives the total number of solitons. 

Many properties of this Young diagram can be described in terms of the simple random walk $\left\{ S_{k}\right\} _{k=0}^{\infty}$ 
defined by $S_{0}=0$ and $S_{k}=\xi_{1}+\cdots+\xi_{k}$. Our first result shows that the $i$ longest rows are of order $n$ for any 
$p\in (0,1)$.

\begin{customthm}{1}
\label{theorem:rows}
Let $X^{n,p}$ and $S_{k}$ be as above. Then the following statements hold.
\begin{description}[leftmargin=!,labelwidth=\widthof{\bfseries [(ii)]}]
	\item[(i)]{(SLLN for rows)} 
        Let $\varsigma=\inf\{k>0\st S_{k}=0\}\in\N\cup\{\infty\}$ be the first return time of $S_{k}$ to $0$. Then for any fixed $i\geq 1$,
		\[
		\frac{\rho_{i}(n)}{n}\rightarrow \PP\Big\{\max_{0\leq k \leq \varsigma} S_{k}=i\Big\}>0 \;\;\; \text{ a.s.} 
		\;\;\; \text{as $\; n\rightarrow \infty$}.
		\]
	\item[(ii)]{(CLT for the first row)} 
		\[
		\frac{\rho_{1}(n)-np(1-p)}{\sqrt{np(1-p)[1-3p(1-p)]}}\Rightarrow Z 
		\]
	where $Z\sim\mathcal{N}(0,1)$, the standard normal distribution.
\end{description}	
\end{customthm}

Denote by $C(\R)$ the space of continuous functions $f:\R\rightarrow\R$ endowed with the topology of uniform convergence on compact sets, 
and let $C_{0}^{+}(\R)$ be the subspace of $C(\R)$ consisting of nonnegative compactly supported functions $f$ such that 
$f\equiv 0$ on $(-\infty,0]$. For any closed interval $I\subseteq \R$ containing $0$, denote by $C(I)$ and $C_{0}^{+}(I)$ the space of restrictions 
$f|_{I}$ where $f\in C(\R)$ and $f\in C_{0}^{+}(\R)$, respectively. For $b\in I$, define the operator 
$\E_{b}:C(I)\rightarrow C(I)$ by 
\[
\E_{b}(f)(t) = f(t) - \min_{ b\wedge t \leq s \leq b\vee t } f(s),
\]
where $y\wedge z=\min(y,z)$ and $y\vee z=\max(y,z)$. We call $b$ the \emph{pivot} of $\E_{b}$. 
Define $\m:C_{0}^{+}(I)\rightarrow \R^{+}$ by $\m(g)=\sup\{x\in I\st g(x)=\max(g)\}$, the location of the rightmost global maximum of $g$. 
Finally, define the \emph{excursion operator} $\E$ on $C_{0}^{+}(I)$ by $\E(g)=\E_{\m(g)}(g)$. 
See Figure \ref{Young diagram:column} for an illustration.

We now state the main result of the paper.

\begin{customthm}{2}
\label{theorem:columns}
Let $X^{n,p}$ be as above and set $\theta=(1-p)/p$. Let $\lambda_{j}(n)$ denote the $j^{\text{th}}$ longest soliton length. 
\begin{description}[leftmargin=!,labelwidth=\widthof{\bfseries [(iii)]}]
	\item[(i)] (Subcritical phase) For $p<1/2$, $\lambda_j(n)$ is concentrated around  
	$\mu_{n}:=\log_{\theta}\left(\frac{(1-2p)^2}{1-p} n\right)$ for each fixed $j\geq 1$ in the sense that for all $x\in\R$, 
	\begin{align}
	\exp(-\theta^{-x}) \sum_{k=0}^{j-1} \frac{\theta^{-k(x+1)}}{k!}
	 &\leq \liminf_{n\rightarrow \infty}\PP\left\{\lambda_{j}(n) \leq x+\mu_{n} \right\}  \\
	&\leq \limsup_{n\rightarrow \infty}\PP\left\{ \lambda_{j}(n) \leq x+\mu_{n} \right\} 
	\leq \exp(-\theta^{-(x+1)})\sum_{k=0}^{j-1} \frac{\theta^{-kx}}{k!}.
	\end{align}
	In particular, $\lambda_{j}(n)=\Theta(\log n)$.
	\medskip

	\item[(ii)] (Critical phase) For $p=1/2$, let $B=\{B_{t}\}_{0\leq t \leq 1}$ be a standard Brownian 
	motion on $[0,1]$. 	Then for each fixed $j\geq 1$, 
	\[
	n^{-1/2}[\lambda_{1}(n),\lambda_{2}(n),\ldots,\lambda_{j}(n)] \Rightarrow 
	[\max |B|,\,\max \E(|B|),\ldots,\,\max \E^{j-1}(|B|)],
	\]
	In particular, $\lambda_{j}(n)=\Theta(\sqrt{n})$.\\
	Furthermore, for any integers $j,k\geq 1$, 
	\[
	\lim_{n\rightarrow\infty} n^{-k/2}\EE[ (\lambda_{j}(n))^{k} ]  = \EE\left[ \left( \max \E^{j-1}(|B|) \right)^{k} \right].
	\]				
	\medskip

	\item[(iii)] (Supercritical phase)	For $p>1/2$, 
	\begin{align}
	\frac{\lambda_{1}(n)-(2p-1)n}{2\sqrt{p(1-p)n}}\Rightarrow Z\sim\mathcal{N}(0,1).
	\end{align}
	Furthermore, there exists a constant $c=c(p)>0$ such that
	\begin{align}
	\PP\{|\lambda_{1}(n)-(2p-1)n|\geq x\} \leq c\exp(-x^{2}/(8n)),
	\end{align}
	and for all $j\geq 2$, $\lambda_{j}(n)$ is concentrated around $\hat{\mu}_{n}:=\log_{\theta^{-1}}\left(\frac{(1-2p)^{2}}{p}n \right)$ 
	in the sense that for all $x\in\R$,  
	\begin{align}
	 \left(  \exp(-\theta^{-\frac{x}{2}}) \sum_{k=0}^{j-1} \frac{\theta^{-k(\frac{x}{2}+1)}}{k!}\right) -c\theta^{\frac{x}{8}} 
	& \leq \liminf_{n\rightarrow \infty}\PP\left\{ \lambda_{j}(n) \leq x+\hat{\mu}_{n} \right\}  \\
	& \hspace{-1cm} \leq \limsup_{n\rightarrow \infty}\PP\left\{ \lambda_{j}(n) \leq x+\hat{\mu}_{n} \right\} \leq 
	\left( \exp(-\theta^{-(\frac{3x}{2}+1)}) \sum_{k=0}^{j-1} \frac{\theta^{-\frac{3kx}{2}}}{k!} \right) + c\theta^{\frac{x}{8}}.
	\end{align}
	In particular, $\lambda_{1}(n)=\Theta(n)$ and $\lambda_{j}(n)=\Theta(\log n)$ if $j\geq 2$.
\end{description}
\end{customthm}

We call the statement in Theorem \ref{theorem:columns} (iii) a \emph{condensation phenomenon} because in the 
supercritical regime, a linear number of balls condense into the longest soliton while the next $j$ longest solitons each have $\Theta(\log n)$ balls. 

The methods that we develop in this paper to study the box-ball system yield several interesting results on lengths of monotone subsequences 
in random pattern avoiding permutations. The study of statistics involving longest increasing or decreasing subsequences in different types of random 
permutations has a long history and rich connections to many other fields \cite{romik2015surprising}. In the context of the box-ball system, the class of 
$312$-avoiding permutations arises naturally, and we are able to generalize some classical results on such permutations in multiple directions.  

For each $n\in \N$, let $\ess_{n}$ be the set of all permutations on $\{1,2,\ldots,n\}$. Given two permutations $\sigma\in \ess_{n}$ and $\tau\in \ess_{k}$ 
with $1< k\leq n$, we say that $\sigma$ is \textit{$\tau$-avoiding} if no subsequence of $\sigma$ has the same relative order as $\tau$. (For example, a 
permutation is $312$-avoiding if there is no subsequence of the form $z,x,y$ with $x<y<z$.) Denote by 
$\ess^{\tau}_{n}$ the set of all $\tau$-avoiding permutations in $\ess_{n}$. Note that $\sigma$ is $\tau$-avoiding if and only if $\sigma^{-1}$ is 
$\tau^{-1}$-avoiding. (In particular, $\sigma$ is 231-avoiding if and only if $\sigma^{-1}$ is 312-avoiding.) Given a permutation $\sigma\in \ess_{n}$, 
define integers $\lambda_{1},\ldots,\lambda_{k}$ (resp. $\rho_{1},\ldots, \rho_{k}$) recursively so that 
$\lambda_{1}(\sigma)+\cdots+\lambda_{k}(\sigma)$ (resp. $\rho_{1}(\sigma)+\cdots+\rho_{k}(\sigma)$) equals the length of the longest subsequence 
in $\sigma$ obtained by taking a disjoint union of $k$ decreasing (resp. increasing) subsequences. 

In a classic work \cite{rotem1981stack}, Rotem studied properties of 231-avoiding permutations chosen uniformly at random among all such 
permutations of a given length. He showed that if $\Sigma^{n}$ is a permutation in $\ess^{231}_{n}$ chosen uniformly at random, then
\[
\EE[\rho_{1}(\Sigma^{n})] = (n+1)/2, \quad \EE[\lambda_{1}(\Sigma^{n})] = \sqrt{\pi n}+O(1).
\]
Our next theorem is an extension of the above result both to the higher moments and to `subsequent' longest increasing and decreasing 
subsequences of $\Sigma^{n}$. 

\begin{customthm}{3}
\label{theorem:permutations}
Let $\Sigma^{n}$ be a uniformly chosen random $312$- (or $231$-) avoiding permutation of length $n$. 
\begin{description}[leftmargin=!,labelwidth=\widthof{\bfseries [(ii)]}]
	\item[(i)] Suppose that $T^{n}_{1},T^{n}_{2},\ldots,T^{n}_{i}$ is a sequence of rooted trees where $T^{n}_{1}$ is chosen uniformly 
	at random 	among all rooted plane trees on $n+1$ nodes, and for $r\geq 1$, $T^{n}_{r+1}$ is obtained from $T^{n}_{r}$ by deleting all 
	leaves. Then
	\[
	\qquad[\rho_{1}(\Sigma^{n}),\rho_{2}(\Sigma^{n}),\ldots,\rho_{i}(\Sigma^{n})] =_{d} 
	[\text{$\#$ of leaves in $T^{n}_{1}$}, \text{$\#$ of leaves in $T^{n}_{2}$}, \ldots, \text{$\#$ of leaves in $T^{n}_{i}$}]. 
	\]
	
	\item[(ii)] Let $\{S_{k}\}_{k=0}^{\infty}$ be a simple symmetric random walk with $S_{0}=0$ and let 
	$\varsigma=\inf\{k>0\st S_{k}=0\}$ be the time of its first return to $0$. Then for any fixed $i\geq 1$,
	\[
	\frac{\rho_{i}(\Sigma^{n})}{2n}\rightarrow \PP\Big\{\max_{0\leq k \leq \varsigma} S_{k}=i\Big\}>0  \;\;\; \text{ a.s.} 
	\;\;\; \text{as $\; n\rightarrow \infty$}.
	\]
	
	\item[(iii)] Let $B^{\text{ex}}=(B_{t}^{\text{ex}})_{0\leq t \leq 1}$ be a standard Brownian excursion on $[0,1]$. Then for each fixed $j\geq 1$, 
	\[
	\qquad\quad n^{-1/2}[\lambda_{1}(\Sigma^{n}),\lambda_{2}(\Sigma^{n}),\ldots,\lambda_{j}(\Sigma^{n})] \Rightarrow 
	\sqrt{2}[\max B^{\text{ex}},\,\max \E(B^{\text{ex}}),\ldots,\max \E^{j-1}(B^{\text{ex}})].
	\]
	Furthermore, for any integers $j,k\geq 1$, 
	\[
	\lim_{n\rightarrow\infty} n^{-k/2}\EE[ (\lambda_{j}(\Sigma^{n}))^{k} ]  = 2^{k/2} \, \EE\left[ \left( \max \E^{j-1}(B^{ex}) \right)^{k} \right].
	\]
\end{description}	
\end{customthm}

We remark that given a $312$-avoiding permuatation $\sigma$, we can actually interpret $\lambda_{k}(\sigma)$ as the length of the longest decreasing 
subsequence after successively deleting an arbitrary longest decreasing subsequence $k-1$ times. For the rows, we can interpret $\rho_{k}$ similarly but 
the longest increasing subsequence we delete at each step must be a special one; see Proposition \ref{prop:permutation_Greene}. 
Note that such an interpretation is not valid for general permutations.

\subsection{Outline and organization} 
\label{subsection:Outline}

Broadly speaking, we proceed by observing correspondences between various combinatorial objects related to box-ball configurations, such as 
Motzkin paths, rooted forests, and $312$-avoiding permutations; see Figure \ref{fig:correspondences}. We can then interpret 
the rows and columns of the Young diagram associated with a box-ball configuration in terms of these objects (Table \ref{fig:table2_diagram}). 
This allows us to reformulate the original soliton problem in other languages and vice versa. 

\begin{figure*}[h]
	\centering
	\includegraphics[width=.95 \linewidth]{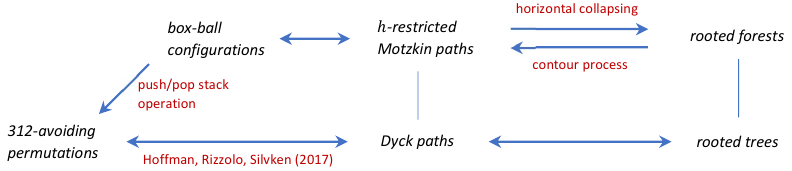}
	\caption{Correspondences and inclusions between six combinatorial objects. Objects in the same row are in bijective correspondence, 
		     while the vertical lines indicate that the bottom objects are special cases of the top.}
	\label{fig:correspondences}
\end{figure*}

For us, Motzkin paths provide the most useful framework, especially in the random setting. This is because the random box-ball 
configuration $X^{n,p}$ can be viewed as the increment sequence of the first $n$ steps of a simple random walk driven by the $\text{Bernoulli}(p)$ 
measure. The corresponding random ($h$-restricted) Motzkin path is the same simple random walk except that downstrokes at height $0$ are 
censored. The problem then essentially boils down to studying properties of the excursions of such censored random walks. The results for random 
Motzkin paths can then be translated back to solitons or permutations. 

\begin{table}[h]
	\centering
	\includegraphics[width=.95 \linewidth]{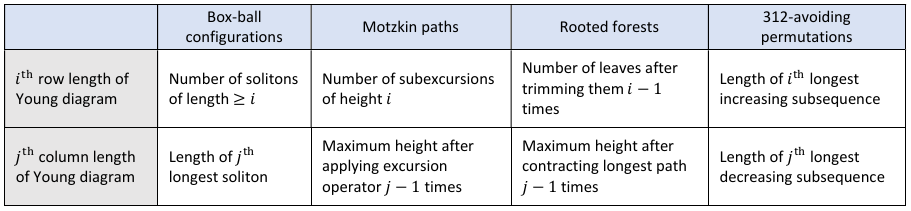}
	\vspace{0.2cm}
	\caption{Interpretation of rows and columns of the Young diagram associated with four combinatorial objects.}
	\label{fig:table2_diagram}
\end{table}
 
This paper is organized as follows: In Section \ref{section:Young Diagram}, we describe relations between box-ball configurations, Motzkin paths, 
and rooted forests, and show how to construct the Young diagram from these objects. In Section \ref{section:random box-ball and rows}, we 
discuss a correspondence between random box-ball configurations, a birth-and-death chain, and a Galton-Watson forest. We prove Theorem 
\ref{theorem:rows} in Section \ref{section:SLLN for Rows}, and the proof of Theorem \ref{theorem:columns} is given in Sections \ref{section:Subcritical}, 
\ref{section:Critical}, and \ref{section:Supercritical}. In Section \ref{section:Permutation}, we discuss a connection between box-ball configurations and 
pattern-avoiding permutations and prove Theorem \ref{theorem:permutations}. Finally, in Appendix \ref{Section:proof of combinatorial Lemmas}, we prove 
the three lemmas stated in Subsection \ref{subsection:Youngdiagram_construction} along with some results concerning 312-avoiding permutations.

\section{Constructing the time-invariant Young Diagram}
\label{section:Young Diagram}

In this section, we establish some important statements about the Young diagram which will be used crucially in later sections.

\subsection{Motzkin paths}
\label{subsection:Motzkin paths}

We begin with a bijection between box-ball states and a class of lattice paths we call \emph{$h$-restricted Motzkin}, a minor variant of the 
bijection with Dyck paths in \cite{torii1996combinatorial}. A function $f:\R^{+}\rightarrow \R$ is a \textit{lattice path} if $f$ is the 
linear interpolation of some function $\gamma:\N_{0}\rightarrow \Z$. A lattice path $f$ is called \textit{Motzkin} if it is nonnegative, compactly supported, 
and consists only of $(1,1)$, $(1,-1)$, and $(1,0)$ steps (which we refer to as `upstrokes,' `downstrokes,' and `$h$-strokes,' respectively). 
We say that a Motzkin path is \emph{$h$-restricted} if its $h$-strokes occur only on the $x$-axis. Finally, if $\Gamma$ is a Motzkin path, we write 
$\Gamma_{k}$ for $\Gamma(k)$, $k\geq 0$.

The aforementioned bijection maps a (compactly supported) configuration $X:\N\rightarrow\{0,1\}$ to the $h$-restricted Motzkin path $\Gamma(X)$ 
defined by linear interpolation of its values on $\N_{0}^{2}$, which are given recursively by $\Gamma(X)_{0}=0$ and 
\[ 
\Gamma(X)_{k+1} - \Gamma(X)_k = 
\begin{cases}
	+1 & \text{if $X(k+1)=1$} \\
	-1 & \text{if $X(k+1)=0$ and $\Gamma(X)_{k}\geq 1$} \\
	0 & \text{if $X(k+1)=0$ and $\Gamma(X)_{k}=0$}
\end{cases} 
\]
for all $k\geq 0$. The inverse map from paths to configurations proceeds by writing a $0$ for each downstroke or $h$-stroke and a $1$ for each 
upstroke. See Figure \ref{fig:Motzkin} for an illustration. 

\begin{figure*}[h]
	\centering
	\includegraphics[width=0.9 \linewidth]{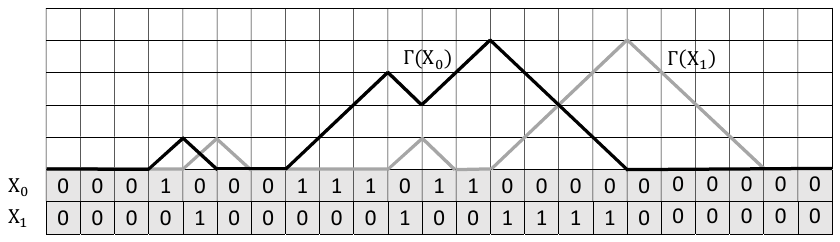}
	\caption{The top shaded row shows an initial box-ball configuration $X_{0}$, and the black path is $\Gamma(X_{0})$. 
                     To update, balls are placed at downstrokes of $\Gamma(X_{0})$, resulting in the configuration $X_{1}$ and the 
                     grey path $\Gamma(X_{1})$.}
	\label{fig:Motzkin}
\end{figure*}

The shape of this path tells us how to evolve the system by a single sweep: A ball is picked up at each upstroke and deposited at each downstroke. 
Specifically, label the balls $1,\ldots,m$ from left to right. (This labeling applies only to states, not the system as a 
whole. In subsequent sweeps, the label of a particular ball may change.) Then the $j^{\text{th}}$ upstroke occurs at the site where the carrier picks 
up the ball labeled $j$. The site at which she deposits ball $j$ is determined by drawing a horizontal line from the center of the $j^{\text{th}}$ 
upstroke to the first downstroke on its right. From this description, we see that the height of the path at any site equals the number of balls in the 
carrier's stack after she visits that site. When the sweep is completed, the new state of the system corresponds to the unique path formed by converting 
each downstroke to an upstroke and then adding $h$-strokes and downstrokes so that it is $h$-restricted Motzkin. 

Formally, the box-ball state $X_{s+1}$ is given in terms of the Motzkin path $\Gamma(X_s)$ by 
\begin{equation}
\label{eq:Motzkin_to_boxball}
X_{s+1}(k+1) = \one\big\{ \Gamma(X_{s})_{k+1}-\Gamma(X_{s})_{k}=-1 \big\}
\end{equation}
where $\one$ is the indicator function.

\subsection{Hill-flattening and excursion operators }
\label{subsection:Youngdiagram_construction}

We now describe two methods of constructing a Young diagram $\Lambda(\Gamma)$ associated with a (not necessarily $h$-restricted) Motzkin path 
$\Gamma$. As usual, we denote the $i^{\text{th}}$ row and $j^{\text{th}}$ column by $\rho_{i}(\Gamma)$ and $\lambda_{j}(\Gamma)$. 

First we give the row-wise construction using the \emph{hill-flattening operator} $\HH$ defined on the set of all Motzkin paths. To begin, we 
say that an interval $[a,b]$ with $a,b\in\N_{0}$ and $a\leq b$ is a \emph{hill interval} of the Motzkin path $\Gamma$ if for every $c\in[a,b]$, 
$\Gamma_{a-1}=\Gamma_{c}-1=\Gamma_{b+1}$. We write $\I(\Gamma)$ for the collection of all hill intervals of $\Gamma$, 
and denote the number of hill intervals by $\rho(\Gamma)=|\I(\Gamma)|$. The hill-flattening operator $\HH$ is then defined by
\[
\HH(\Gamma)_{k}=
\begin{cases}
\Gamma_{k}-1 & \text{if $k$ is contained a hill interval of $\Gamma$} \\
\Gamma_{k} & \text{otherwise}
\end{cases}
\]
for $k\in\N_{0}$. 

A \textit{hill} of $\Gamma$ is the graph of $\Gamma$ over $[a-1,b+1]$ with $[a,b]$ a hill interval. Thus hills consist of a single upstroke, followed by zero 
or more $h$-strokes, followed by a single downstroke. Call a hill with no $h$-strokes a  \emph{peak}. Then the hill-flattening operator $\HH$, 
when applied to $\Gamma$, \textit{flattens} each hill of $\Gamma$ by replacing the upstroke and downstroke with $h$-strokes and then lowering any 
intermediate $h$-strokes so that the path remains connected. 

Note that each application of the hill-flattening operator decreases the maximum height of the Motzkin path by $1$ and never increases 
the number of hills, so
\[
\rho(\Gamma)\geq \rho(\HH(\Gamma)) \geq \rho(\HH^{2}(\Gamma)) \geq \cdots 
\geq \rho(\HH^{\max \Gamma}(\Gamma))=0.
\]
We define the Young diagram $\Lambda(\Gamma)$ associated to the Motzkin path $\Gamma$ as having $i^{\text{th}}$ row of length 
$\rho_{i}(\Gamma)=\rho( \HH^{i-1}(\Gamma))$ for $1 \leq i \leq \max \Gamma$. Here repeated applications of $\HH$ are 
denoted by $\HH^{j+1}(f)=\HH\big(\HH^{j}(f)\big)$ with $\HH^{0}$ the identity operator. In particular, given 
a box-ball configuration $X:\N_{0}\rightarrow \{0,1\}$ of finite support, we can construct the Young diagram $\Lambda(\Gamma(X))$. 
See Figure \ref{Young diagram} for an illustration.

\begin{figure*}[h]
	\centering
	\includegraphics[width=0.9 \linewidth]{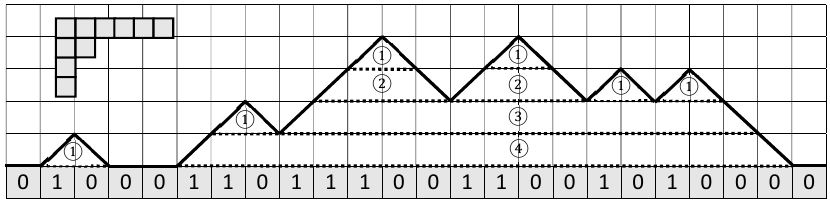}
	\caption{Construction of Young diagram via hill-flattening procedure applied to the Motzkin path associated with box-ball configuration $X$. 		
		    The bottom row is the configuration $X$ and the black path is $\Gamma=\Gamma(X)$. Trapezoidal regions with label $i$ are the hills of 
		    $\HH^{i-1}(\Gamma)$, each of which becomes a distinct cell in the $i^{\text{th}}$ row of $\Lambda(\Gamma)$. The resulting 
		    Young diagram $\Lambda(\Gamma)$ is depicted in the upper left corner.}
	\label{Young diagram}
\end{figure*}

Now consider a box-ball system $(X_{s})_{s\geq 0}$ started from a configuration $X_{0}:\N_{0}\rightarrow \{0,1\}$. The following 
lemma says that for each $s\geq 0$, the corresponding 
Young diagram $\Lambda(\Gamma(X_{s}))$ is independent of $s$ and its column lengths correspond to the lengths of the solitons.

\begin{lemma}
\label{lemma:invariance_Youngdiagram}
$\Lambda(\Gamma(X_{s}))=\Lambda(\Gamma(X_{s+1}))$ for all $s\geq 0$. Moreover, $\Lambda(\Gamma(X_{0}))=\Lambda(X_{0})$.
\end{lemma}

Next, we give the column-wise construction of $\Lambda(\Gamma)$. The key observation is that the $j^{\text{th}}$ longest column length, which 
we denote by $\lambda_{j}$, is obtained by successively applying the excursion operator to $\Gamma$ $j-1$ times and then taking a maximum. 

\begin{lemma}
\label{lemma:columnbyexcursionop}
Let $\Gamma$ be a Motzkin path and let $\lambda_{j}(\Gamma)$ denote the length of the $j^{\text{th}}$ column of $\Lambda(\Gamma)$. Then 
\[
\lambda_{j}(\Gamma) = \max \E^{j-1}(\Gamma), \quad  1\leq j \leq \rho(\Gamma).
\]
In particular, if $(X_{s})_{s\in \N_{0}}$ is a finitely supported box-ball system with initial configuration $X_{0}:\N\rightarrow\{0,1\}$, then 
\[
\lambda_{j}(X_{0}) = \max \, \E^{j-1}(\Gamma(X_{0})).
\]
\end{lemma}
We relegate the proofs of these lemmas, along with that of Lemma \ref{lemma:Lipschitz} below, to 
Appendix \ref{Section:proof of combinatorial Lemmas} in order to maintain the flow of the paper. 

Lemma \ref{lemma:columnbyexcursionop} gives the following column-wise construction of $\Lambda(\Gamma)$.
Let $\m=\m(\Gamma)$ be the location of the rightmost global maximum of $\Gamma$, and set $\lambda_{1}(\Gamma)=\Gamma_{\m}$, 
the maximum height of $\Gamma$. 
To find $\lambda_{2}(\Gamma)$, one first computes $\E(\Gamma)$ by traversing $\Gamma$ to the left and right of $\m$ as follows: 
Starting with height $0$ at $\m$, move to the left, remaining at height $0$ until the first local minimum, and then record the sequence of 
strokes until the original lattice path returns to the height of this minimum. Then repeat the process, staying at height $0$ until encountering 
a local minimum and then recording the path of the second such excursion. Continue to the beginning of the path and then repeat the procedure 
moving to the right from $\m$. The resulting path precisely records all `subexcursions' which are not subsumed by the maximum $(\m,\Gamma_{\m})$. 
$\lambda_{2}(\Gamma)$, the length of the second column of $\Lambda(\Gamma)$, is equal to the maximum of $\E(\Gamma)$. 
Continuing in this fashion gives $\lambda_{j}(\Gamma)=\max \E^{j-1}(\Gamma)$ for all $j\geq 1$. 

\begin{figure*}[h]
\centering
\includegraphics[width=.9 \linewidth]{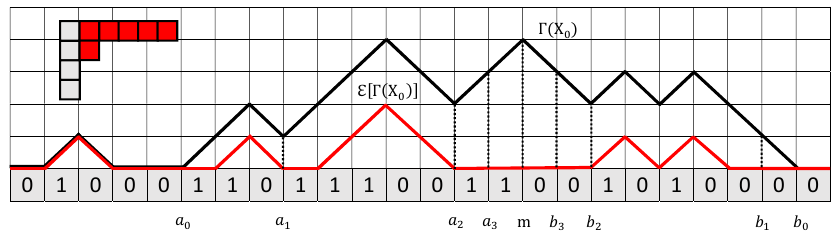}
	\caption{The black path is $\Gamma(X_{0})$, the red path is $\E(\Gamma(X_{0}))$, and $\m$ is the location of the 
		    rightmost maximum of $\Gamma(X_{0})$. The $k^{\text{th}}$ box in the bottom row is $X_{0}(k)$, and the 
		    red Young diagram is constructed from the red path via hill flattening. The $a$ and $b$ terms are defined in 
		    the proof of Proposition \ref{prop:numberofhills}.}
\label{Young diagram:column}
\end{figure*}

In light of Lemma \ref{lemma:columnbyexcursionop}, it is natural to call $\max \E^{j-1}$ the $j^{\text{th}}$ 
\textit{column length functional}. A crucial advantage of extracting the column length $\lambda_{j}$ from the functional $\max \E^{j-1}$ 
is that this operation is continuous with respect to the topology of $C_{0}^{+}(\R^{+})$ as stated in the lemma below. 
This enables us to take various scaling limits of the system. 

\begin{lemma} 
\label{lemma:Lipschitz}
For any interval $I\subseteq \R^{+}$, functions $f,g\in  C_{0}^{+}(I)$, and $j\geq 1$, 
\[
\Big| \max\E^{j-1}(f) - \max\E^{j-1}(g) \Big| \leq 2\lVert f-g \rVert_{\infty}.
\]
\end{lemma}

\begin{remark}[Depth process with drains]
\textup{
In private communication with Jim Pitman, we learned that an operator equivalent to $\E_{b}$ was used in studying Brownian paths and 
continuum random trees. In our context, given a Motzkin path $\Gamma$, flip it upside down and consider it as a bucket filled to the top with water. 
Given $b\in \R^{+}$, poke a hole at point $(b,-\Gamma(b))$. This will drain some of the water, and $-\E_{b}(\Gamma)(x)$ gives the water 
level at each $x\in \R^{+}$. For instance, the red path in Figure \ref{Young diagram:column} can be obtained from the black one in this way with 
drain at $b=\m(\Gamma)$. A similar procedure can be defined with multiple drains. This operation was applied to Brownian paths to study, for 
example, the line-breaking construction of the continuum random tree in a Brownian excursion \cite{aldous1993continuum}; sampling 
bridges, meanders, and excursions at independent uniform times \cite{pitman1999brownian}; and developments in the tree setting with different 
metaphors such as ``forest growth'' and ``bead crushing'' \cite{pitman2005growth,pitman2015regenerative}.   
}   
\end{remark}

\smallskip

\subsection{Rooted forests}
\label{subsection:Trees}

In this subsection, we develop an alternative perspective for constructing the Young diagram from an associated rooted forest. The idea is to collapse 
a Motzkin path to  a rooted forest by horizontal identification. Intuitively, one paints the underside of the graph of each excursion with glue and then 
compresses it horizontally to obtain a tree. Then the original Motzkin path can be viewed as the contour process (or Harris walk in the random 
setting) of the rooted forest so constructed. This point of view will be especially useful for thinking about arguments in Section 
\ref{section:Supercritical}.

To begin, recall that a \emph{rooted forest} is a sequence of vertex-disjoint plane trees $\{T_{i}\}_{i\geq 1}$ such that each $T_{i}$ 
is rooted at a vertex $\r_{i}\in V(T_{i})$. The \emph{level} of a vertex $v\in T_{i}$ is defined as $\ell(v)=d(v,\r_{i})$ where $d$ is 
the graph distance. Given a Motzkin path $\Gamma$, we define a rooted forest $\F(\Gamma)$ as follows: 
Let $G(\Gamma)=(V,E)$ be the graph with vertex set $V=\left\{(k,\Gamma_{k})\right\}_{k\in \N_{0}}\subset \N_{0}^{2}$ and 
adjacency relation 
\[
(a,\Gamma_{a}) \overset{\text{adj}}{\sim} (b,\Gamma_{b}) \Longleftrightarrow |a-b|=1
\text{ and } \Gamma_{a},\Gamma_{b}\text{ not both }0.
\]
In words, $G(\Gamma)$ is obtained from $\Gamma$ by removing the $h$-strokes at $0$ but retaining all vertices.
Clearly each component of $G(\Gamma)$ is isomorphic to a path beginning and ending at height $0$, and there are only finitely 
many such paths since $\Gamma$ has finite support. Arranging the components from left to right so that their vertex labels are increasing, 
let $P_{i}$ denote the $i^{\text{th}}$ component from the left. 
Define an equivalence relation $\sim$ on the vertex set of $G(\Gamma)$ by 
\[
(a,\Gamma_{a})\sim (b,\Gamma_{b}) \Longleftrightarrow \, 0<\Gamma_{a}=\Gamma_{b}\leq \Gamma_{x} 
\text{ for all } x\in [a,b],
\]
and write $T_{i}=P_{i}/\mathord{\sim}$ for the resulting rooted tree; see Figure \ref{fig:tree diagram}. The rooted forest associated with $\Gamma$ 
is $\F(\Gamma)=\{T_{i}\}_{i\geq 1}$.

\begin{figure*}[h]
	\centering
	\includegraphics[width=0.95 \linewidth]{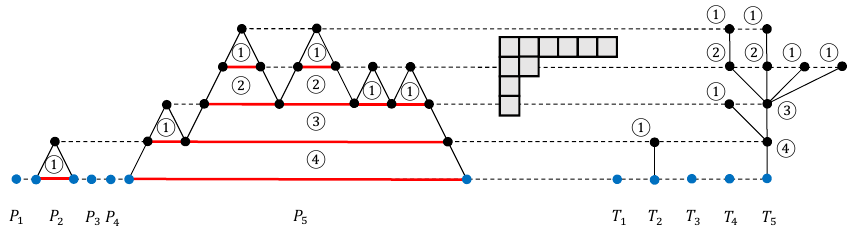}
	\caption{Rooted forest $\F(X_{0})$ corresponding to the box-ball configuration $X_{0}$ given in 
		Figure \ref{Young diagram}. Each connected component of $G(\Gamma(X_{0}))$ (left) becomes a tree rooted 
		at a blue node (right) by identifying vertices connected by the red horizontal lines. 
		Flattening hills of $\Gamma(X_{0})$ corresponds to trimming leaves from $\F(X_{0})$.}
	\label{fig:tree diagram}
\end{figure*}

We can recover $\Gamma$ from $\F(\Gamma)$ by keeping track of the levels of the vertices explored in depth-first search. 
This exploration process begins at the root of $T_{1}$ and visits nodes from bottom to top and from left to right in such a way that 
it backtracks to the parent of the current node only if there is no child left to visit. After exhausting all nodes in $T_{1}$, the explorer 
moves to the second tree $T_{2}$, and so on. 

More concretely, let $\iota:\N_{0}\rightarrow V(\F)$ be the function which maps $k$ to the location of the depth-first search at step $k$ so that 
$\iota(0)=\r_{1}$, $\iota(k+1)$ is the leftmost unvisited child of $\iota(k)$ if such a child exists, and $\iota(k+1)$ is the parent of $\iota(k)$ if its children 
have all been visited. (Here the parent of $\r_{i}$ is taken to be $\r_{i+1}$.) The depth-first-search ordering of the vertices of $\F$ is given by 
$u\prec v$ if $\min\{k:\iota(k)=u\}<\min\{k:\iota(k)=v\}$. Finally, the \emph{contour process} on $\F$ is the function $H(\F):\N_{0}\rightarrow \N_{0}$ 
which maps $k$ to the level of $\iota(k)$ in $\F$. By construction, $H(\F)(k) = \Gamma_{k}$ for every $k\in\N_{0}$.

Now we discuss how to compute the Young diagram $\Lambda(\Gamma)$ from the corresponding rooted forest $\F(\Gamma)$. In the previous 
subsection, we constructed the diagram from the Motzkin path via successive applications of the hill-flattening and excursion operators. In terms 
of rooted forests, these operators can be interpreted in terms of `trimming' and `lopping.' Namely, let $\Upsilon_{0}$ be the collection of all rooted 
forests with finitely many vertices and consider the \emph{trimming operator} $\T:\Upsilon_{0}\rightarrow \Upsilon_{0}$ which deletes all leaves 
of the input forest; see Figure \ref{fig:tree diagram}.

Next, the \textit{lopping operator} $\L:\Upsilon_{0}\rightarrow \Upsilon_{0}$ is defined as follows: 
Given a rooted forest $\F=\{T_{i}\}\in \Upsilon_{0}$, find the rightmost node of maximal level, say $v_{\m}\in V(T_{k})$. Set $q=\iota^{-1}(v_{\m})$ 
and  let $\gamma$ be the unique path from $\mathtt{r}_{k}$ to $v_{\m}$. Now let $\F_{1}$ and $\F_{2}$ be the rooted forests induced from $\F$ such 
that$V(\F_{1})=\iota([1,q])$ and $V(\F_{2})=\iota([q,\infty))$. Then $\L(\F)$ is obtained by first deleting all edges contained in the copies of $\gamma$ 
from $\F_{1}$ and $\F_{2}$, and then taking the union of the resulting rooted forests with components ordered according to the depth-first search; see 
Figure \ref{fig:tree youngdiagram}.

\begin{figure*}[h]
\centering
\includegraphics[width=0.95 \linewidth]{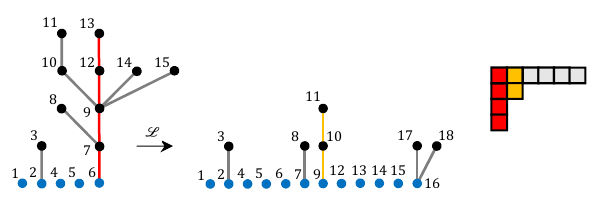}
	\caption{ The rooted forest $\F=\F(\Gamma(X_{0}))$ on the left appeared in Figure \ref{fig:tree diagram}, and the one on the right is 
	$\L(\F)$. Numbers next to nodes indicate depth-first-search ordering and $q=13$. Note that the maximum height of $\F$ and $\L(\F)$ 
	correspond to the first and second columns of $\Lambda(X_{0})$, respectively.}
	\label{fig:tree youngdiagram}
\end{figure*}

The following proposition shows that these operators are compatible with each other and gives a way to construct the Young diagram 
$\Lambda(\Gamma)$ from $\F(\Gamma)$.

\begin{prop}
\label{prop:Youngdiagram_forest}
For each Motzkin path $\Gamma$, we have the following:
	\begin{description}[leftmargin=!,labelwidth=\widthof{\bfseries [(iii)]}]
		\item[(i)] $ \F\big(\HH(\Gamma)\big) = \T\big(\F(\Gamma)\big)$.
		\item[(ii)] $\F\big(\E(\Gamma)\big) = \L\big(\F(\Gamma)\big)$.
		\item[(iii)] For each $1\leq i \leq \max \Gamma$, $\rho_{i} =\text{ $\#$ of leaves in $\T^{i-1}(\F(\Gamma))$ }$.
		\item[(iv)] For each $1\leq j \leq \rho(\Gamma)$, $\lambda_{j} = \text{ maximal level of nodes in $\L^{j-1}(\F(\Gamma))$ }$.
	\end{description}
\end{prop}

\begin{proof}
For (i), note that leaves in the forest correspond to hills in the path, so applying $\HH$ to $\Gamma$ results in the forest obtained by applying $\T$ 
to $\F(\Gamma)$.
For (ii), observe that $\E$ only affects the rightmost excursion of maximal height in $\Gamma$, $\L$ only affects the rightmost tree of maximal 
height in $\F(\Gamma)$, and the `bushes' growing off of the `trunk' of this tree correspond precisely to the subexcursions in the corresponding path 
component which are not subsumed by the maximum. 
	
Now assertion (i) shows that $\F(\HH^{i-1}(\Gamma)) = \T^{i-1}(\F(\Gamma))$ for all $1\leq i \leq \max \Gamma$, and $\rho_{i}$ is the number 
of hill intervals of $\HH^{i-1}(\Gamma)$, which equals the number of leaves in $\F(\HH^{i-1}(\Gamma))=\T^{i-1}(\F(\Gamma))$, and (iii) follows. 
Finally, given a rooted forest $\F$, denote by $\lVert \F \rVert$ the maximal level of nodes in $\F$. Then 
$\lVert \F(\Gamma) \rVert = \max \Gamma$, so (ii) implies
\[
\Big\lVert \L^{i-1}\big(\F(\Gamma)\big) \Big\rVert = \Big\lVert \F\big(\E^{i-1}(\Gamma)\big) \Big\rVert 
 = \max \E^{i-1}(\Gamma) = \lambda_{i}(\Gamma).\qedhere
\]
\end{proof}

We remark that Proposition \ref{prop:Youngdiagram_forest} (iv) holds if we replace the lopping operator $\mathcal{L}$ by the much simpler one 
which simply contracts the rightmost longest path into a single root. However, for this contraction operator Proposition 
\ref{prop:Youngdiagram_forest} (ii) no longer holds.

\section{Random box-ball system and Harris walk}
\label{section:random box-ball and rows}

In this section, we describe stochastic objects corresponding to the random box-ball system introduced in Subsection \ref{subsection:Statement}.

\subsection{Harris walks}
\label{subsection:Harris_walk}

Fix $p\in(0,1)$, and let $\xi_{1},\xi_{2},\ldots$ be i.i.d. with $\PP\{\xi_{1}=1\}=p$ and 
$\PP\{\xi_{1}=-1\}=1-p$. 
Let $X^{p},X^{n,p}\in\{0,1\}^{\N}$ be as in Subsection \ref{subsection:Statement}, and let $\{ S_{k} \}_{k=0}^{\infty}$ be the associated 
random walk, where $S_{0}=0$ and $S_{k}=\xi_{1}+\cdots+\xi_{k}$. The \textit{Harris walk} $\{H_{k}\}_{k=0}^{\infty}$ associated with 
$X^{p}$ is defined by $H_{0}=0$ and $H_{k}=\left(H_{k-1}+\xi_{k}\right)\vee 0$ for $k\geq 1$. In other words, $H_{k}$ is a simple random walk 
with increments $\xi_{j}$, except that downsteps at $0$ are censored. 

This defines an irreducible and aperiodic birth-and-death chain on $\N_{0}$ with transition probabilities $P(x,x+1)=p$, and 
$P(x,\left(x-1\right)\vee 0)=1-p$. One readily verifies that the chain is reversible with respect to the measure $\mu(x)=\theta^{-x}$ where  
$\theta=(1-p)/p$. Note that the sum $\sum_{k\geq 1}\theta^{k}$ converges if and only if $p>1/2$, so the chain is transient for these values 
of $p$ and recurrent for $p\leq 1/2$. It is null recurrent when $p=1/2$ since then $\sum_{k\geq 1}\theta^{-k}=\infty$, and it is positive recurrent 
for $p<1/2$ as the latter sum converges in this case. (See \cite{KarlinMcGregor} for background on recurrence criteria for birth-and-death chains.) 
In the ergodic regime, $p<1/2$, we can normalize $\mu$ to obtain the stationary distribution $\pi(x)=[(1-2p)/(1-p)]\theta^{-x}$.

Now the random Motzkin path $\Gamma(X^{n,p})$ is given by the trajectory of the Harris walk up to time $n$, 
completed by appending downstrokes at the end until the height reaches $0$ and appending $h$-strokes thereafter. More precisely, if we define 
$H:\R^{+}\rightarrow \R^{+}$ to be the linear interpolation of the Harris walk, 
$H(t)=H_{\lfloor t \rfloor}+(t- \lfloor t \rfloor )(H_{\lfloor t+1 \rfloor} - H_{\lfloor t \rfloor})$, then we have 
\[
\Gamma(X^{n,p})(x) = H(x) \one_{[0,n]}(x) + \max\big(0,H(n)-x+n\big) \one_{[n,\infty)}(x).
\]
Moreover, an easy induction argument shows that for all $k\in\N_{0}$, 
\begin{equation}
\label{eq:max_Sk}
H_{k} = S_{k} - \min_{0\leq r\leq k}S_{r}.
\end{equation}
Thus if $S:\R^{+}\rightarrow \R$ is the linear interpolation of the random walk $\{S_{k}\}_{k=0}^{\infty}$, then $H=\E_{0}(S)$. 
This observation also shows that, marginally, $H_{k}=_{d}\max_{0\leq r\leq k}S_{r}$.

\subsection{Galton-Watson forests}
\label{subsection:GW forests}

Following the procedure outlined in Subsection \ref{subsection:Trees}, one can construct a random rooted forest 
$\F(X^{n,p})=\F\big(\Gamma(X^{n,p})\big)$ from the trajectory of the truncated Harris walk $\Gamma(X^{n,p})$, and it turns out that 
$\F(X^{n,p})$ has the same law as the sub-forest of a Galton-Watson forest with mean offspring number $p/(1-p)$ consisting of the first 
$n$ nodes revealed by depth-first search.  

To be precise, let $\{\zeta^{k}_{j}\}_{j,k\geq 1}$ be an array of i.i.d. $\N_{0}$-valued random variables, and define the sequence 
$\{Z_{k}\}_{k\geq 0}$ by $Z_{0}=1$ and 
\[
Z_{k+1} = \begin{cases}
\zeta^{k+1}_{1}+\cdots+\zeta^{k+1}_{Z_{k}} & \text{if $Z_{k}>0$}\\
0 & \text{if $Z_{k}=0$}.
\end{cases}
\]
The interpretation is that $Z_{k}$ is the population size in the $k^{\text{th}}$ generation of a species in which individuals 
survive for a single generation and produce an i.i.d. number of offspring before dying. $\zeta^{k+1}_{j}$ is the number of 
offspring of the $j^{\text{th}}$ individual in generation $k$, and the common law of the $\zeta$'s is called the 
\emph{offspring distribution}. The family tree $T$ for this population is known as a \emph{Galton-Watson tree}. We will be 
interested in Galton-Watson trees with geometric offspring distribution 
\[
\PP\{\zeta^{k}_{j}=x\} = p^{x}(1-p), \quad x\in\N_{0},
\]
which is the number of independent $\text{Bern}(p)$ trials preceding the first failure. Observe that $\EE[\zeta^{k}_{j}]=p/(1-p)$, so 
$T$ is subcritical if $0<p<1/2$, critical if $p=1/2$, and supercritical if $1/2<p<1$. The law of a Galton-Watson tree with 
$\text{Geom}(1-p)$ offspring distribution will be denoted by $\mathtt{GWT}(p)$. 

We call a sequence of i.i.d. Galton Watson trees $\F_{GW}=\{T_{i}\}_{i\geq 1}$ a \emph{Galton-Watson forest}, 
and write $\mathtt{GWF}(p)$ for the law of a forest of i.i.d. $\mathtt{GWT}(p)$ trees. It is well known that for $0<p\leq 1/2$, 
each component $T_{i}$ is finite with full probability \cite[Ch.\ 5.3.4]{Durrett}, so the depth-first-search visits all nodes in the forest. 
However, for $p>1/2$, each component has a positive probability of being infinite, so almost surely there exists an index $I<\infty$ 
such that $|T_{i}|<\infty$ for all $i<I$ and $|T_{I}|=\infty$. Thus for $p>1/2$, the depth-first-search cannot pass beyond the 
leftmost infinite branch in $T_{I}$; see Figure \ref{pic:supercritical}. 

Now let $\F_{p}\sim \mathtt{GWF}(p)$, write $\F_{n,p}$ for the vertex-induced subforest of $\F_{p}$ on the nodes  
$\iota([1,n])\subseteq V(\F_{p})$ which are visited by the depth-first-search in the first $n$ steps, and write $\mathtt{GWF}(n,p)$ 
for the law of $\mathcal{F}_{n,p}$. 

\begin{prop}
\label{prop:GWF}
$\F(X^{n,p})\sim \mathtt{GWF}(n,p)$.
\end{prop}

\begin{proof}
Let $\Gamma=\Gamma(X^{p})$ and $\F=\F(\Gamma)$. Denote by $Z_{v}$ the number of children of node $v\in V(\F)$. We will show that the 
$Z_{v}$'s are i.i.d. and have the law of the number of independent $\text{Bern}(p)$ trials before the first failure. This will imply that the Harris walk 
$\{H_{k}\}_{k=0}^{\infty}$ is distributed as the contour process of $\F_{p}$. Then the relation between $\Gamma(X^{n,p})$ and $H$ from the 
previous subsection yields the assertion. 
	
Let $\F(X^{p})=\{T_{i}\}_{i\geq 1}$ and fix a node $v\in V(T_{i})$ for some $i\geq 1$. Let $P_{i}$ be the path component in $G(\Gamma)$ which 
is collapsed to $T_{i}$ via the equivalence relation $\sim$. Note that the number of nodes in $P_{i}\setminus\{v\}$ which are identified with $v$ 
equals the number of children of $v$. Let $x=(a_{0},\Gamma_{a_{0}})$ be such a vertex of $P_{i}$ with $a_{0}$ minimal. If 
$\Gamma_{a_{0}+1}-\Gamma_{a_{0}}=\xi_{a_{0}+1}$ is $1$, then the depth-first search finds the first child of $v$; otherwise, $v$ is childless and the 
search moves to its parent or to the root of next tree $T_{i+1}$ depending on whether $\Gamma_{a_{0}}\geq 1$ or $\Gamma_{a_{0}}=0$. If 
$\xi_{a_{0}+1}=1$, then let $a_{1}=\min\{k\geq a_{0} \st \Gamma_{k}=\Gamma_{a_{0}}\}$ 
be the first return time to level $\Gamma_{a_{0}}$ after $a_{0}$. ($a_{1}$ may be infinite if $p>1/2$.) As before, the depth-first search finds the second 
child of $v$ if and only if $\xi_{a_{1}+1}=1$. Continuing thusly, we see that $Z_{v}$ has a $\text{Geom}(1-p)$ distribution, and the proof is complete.   
\end{proof}

Proposition \ref{prop:GWF} allows us to describe the joint distribution of the first $i$ rows or the first $j$ columns in the random box-ball system 
started at $X^{n,p}$ in terms of Galton-Watson Forests.  

\begin{corollary}
\label{corollary:GWF}
Suppose that $\F\sim \mathtt{GWF}(n,p)$. For each $i\geq 1$, let $\mathfrak{l}_{i}$ and $\mathfrak{h}_{i}$ be the number of leaves in 
$\T^{i-1}(\F)$ and the maximum height of $\L^{i-1}(\F)$, respectively. Then for any $1\leq i \leq \max(\Gamma(X^{n,p}))$ and 
$1\leq j \leq \rho(\Gamma(X^{n,p}))$, we have 
\[
[\rho_{1}(n),\rho_{2}(n),\ldots,\rho_{i}(n)] =_{d} [\mathfrak{l}_{1},\mathfrak{l}_{2},\ldots,\mathfrak{l}_{i}]
\]
and
\[
[\lambda_{1}(n),\lambda_{2}(n),\ldots,\lambda_{j}(n)] =_{d}
[\mathfrak{h}_{1},\mathfrak{h}_{2},\ldots, \mathfrak{h}_{j}].
\]
\end{corollary}

\section{Asymptotics for the rows } 
\label{section:SLLN for Rows}

In this section, we prove our first main result, Theorem \ref{theorem:rows}. From the construction described in Subsection 
\ref{subsection:Youngdiagram_construction}, we have that $\rho_{1}(n)$, the length of the first row of 
$\Lambda^{n,p}$, equals the number of peaks in $\Gamma(X^{n,p})$, which equals the number of $1 \, 0$ 
patterns in $X^{n,p}$. In general, $\rho_{i}(n)$ is the number of subexcursions of height $i$ in the Harris walk 
$\{H_{k}\}_{k=0}^{n}$, and these can also be understood in terms of certain binary patterns in the initial configuration. 

We begin with a proof of the $i=1$ case of Theorem \ref{theorem:rows} using arguments from renewal theory. Strong laws for 
the other rows can be deduced similarly by considering analogous (delayed) renewal processes, but we will find it more 
convenient to pursue an alternative approach that will be of use in Section \ref{section:Permutation}.

\begin{proof}[\textbf{Proof of Theorem \ref{theorem:rows} for }$\boldsymbol{i=1}$]
First observe that the number of solitons in $X^{n,p}$ is equal to the number of $1 \, 0$ patterns, so 
$\rho_{1}(n)=\one\{\xi_{n}=1\}+N_{10}(n)$ where $N_{10}(n)$ is the number of $1 \, 0$ patterns in 
the first $n$ terms. Because of the scaling, it suffices to prove that 
$N_{10}(n)=\sum_{k=1}^{n-1}\one\big\{\xi_{k}=1,\xi_{k+1}=-1\big\}$ satisfies the asserted limit theorems. 

Now $N_{10}(n)$ counts occurrences of `head, tail' patterns in a sequence of independent coin flips, which we view 
as a renewal process. Let $T_{10}$ be distributed as the inter-event times in this process. Then the elementary renewal 
theorem gives $\EE[N_{10}(n)]/n\rightarrow 1/\EE[T_{10}]$. Since $\EE[N_{10}(n)]=(n-1)p(1-p)$, it follows from the 
strong law for renewal processes that 
\[
\frac{N_{10}(n)}{n}\rightarrow \frac{1}{\EE[T_{10}]}=p(1-p) \text{ a.s.}
\]

Renewal theory also shows that $N_{10}(n)$ converges weakly to a standard normal random variable when appropriately 
normalized \cite{cox1962renewal}. To compute the variance, we write $W_{k}=\one\big\{\xi_{k}=1,\xi_{k+1}=-1\big\}$ and 
observe that $\EE[W_{k}]=p(1-p)$, $\EE[W_{k}W_{k+1}]=0$, and $\EE[W_{k}W_{\ell}]=p^{2}(1-p)^{2}$ when $|k-\ell|>1$, hence
\[
\EE[N_{10}(n)^{2}]=\sum_{k=1}^{n-1}\EE[W_{k}^{2}]+\sum_{|k-\ell|>1}\EE[W_{k}W_{\ell}]
=(n-1)p(1-p)+(n-2)(n-3)p^{2}(1-p)^{2},
\]
so
\[
\Var\big(N_{10}(n)\big)=\EE[N_{10}(n)^{2}]-\EE[N_{10}(n)]^{2}=(n-1)p(1-p)-(3n-5)p^{2}(1-p)^{2}.
\]
The second part of the theorem follows upon invoking Slutsky's theorem to simplify the expression 
$(N_{10}(n)-\EE[N_{10}(n)])/\Var(N_{10}(n))^{1/2}$.
\end{proof}

\begin{remark}
\textup{
The normal convergence of $\rho_{1}(n)$ can also be established using Stein's method for sums of locally 
dependent random variables (see \cite[Ch. 9]{CGS}). Though this approach is more involved, it has the upshot of supplying 
a Berry-Esseen rate of order $O(n^{-1/2})$. One can show that a central limit theorem also holds for the other row lengths 
by a similar renewal theory argument, but the corresponding variance computations are not as straightforward. 
}
\end{remark}

To treat the $i>1$ case, we need to establish some more terminology and a useful lemma. Let $\gamma:\R^{+}\rightarrow \R$ 
be any nearest neighbor lattice path (so $|\gamma_{k+1}-\gamma_{k}|\in \{-1,0,1\}$ for all $k\in \N_{0}$). We say that $\gamma$ 
has a \textit{subexcursion of height $h$} on the interval $[r,t]$ if $\gamma_{r}=\gamma_{t}<\gamma_{s}$ for all $s\in(r,t)$ and 
$\max_{r<s<t}\gamma_{s}-\gamma_{r}=h$. Such a subexcursion is said to begin at $r$ and end at $t$. 

Let $\{ S_{k}  \}_{k=0}^{\infty}$ be the simple random walk with increment distribution $\PP\{S_{k+1}-S_{k}=1\}=1-\PP\{S_{k+1}-S_{k}=-1\}=p$.
For each $i\geq 1$ and $\ell\geq 0$, define the indicator random variable
\[
J_{\ell}^{i}=\one\{\text{$S$ has subexcursion of height $i$ beginning at time $\ell$}\}
\]
and let $\tau_{\ell}^{i}$ be the length of the subexcursion of $S_{k}$ beginning at $k=\ell$, conditional on $J_{\ell}^{i}=1$. Note that the distribution 
of $\tau_{\ell}^{i}$ does not depend on $\ell$ by the Markov property of $S_{k}$, so we may drop the subscript when notationally convenient. 
Moreover, due to the negative drift of $S_{k}$ for $0<p<1/2$, it is not hard to see that $\tau^{i}$ has an exponential tail. 

The following lemma establishes a polynomial tail bound for the sum of centered indicators and thereby a strong law for the number of subexcursions of 
fixed height in the interval $[0,n]$. This bound (with $m=4$ and $\eps= 1/\log n$) will also be used in in the proof of Theorem \ref{theorem:permutations} 
in Section \ref{section:Permutation}. 

\begin{lemma}
\label{lemma:chernoff}
Let $\varsigma=\inf\{ k> 0 \st S_{k}=0 \}$ be the first return time of $S_{k}$ to zero. Fix $i\geq 1$ and $\eps>0$. Set 
$\mu_{i}=\PP\{ \text{$\max_{0\leq k \leq \varsigma}S_{k}=i$}\}$. Then for each fixed $m\geq 2$, there exists a constant $C=C(m,i,p)>0$ 
such that for each $n\geq 1$, 
\[
\PP\Bigg\{ \Bigg| \sum_{\ell=0}^{n} \big(J_{\ell}^{i} - \mu_{i}\big) \Bigg|>\eps n\Bigg\} \leq \frac{C}{\eps^{2m} n^{m-1}}. 
\]
\end{lemma}

With the above lemma (proved at the end of this section), it is easy to deduce Theorem \ref{theorem:rows}. 

\begin{proof}[\textbf{Proof of Theorem \ref{theorem:rows} for }$\boldsymbol{i\geq 1}$]
The hill-flattening procedure produces a unique column of length at least $i$ for each 
such subexcursion, so $\rho_{i}(n)$ is the number of height $i$ subexcursions of $H$ on $[0,n]$. Since the Harris 
walk $H_{k}$ and the associated simple random walk $S_{k}$ over $[0,n]$ share subexcursions of positive height, 
we may regard $\rho_{i}(n)$ as the number of subexcursions of $S_{k}$ occuring on $[0,n]$. Furthermore, we can 
approximate $\rho_{i}(n)$ by $N_{i}(n):=\sum_{\ell=0}^{n} J_{\ell}^{i}$ since the two only 
differ when $H$ has a subexcursion of height at least $i$ beginning at or after $n-i$, hence $|N_{i}(n)-\rho_{i}(n)|\leq 1$. 
Therefore, the assertion follows from Lemma \ref{lemma:chernoff} with $m=3$, $\eps=1/\log n$ and the first Borel-Cantelli lemma.
\end{proof}

Our proof of Lemma \ref{lemma:chernoff} is based on joint moment estimates of the random variables $J_{\ell}^{i}-\mu_{i}$. Before undertaking this 
task, we give some preliminary calculations and remarks to set the stage. Fix integers $i\geq 0$ and $0\leq \ell_{1}<\ell_{2}<\ell_{3}$. Clearly 
$\EE[J_{\ell_{1}}^{i}-\mu_{i}]=0$, and we compute
\begin{align}
\EE\big[ (J_{\ell_{1}}^{i}-\mu_{i})(J_{\ell_{2}}^{i}-\mu_{i}) \big] 
 & = \mu_{i} \EE\big[ J_{\ell_{2}}^{i}-\mu_{i} \,|\, J_{\ell_{1}}^{i}=1\big] - \mu_{i} \EE\big[ J_{\ell_{2}}^{i}-\mu_{i}\big] \\
 & = \mu_{i} \EE\big[ J_{\ell_{2}}^{i}-\mu_{i} \big] \PP\big\{ \tau_{\ell_{1}}^{i} \leq \ell_{2}-\ell_{1}\big\}
  - \mu_{i}^{2} \PP\big\{ \tau_{\ell_{1}}^{i}>\ell_{2}-\ell_{1}\big\}\\
 & = -\mu_{i}^{2} \,\PP\big\{\tau^{i}>\ell_{2}-\ell_{1}\big\},
\end{align}
where we used the fact that $J_{\ell_{2}}^{i}$ is independent of $\big\{J_{\ell_{1}}^{i}=1\big\}$ if the excursion starting at $\ell_{1}$ ends at or before 
$\ell_{2}$, and $J_{\ell_{2}}^{i}=0$ otherwise. Arguing analogously, we find that
\begin{eqnarray*}
	&&\EE\big[ (J_{\ell_{1}}^{i}-\mu_{i}) (J_{\ell_{2}}^{i}-\mu_{i}) (J_{\ell_{3}}^{i}-\mu_{i}) \big] \\
	&&\qquad\qquad = \mu_{i} \EE\big[ (J_{\ell_{2}}^{i}-\mu_{i})(J_{\ell_{3}}^{i}-\mu_{i})  \,|\, J_{\ell_{1}}^{i}=1\big]
	 -  \mu_{i} \EE\big[ (J_{\ell_{2}}^{i}-\mu_{i})(J_{\ell_{3}}^{i}-\mu_{i}) \big] \\
	&&\qquad\qquad  = \mu_{i}\EE\big[ (J_{\ell_{2}}^{i}-\mu_{i})(J_{\ell_{3}}^{i}-\mu_{i})\big] \PP\big\{\tau^{i}_{\ell_{1}}\le \ell_{2}-\ell_{1} \big\}
 	- \mu_{i}^{2}\EE\big[ J_{\ell_{3}}^{i}-\mu_{i}\big] \PP\big\{\tau^{i}_{\ell_{1}}\in (\ell_{2}-\ell_{1},\ell_{3}-\ell_{1}]\big\} \\
	&& \qquad \qquad\qquad   + \mu_{i}^{3} \PP\big\{\tau^{i}_{\ell_{1}}>\ell_{3}-\ell_{1}\big\}
	 -  \mu_{i} \EE\big[ (J_{\ell_{2}}^{i}-\mu_{i})(J_{\ell_{3}}^{i}-\mu_{i}) \big] \\
	&&\qquad\qquad  = -\mu_{i}\EE\big[ (J_{\ell_{2}}^{i}-\mu_{i})(J_{\ell_{3}}^{i}-\mu_{i})\big] \PP\big\{\tau^{i}_{\ell_{1}} >\ell_{2}-\ell_{1}\big\}
	 + \mu_{i}^{3} \PP\big\{\tau^{i}>\ell_{3}-\ell_{1}\big\} \\
	&& \qquad\qquad  = \mu_{i}^{3} \PP\big\{ \tau^{i}>\ell_{3}-\ell_{2}\big\} \PP\big\{\tau^{i}>\ell_{2}-\ell_{1}\big\}
	+ \mu_{i}^{3} \PP\big\{\tau^{i}>\ell_{3}-\ell_{1}\big\}.
\end{eqnarray*}
This shows that $J_{1}^{i},J_{2}^{i},\ldots,J_{n}^{i}$ are \emph{not} negatively associated for $n\geq 3$, so an immediate Chernoff-Hoeffding type 
bound is not applicable in our case. 

Now in order to prove Lemma \ref{lemma:chernoff}, we need to estimate the joint central moments of the random variables $J_{\ell}^{i}$. For 
the sake of readability, this is split up into two propositions. Here and henceforth, the empty product is understood to equal one. 

\begin{prop}
\label{prop:chernoff_induction}
Fix integers $r\geq 2$, $0\leq \ell_{1} < \ell_{2} < \cdots < \ell_{r}$, and $\alpha_{1},\ldots,\alpha_{r}>0$. Then 
\begin{align*}
\Bigg|\EE\Big[\prod_{k=1}^{r}(J^{i}_{\ell_{k}}-\mu_{i})^{\alpha_{k}}\Big]\Bigg|
 & \leq \sum_{s=2}^{r}\Bigg|\EE\Big[\prod_{k=s+1}^{r}(J^{i}_{\ell_{k}}-\mu_{i})^{\alpha_{k}}\Big]\Bigg|\PP\big\{ \tau^{i}>\ell_{s}-\ell_{1}\big\}\\
 & \quad+\left(\left|\EE\big[(J^{i}_{\ell_{1}}-\mu_{i})^{\alpha_{1}}\big]\right|+\PP\big\{ \tau^{i}>\ell_{2}-\ell_{1}\big\} \right)
 \Bigg|\EE\Big[\prod_{k=2}^{r}(J^{i}_{\ell_{k}}-\mu_{i})^{\alpha_{k}}\Big]\Bigg|.
\end{align*}
\end{prop}

\begin{proof}
Write $\beta_{s}=\prod_{k=2}^{s-1}(-\mu_{i})^{\alpha_{k}}$. 
Casing out according to whether $J^{i}_{\ell_{1}}$ is $0$ or $1$, we see that $(J^{i}_{\ell_{1}}-\mu_{i})^{\alpha_{1}}=c_{1}J^{i}_{\ell_{1}}+d_{1}$ 
where $d_{1}=(-\mu_{i})^{\alpha_{1}}$ and $c_{1}=(1-\mu_{i})^{\alpha_{1}}-(-\mu_{i})^{\alpha_{1}}$. Since $\mu_{i}\in[0,1]$, a little calculus 
shows that $\left|c_{1}\right|,\left|d_{1}\right|\leq1$. Now the strong Markov property for $S_{k}$ implies that for any $s\geq2$, $J^{i}_{\ell_{s}}$ is 
independent of $J^{i}_{\ell_{1}}$ if the excursion starting at $\ell_{1}$ ends at a site less than or equal to $\ell_{s}$; otherwise $J^{i}_{\ell_{s}}=0$. 
By partitioning according to the length $\tau_{\ell_{1}}^{i}$ of the first excursion we compute 
\begin{align*}
\EE&\Big[\prod_{k=1}^{r}(J^{i}_{\ell_{k}} -\mu_{i})^{\alpha_{k}}\Big]
 = d_{1}(1-\mu_{i})\EE\Big[\prod_{k=2}^{r}(J^{i}_{\ell_{k}} -\mu_{i})^{\alpha_{k}}\Big|J^{i}_{\ell_{1}}=0\Big] 
+ (c_{1}+d_{1})\mu_{i}\EE\Big[\prod_{k=2}^{r}(J^{i}_{\ell_{k}} -\mu_{i})^{\alpha_{k}}\Big|J^{i}_{\ell_{1}}=1\Big]\\
 & =c_{1}\mu_{i}\EE\Big[\prod_{k=2}^{r}(J^{i}_{\ell_{k}}-\mu_{i})^{\alpha_{k}}\Big|J^{i}_{\ell_{1}}=1\Big]
 +d_{1}\EE\Big[\prod_{k=2}^{r}(J^{i}_{\ell_{k}}-\mu_{i})^{\alpha_{k}}\Big]\\
 & =c_{1}\mu_{i}\sum_{s=2}^{r}
\EE\Big[\prod_{k=2}^{r}(J^{i}_{\ell_{k}}-\mu_{i})^{\alpha_{k}}\Big|J^{i}_{\ell_{1}}=1,\tau_{\ell_{1}}^{i}\in(\ell_{s-1}-\ell_{1},\ell_{s}-\ell_{1}]\Big]
 \PP\big\{ \tau_{\ell_{1}}^{i}\in(\ell_{s-1}-\ell_{1},\ell_{s}-\ell_{1}]\big\} \\
 & \qquad+c_{1}\mu_{i}\EE\Big[\prod_{k=2}^{r}(J^{i}_{\ell_{k}}-\mu_{i})^{\alpha_{k}}\Big|J^{i}_{\ell_{1}}=1,\tau_{\ell_{1}}^{i}>\ell_{r}-\ell_{1}\Big]
 \PP\big\{ \tau_{\ell_{1}}^{i}>\ell_{r}-\ell_{1}\big\} + d_{1}\EE\Big[\prod_{k=2}^{r}(J^{i}_{\ell_{k}}-\mu_{i})^{\alpha_{k}}\Big]\\
 & =c_{1}\mu_{i}\sum_{s=2}^{r}\beta_{s}\EE\Big[\prod_{k=s}^{r}(J^{i}_{\ell_{k}}-\mu_{i})^{\alpha_{k}}\Big]
 \PP\big\{ \tau^{i}\in(\ell_{s-1}-\ell_{1},\ell_{s}-\ell_{1}]\big\}\\
 & \qquad  +c_{1}\mu_{i}\beta_{r+1}\PP\big\{ \tau^{i}>\ell_{r}-\ell_{1}\big\} +d_{1}
 \EE\Big[\prod_{k=2}^{r}(J^{i}_{\ell_{k}}-\mu_{i})^{\alpha_{k}}\Big]\\
 & =c_{1}\mu_{i}\sum_{s=3}^{r}\beta_{s}\EE\Big[\prod_{k=s}^{r}(J^{i}_{\ell_{k}}
    -\mu_{i})^{\alpha_{k}}\Big]\PP\big\{ \tau^{i}\in(\ell_{s-1}-\ell_{1},\ell_{s}-\ell_{1}]\big\}
    +\big(c_{1}\mu_{i}+d_{1}\big)\EE\Big[\prod_{k=2}^{r}(J^{i}_{\ell_{k}}-\mu_{i})^{\alpha_{k}}\Big]\\
 & \qquad -c_{1}\mu_{i}\EE\Big[\prod_{k=2}^{r}(J^{i}_{\ell_{k}}-\mu_{i})^{\alpha_{k}}\Big]\PP\big\{\tau^{i}>\ell_{2}-\ell_{1}\big\}  
    +c_{1}\mu_{i}\beta_{r+1}\PP\big\{ \tau^{i}>\ell_{r}-\ell_{1}\big\}.
\end{align*}

\pagebreak

Since $|c_{1}|,|\mu_{i}|,|\beta_{s}|\leq1$, 
$\PP\big\{ \tau^{i}\in(\ell_{s-1}-\ell_{1},\ell_{s}-\ell_{1}]\big\} \leq\PP\big\{ \tau^{i}>\ell_{s-1}-\ell_{1}\big\}$, and 
$c_{1}\mu_{i}+d_{1}=\EE\big[(J^{i}_{\ell_{1}}-\mu_{i})^{\alpha_{1}}\big]$, the triangle inequality yields 
\begin{align*}
\Bigg|\EE\Big[\prod_{k=1}^{r}(J^{i}_{\ell_{k}}-\mu_{i})^{\alpha_{k}}\Big]\Bigg|
 & \leq\sum_{s=3}^{r}\bigg|\EE\Big[\prod_{k=s}^{r}(J^{i}_{\ell_{k}}-\mu_{i})^{\alpha_{k}}\Big]\bigg|\PP\big\{ \tau^{i}>\ell_{s-1}-\ell_{1}\big\}
    +\PP\big\{ \tau^{i}>\ell_{r}-\ell_{1}\big\} \\
 & \quad+\Big|\EE\big[(J^{i}_{\ell_{1}}-\mu_{i})^{\alpha_{1}}\big]\Big|\Bigg|\EE\Big[\prod_{k=2}^{r}(J^{i}_{\ell_{k}}-\mu_{i})^{\alpha_{k}}\Big]\Bigg| 
 + \Big|\EE\Big[\prod_{k=2}^{r}(J^{i}_{\ell_{k}}-\mu_{i})^{\alpha_{k}}\Big]\Big|\PP\big\{ \tau^{i}>\ell_{2}-\ell_{1}\big\} \\
 & =\sum_{s=2}^{r}\Bigg|\EE\Big[\prod_{k=s+1}^{r}(J^{i}_{\ell_{k}}-\mu_{i})^{\alpha_{k}}\Big]\Bigg|\PP\big\{ \tau^{i}>\ell_{s}-\ell_{1}\big\} \\
 & \quad+\left(\left|\EE\big[(J^{i}_{\ell_{1}}-\mu_{i})^{\alpha_{1}}\big]\right|+\PP\big\{ \tau^{i}>\ell_{2}-\ell_{1}\big\} \right)
 \Bigg|\EE\Big[\prod_{k=2}^{r}(J^{i}_{\ell_{k}}-\mu_{i})^{\alpha_{k}}\Big]\Bigg|. \qedhere
\end{align*}
\end{proof}

The key intuition for the next step is that each linear factor $(J^{i}_{\ell_{k}}-\mu_{k})$ 
effectively decreases the `degrees of freedom' by at least a half. This idea is codified in the following proposition.     

\begin{prop}
\label{prop:chernoff}
Fix integers $r\geq 2$, $0\leq \ell_{1} < \ell_{2} < \cdots < \ell_{r}$, and $\alpha_{1},\ldots,\alpha_{r}\geq 1$. 
If $I:=\{ 1\leq k <r \st \alpha_{k}=1 \}$ is nonempty, then there exist constants $C_{r},D>0$ such that 
\begin{equation*}
\Bigg|\EE\Big[ \prod_{k=1}^{r} (J_{\ell_{k}}^{i}- \mu_{i})^{\alpha_{k}} \Big] \Bigg| 
\leq C_{r} \sum_{\substack{I_{0}\subseteq [1,r): \\ 2|I_{0}|\geq |I|}} \exp\left( -D\sum_{j\in I_{0}} (\ell_{j+1}-\ell_{j}) \right).  
\end{equation*}
\end{prop}

\begin{proof}
Since the length $\tau^{i}$ of a subexcursion of height $i$ in $S_{k}$ has exponential tail, we may choose constants $D,D_{0}>0$ such that 
\begin{equation}
\label{eq:tau_tail}
\PP(\tau^{i}>t) \leq D_{0}\exp (-Dt)
\end{equation}
for all $t\geq 0$. 

Also, the exponential is nonnegative, so it's enough to establish the inequality when the outer sum on the right-hand side is taken over a subset of 
those $I_{0}\subseteq[1,r)$ with cardinality at least half that of $I$. Thus, for instance, we may dispense with the $I=\{1\}$ case by showing that 
the expectation on the left is bounded by a constant multiple of $\exp\left( -D(\ell_{2}-\ell_{1}) \right)$. This is an immediate consequence of Proposition 
\ref{prop:chernoff_induction} and Equation \eqref{eq:tau_tail} since $\EE[J^{i}_{\ell_{1}}-\mu_{i}]=0$, 
$\left|\EE\big[\prod_{k=s}^{r}(J_{\ell_{k}}-\mu_{i})^{\alpha_{k}}\big]\right|\leq 1$, and
$\PP\big\{ \tau^{i}>\ell_{s}-\ell_{1}\big\}\leq \PP\big\{ \tau^{i}>\ell_{2}-\ell_{1}\big\}$ for $s\geq 2$. 

We now proceed by induction on $r$. The base case follows from the previous observation as the assumption that $I\neq\emptyset$ implies $I=\{1\}$ 
when $r=2$. For the induction step, let $r\geq 3$. Denote by $B_{1}$ and $B_{2}$ the first and second term in the right-hand side of the displayed 
inequality in Proposition \ref{prop:chernoff_induction}, and let $K$ denote the sum over $I_{0}$ in the right-hand side of the displayed inequality in 
Proposition \ref{prop:chernoff}. 
By Proposition \ref{prop:chernoff_induction}, it suffices to show that both $B_{1}$ and $B_{2}$ can be bounded by some constant times $K$. 

For the bound on $B_{2}$, note that the induction hypothesis gives
\begin{equation}
\Bigg|\EE\Big[ \prod_{k=2}^{r} (J_{\ell_{k}}^{i}- \mu_{i})^{\alpha_{k}} \Big] \Bigg| 
\leq C_{r-1} \sum_{\substack{I_{0}'\subseteq [2,r): \\ 2|I_{0}'|\geq |I\cap [2,r)|}} \exp\left( -D\sum_{j\in I_{0}'} (\ell_{j+1}-\ell_{j}) \right).
\end{equation}
If $\alpha_{1}\geq 2$, then $I\subseteq [2,r)$, so we have $B_{2}\leq 2C_{r-1} K$. Otherwise $\alpha_{1}=1$ and we are assuming $I\neq \{1\}$, 
so the induction hypothesis and Equation \eqref{eq:tau_tail} imply
\[
\Bigg|\EE\Big[ \prod_{k=2}^{r} (J_{\ell_{k}}^{i}- \mu_{i})^{\alpha_{k}} \Big] \Bigg| \PP\{ \tau^{i}>\ell_{2}-\ell_{1} \} 
\leq  C_{r-1} \sum_{\substack{I_{0}'\subseteq  [2,r): \\ 2|I_{0}'|\geq |I\cap [2,r)|}} \exp\left( -D \left( \ell_{2}-\ell_{1}
+ \sum_{j\in I_{0}'} (\ell_{j+1}-\ell_{j})\right) \right).
\]
If we set $I_{0}:=\{ 1\}\cup I_{0}'$ for each $I_{0}'\subseteq  [2,r)$ in the above summation, then $2|I_{0}|\geq 2+ |I\cap [2,r)| = 2+ |I|-1>|I|$. 
Moreover, the exponential terms can be written as $\exp(-D( \sum_{j\in I_{0}} \ell_{j+1}-\ell_{j}))$. Accordingly, we have that $B_{2}\leq C_{r-1}K$.

Next, we show that $B_{1}$ can be bounded by some constant times $K$. Writing $m_{1}=\max(I)$, we see that $|I\cap [s+1,r)|\geq 1$ for $s<m_{1}$, 
so it follows from the inductive hypothesis, Equation \eqref{eq:tau_tail}, and the fact that all central moments are bounded in absolute value by one that
\begin{eqnarray*}
B_{1} &\leq& \sum_{s=2}^{m_{1}-1} C_{r-s} D_{0} \sum_{\substack{I_{0}'\subseteq  [s+1,r): \\ 2|I_{0}'|\geq |I\cap [s+1,r)|}} 
\exp\left( -D\left(  \ell_{s}-\ell_{1} +  \sum_{j\in I_{0}'} (\ell_{j+1}-\ell_{j}) \right)\right) \\
 &&\qquad \quad + \sum_{s=m_{1}}^{r} D_{0} \exp\left( -D\left(  \ell_{s}-\ell_{1} \right)\right),
\end{eqnarray*}
with the convention that the empty sum is zero. For the first term, we view its inner sum as ranging over all $I_{0}\subseteq [0,r)$ with 
$I_{0}:= [1,s)\cup I_{0}'$. 
Note that $2|I_{0}| = 2(s-1) + 2|I_{0}'|\geq 2(s-1) + |I\cap [s+1,r)| \geq |I|$ since $s\geq 2$.  Furthermore, the sum of the $\ell_{j+1}-\ell_{j}$ 
terms over $j\in [1,s)$ is exactly $\ell_{s}-\ell_{1}$. Thus the first term above is at most some constant times $K$. 
Finally, taking $m_{2}=\min(m_{1},r-|I|+1)\geq 2$, we see that the second term is bounded by 
$D_{0}\sum_{s=m_{2}}^{r}\exp\left(-D(\ell_{s}-\ell_{s-1})\right)$, which is a single summand in $K$ and so less than $K$. 
This completes the inductive step and the proof.
\end{proof}

We are now ready to prove Lemma \ref{lemma:chernoff}.

\begin{proof}[\textbf{Proof of Lemma \ref{lemma:chernoff}}] 
Fix $m\in\N$, and use Chebyshev's inequality and the linearity of expectation to write 
\[
\PP\Bigg\{ \bigg| \sum_{\ell=0}^{n} (J_{\ell}^{i}-\mu_{i}) \bigg| \geq  t   \Bigg\} \leq t^{-2m} 
\sum_{0\leq \ell_{1}\leq \cdots \leq \ell_{2m}\leq n}  \EE\left[ (J_{\ell_{1}}^{i}-\mu_{i})\cdots (J_{\ell_{2m}}^{i}-\mu_{i}) \right]
\] 
Our goal is to show that the right-hand side of the above inequality is $O(t^{-2m}n^{m+1})$. Then letting $t=\eps n$ gives the assertion. 
(The Landau notation is in terms of $n\rightarrow\infty$ throughout this proof.) We first observe that it suffices to bound the contribution 
from expectations involving at least $m+1$ distinct $\ell_{k}$'s as there are $O(n^{m})$ summands involving fewer and each is $O(1)$. 

Fix $m< r \leq 2m$ and let $\alpha_{1},\ldots,\alpha_{r}$ be positive integers such that $\sum_{k=1}^{r} \alpha_{k}=2m$. 
Write $r=u+w$ where $w = \sum_{k=1}^{r} \one\{\alpha_{k}=1\}$, and let $I=\{1\leq k<r \st \alpha_{k}=1\}$ as in the preceding proposition. 
Since there are $O(1)$ choices for the $r$'s and $\alpha_{k}$'s, we need only to demonstrate the existence of a constant $C_{1}=C_{1}(r,i,p)>0$ 
such that 
\begin{equation}
\label{eq:pf_lemma_3.4}
\Bigg| \sum_{0\leq \ell_{1}< \cdots < \ell_{r}\leq n} \EE\left[ (J_{\ell_{1}}^{i}-\mu_{i})^{\alpha_{1}}\cdots (J_{\ell_{r}}^{i}-\mu_{i})^{\alpha_{r}}\right] \Bigg| 
\leq C_{1} n^{m+1}
\end{equation}
for all $n\geq 1$. 

Note that $w\geq 2$ so that $|I|\geq 1$ and Proposition \ref{prop:chernoff} applies. Thus we will be done upon showing that for 
each subset $I_{0}\subseteq [1,r)$ such that $2|I_{0}|\geq |I|$, there exists a constant $C_{2}=C_{2}(i,p)>0$ such that 
\begin{equation}
\label{eq:subclaim}
\Bigg| \sum_{0\leq \ell_{1}< \cdots < \ell_{r}\leq n}\exp\bigg(-D \sum_{j\in I_{0}} (\ell_{j+1}-\ell_{j}) \bigg) \Bigg| \leq  C_{2} n^{m+1}
\end{equation}
for all $n\geq 1$. (There are $O(1)$ subsets $I_{0}$ in the sum from Proposition \ref{prop:chernoff}.)

To verify Equation \eqref{eq:subclaim}, first observe that if $\ell_{j+1}-\ell_{j}> n^{1/2m}$ for some $j\in I_{0}$, then the corresponding summand is of 
order $O(\exp(-Dn^{1/2m}))$. As there are $O(n^{2m})$ choices, the contribution from such terms is of order $O(1)$. Accordingly, it suffices to show that 
there are $O(n^{m+1})$ sequences $0\leq \ell_{1}<\cdots<\ell_{r}\leq n$ not verifying this condition. To this end, let $L$ be the set of maps 
$\ell:[r]\rightarrow [n]\cup\{0\}$ such that $\ell(j+1)-\ell(j)\leq n^{1/2m}$ for all $j\in I_{0}$, and let $G=([r],E)$ be the graph with vertex set $[r]$ and 
edge set $E=\big\{ \{j,j+1\} \st j\in I_{0} \big\}$. 
Then $G$ contains at most $r-|E|=r-|I_{0}|$ connected components, say $P_{1},\ldots,P_{N}$ where $P_{i}$ is a path consisting of vertices 
$\{j_{i},j_{i}+1,\ldots,j_{i}+s_{i}-1\}$, $s_{i}=|P_{i}|$. Now for any $\ell\in L$ and $1\leq i \leq N$, there at most $n^{1+s_{i}/2m}$ possible choices 
for $\ell(P_{i})$---$n$ for $\ell(j_{i})$ and $n^{1/2m}$ for each of the $s_{i}-1$ successive vertices. Since $N\leq r-|I_{0}|$ and 
$\sum_{i=1}^{N}s_{i}=r\leq 2m$, this gives 
\[
|L|\leq \prod_{i=1}^{N} n^{1+s_{i}/2m} =  n^{N+r/2m} \leq n^{r-|I_{0}|+1}.
\]
The assertion then follows since 
\[
2r-2|I_{0}| + 2 \leq 2r - w +2 \leq 2u+w+2 \leq 2m+2,
\]
where we have used the fact that $2|I_{0}|\geq |I|$ and $2u+w\leq \sum_{k=1}^{r}\alpha_{i} = 2m$. 
\end{proof}

\section{Top soliton lengths in the subcritical regime}
\label{section:Subcritical}

In this section, we prove Theorem \ref{theorem:columns} (i). Fix $p\in (0,1/2)$ and let $\left\{H_{k}\right\}_{k=0}^{\infty}$ denote the Harris walk associated 
with the random box-ball configuration $X^{p}$. The main insight is that the $j^{\text{th}}$ longest soliton length, $\lambda_{j}(n)$, is asymptotically 
equal to the  $j^{\text{th}}$ largest excursion height of $H_{k}$ over the interval $[0,n]$, which we denote by $h_{j}(n)$ (Lemma 
\ref{lemma:soliton_lengths_excursion_heights}). 
This allows us to obtain limit theorems for the $\lambda_{j}(n)$ in terms of the $h_{j}(n)$ (Lemma \ref{lemma:limit_excursion_heights}). 

Before getting into the details, we discuss the main issue in comparing soliton lengths with excursion heights. Clearly $\lambda_{j}(n)\geq h_{j}(n)$ due 
to the hill-flattening construction of the invariant Young diagram (Lemma \ref{lemma:invariance_Youngdiagram}). For $j=1$, we also have 
$\lambda_{1}(n)=h_{1}(n)$ since $\lambda_{1}(n)$ equals the maximum height of the Harris walk over $[0,n]$ by Lemma 
\ref{lemma:columnbyexcursionop}. However, this identity does not hold for $j\geq 2$. Indeed, Lemma \ref{lemma:columnbyexcursionop} shows that 
$\lambda_{2}(n)=\max \E(\Gamma(X^{n,p}))$, the maximum excursion height of the modified Motzkin path $\E(\Gamma(X^{n,p}))$. While all but the 
highest excursion of $\Gamma(X^{n,p})$ are preserved after applying the excursion operator $\E$, it might be the case that there is a large subexcursion within 
the highest excursion which dominates the contribution from the second highest excursion of $\Gamma(X^{n,p})$. In Subsection 
\ref{subsection:subsubexc}, we show that this is not the case asymptotically.

\subsection{Overview and main results}
\label{subsection:sub_main_results}

We begin by stating the main results of this section and using them to prove Theorem \ref{theorem:columns} (i). Our first step is to obtain limit theorems 
for the $h_{j}(n)$ (which will be defined more carefully in the following subsection).

\begin{lemma} 
\label{lemma:limit_excursion_heights}
Set $\theta=(1-p)/p$, $\sigma=(1-2p)/(1-p)$, and $\mu_{n}=\log_{\theta}\left((1-2p)\sigma n\right)$. 
Let $h_{j}(n)$ be the $j^{\text{th}}$ largest excursion height of the associate Harris walk over $[0,n]$. Then for any nondecreasing real sequence 
$\{x_{n}\}_{n\geq 1}$, 
\[
\liminf_{n\rightarrow \infty} \,\, \exp\big(\theta^{-x_{n}}\big) \left(\sum_{k=0}^{j-1} \frac{\theta^{-k(x_{n}+1)}}{k!}\right)^{-1} 
\PP\left\{ h_{j}(n) \leq x_{n}+\mu_{n} \right\} \geq 1,
\]
and
\[
\limsup_{n\rightarrow \infty} \,\, \exp\big(\theta^{-(x_{n}+1)}\big)\left( \sum_{k=0}^{j-1}\frac{\theta^{-kx_{n}}}{k!}\right)^{-1} 
\PP\left\{ h_{j}(n) \leq x_{n}+\mu_{n} \right\} \leq 1.
\]
\end{lemma}

Next, we show that the soliton lengths and excursion heights are essentially the same objects. 

\begin{lemma}
\label{lemma:soliton_lengths_excursion_heights}
Fix $p\in (0,1/2)$. Then for each $j\geq 1$,
\[
\lim_{n\rightarrow \infty} \PP\left\{ \lambda_{j}(n)\neq h_{j}(n) \right\} =0.
\]
\end{lemma}

It is then straightforward to derive the main result for soliton lengths in the subcritical regime. 

\begin{proof}[\textbf{Proof of Theorem} 
\ref{theorem:columns} \textup{(i)}]
Fix $j\geq 1$, $x\in\R$, and let $\mu_{n}=\log_{\theta}\left((1-2p)\sigma n\right)$. Since $\lambda_{j}(n)\geq h_{j}(n)$, we have 
\[
\PP\left\{ \lambda_{j}(n) \leq x +\mu_{n} \right\} \leq \PP\left\{ h_{j}(n) \leq x +\mu_{n} \right\}.
\]	
Hence Lemma \ref{lemma:limit_excursion_heights} shows 
\[
\limsup_{n\rightarrow \infty}\PP\left\{ \lambda_{j}(n) \leq x+\mu_{n} \right\} \leq \exp(-\theta^{-(x+1)})\sum_{k=0}^{j-1} \frac{\theta^{-kx}}{k!}.
\]
For the other inequality, we have  
\[
\PP\left\{ h_{j}(n) \leq x +\mu_{n} \right\} \le \PP\left\{ \lambda_{j}(n) \leq x +\mu_{n} \right\} + \PP\left\{\lambda_{j}(n)\ne h_{j}(n) \right\}, 
\]	
so Lemmas \ref{lemma:limit_excursion_heights} and \ref{lemma:soliton_lengths_excursion_heights} show that
\[
\liminf_{n\rightarrow \infty}\PP\left\{ \lambda_{j}(n) \leq x+\mu_{n} \right\} \ge \exp(-\theta^{-x})\sum_{k=0}^{j-1} \frac{\theta^{-k(x+1)}}{k!}.\qedhere
\]
\end{proof}

\subsection{Excursion heights}
\label{subsection:limit_order_statistics}

This subsection is devoted to proving Lemma \ref{lemma:limit_excursion_heights}. Roughly speaking, we proceed by showing that the Harris walk has 
$\Theta(n)$ excursions by time $n$. By relating the excursion heights to a gambler's ruin problem, we argue that their distribution has an exponential tail. 
Taking the maximum over the $\Theta(n)$ excursions shows that the law of $h_{1}(n)$ is approximated by a Gumbel distribution after scaling appropriately. 
The other order statistics are handled similarly.

To begin, set $\tau_{1}=0$ and for $k>1$, define $\tau_{k}=\inf\{j>\tau_{k-1} \st H_{j}=0\}$ to be the time of the $k^{\text{th}}$ 
visit to $0$. Thus $\tau_{k}$ is the beginning of the $k^{\text{th}}$ excursion above the $x$-axis, and $\tau_{k+1}$ is the end of 
the $k^{\text{th}}$ such excursion. (In this section, if the random walk stays at $0$, this counts as an excursion of height $0$.) Let 
\[
h_{k} = \sup \{H_{t} \st \tau_{k}< t \leq \tau_{k+1}\} 
= \sup \bigg\{\sum_{i=\tau_{k}+1}^{t}\xi_{i} \st \tau_{k}< t \leq \tau_{k+1} \bigg\} \vee 0
\]
be the maximum height of the $k^{\text{th}}$ excursion. The strong Markov property ensures that $h_{1},h_{2},\ldots$ are 
i.i.d. $\N_{0}$-valued random variables. To compute their distribution function, $F(x)=\PP\{h_{1}\leq x\}$, we observe that 
$\PP\{h_{1}=0\}=1-p$ and $\PP\{h_{1}\leq x\}=\PP\{1\leq h_{1}\leq x\}+\PP\{h_{1}=0\}$ for $x\geq 1$. 
In order for the event $\left\{1 \leq h_{1} \leq x\right\}$ to occur, the random walk must begin with an upstep and then 
visit zero before visiting $x+1$. The latter occurs with the `gambler's ruin' probability that a simple random walker, started 
at the origin and moving right with probability $p$, hits $-1$ before hitting $x$, which is given by 
$\big(\theta^{x}-1\big)/\big(\theta^{x}-\theta^{-1}\big)$ \cite[Ch.\ 5.7]{Durrett}. Putting all of this together shows that 
$F(x)=(1-p)+p\big(\theta^{x}-1\big)/\big(\theta^{x}-\theta^{-1}\big)$ for all $x\in\N_{0}$. 
After a bit of rearranging, we get
\begin{equation}
\label{eq:CDF_height_sub_excursion}
F(x)= \left( 1-\frac{1-2p}{\theta^{\left\lfloor x\right\rfloor +1}-1}\right)\one_{[0,\infty)}(x).
\end{equation}

Now let $h_{1:m},\ldots,h_{m:m}$ denote the (reversed) order statistics of $h_{1},\ldots,h_{m}$ so that 
$h_{1:m}\geq  \cdots \geq  h_{m:m}$ and $\{h_{1:m},\ldots,h_{m:m}\}=\{h_{1},\ldots,h_{m}\}$ as multisets. Then 
\begin{equation}
\label{eq:CDF_order_statistics}
F_{j:m}(x) := \PP\{h_{j:m}\leq x\} = \sum_{k=0}^{j-1} \binom{m}{k} F(x)^{m-k}(1-F(x))^{k}, \quad j=1,\ldots,m.
\end{equation}
In particular, the maximum $h_{1:m}$ has distribution function 
\begin{equation}
\label{eq:CDF_max_height}
F_{1:m}(x) = \left( 1-  \frac{1-2p}{\theta^{\lfloor x \rfloor+1}-1}\right)^{m} \one_{[0,\infty)}(x).
\end{equation}

Write $M_{n} = \sup\{k \st \tau_{k+1}\leq n\}$ for the number of excursions completed by time $n$ and let 
$r_{n}=\max\{\sum_{i=\tau_{M_{n}+1}}^{r}\xi_{i} \st \tau_{M_{n}+1}\leq r \leq n\}$ be the maximum height attained 
after the last complete excursion. The \emph{excursion heights} $h_{1}(n)\geq h_{2}(n)\geq\cdots\geq h_{M_{n}+1}(n)$ 
are the (reversed) order statistics for $h_{1},\ldots,h_{M_{n}},r_{n}$. We begin by showing that $M_{n}$ is sharply 
concentrated around its mean so that we can essentially treat it as a deterministic sequence. 

\begin{prop}
\label{prop:Mn_slln}
If $M_{n}$ is the number of excursions of $H$ completed by time $n$, then
\[
\frac{M_{n}}{n} \rightarrow \frac{1-2p}{1-p} \text{ a.s.} \quad \text{as $n\rightarrow \infty$}.
\]
\end{prop}

\begin{proof}
We may write $M_{n}=\sum_{k=1}^{n}\one\{H_{k}=0\}$, the number of visits to $0$ in $[1,n]$. Since the Harris walk 
is ergodic with stationary distribution $\pi(x)=[(1-2p)/(1-p)]\theta^{-x}$ for $p<1/2$, we can apply the Markov chain 
ergodic theorem to obtain 
\[
\frac{M_{n}}{n}\rightarrow \pi(0)=\frac{1-2p}{1-p} \text{ a.s.} \qedhere
\]
\end{proof}

The next ingredient in our argument is a simple stochastic monotonicity result.
 
\begin{prop}
\label{prop:concentration}
Set $\sigma=(1-2p)/(1-p)$, $p\in (0,1/2)$. For any real sequence $\{x_{n}\}_{n\geq 1}$ and any positive integer $j$, 
we have that for all $\eps>0$,
\[
\limsup_{n\rightarrow \infty} \PP\left\{h_{j}(n)\leq x_{n}\right\} 
\leq \limsup_{n\rightarrow \infty}  \PP\left\{h_{j:\lfloor \left(\sigma  -\eps\right)n\rfloor} \leq x_{n}\right\}
\]
and 
\[
\liminf_{n\rightarrow \infty} \PP\{h_{j}(n)\leq x_{n}\} 
\geq \liminf_{n\rightarrow \infty} \PP\left\{h_{ j:\lceil (\sigma  +\eps) n \rceil} \leq x_{n}\right\}.
\]
\end{prop}

\begin{proof}
Define 
\[
N^{-}(n,\eps) = \sup\left\{t \st M_{t}\leq (\sigma-\eps)n \right\}
\]
and 
\[
N^{+}(n,\eps) = \inf \left\{t \st M_{t}\geq (\sigma+\eps)n \right\}.
\]	
It follows from Proposition \ref{prop:Mn_slln} that there is an a.s. finite $N$ such that 
\[
\{h_{1},\ldots,h_{M_{N^{-}(n,\eps)}}\}\subseteq\{h_{1},\ldots,h_{M_{n}+1}\}\subseteq\{h_{1},\ldots,h_{M_{N^{+}(n,\eps)}}\}
\]
with probability one for all $n\geq N$. Because $r_{n}\leq h_{M_{n}+1}$ and the probability that $h_{M_{n}+1}$ is among the $j$ largest of 
$h_{1},\ldots,h_{M_{n}+1}$ goes to zero as $n\rightarrow\infty$, we see that for any $\eps>0$, 
\[
\PP\big\{h_{j:M_{N^{-}(n,\eps)}}\leq h_{j}(n)\leq h_{j:M_{N^{+}(n,\eps)}}\big\}>1-\eps
\] 
when $n$ is sufficiently large, hence
\[
\limsup_{n\rightarrow \infty} \PP\{h_{j}(n)\leq x_{n}\} 
\leq \limsup_{n\rightarrow \infty} \PP\big\{h_{j:M_{N^{-}(n,\eps)}}\leq x_{n}\big\}
\]
and 
\[
\liminf_{n\rightarrow \infty} \PP\{h_{j}(n)\leq x_{n}\} 
\geq \liminf_{n\rightarrow \infty} \PP\big\{h_{j:M_{N^{+}(n,\eps)}}\leq x_{n}\big\}.
\]
The desired assertion follows by noting that $M_{N^{-}(n,\eps)}=\lfloor (\sigma-\eps)n \rfloor$ and $M_{N^{+}(n,\eps)}=\lceil (\sigma+\eps)n \rceil$ 
a.s. since $0$ is a recurrent state of $\{H_{k}\}$.
\end{proof}

We are now in a position to prove the main result of this subsection.

\begin{proof}[\textbf{Proof of Lemma} 
\ref{lemma:limit_excursion_heights}] 

First, we claim that for any sequence $\{b_{n}\}_{n\geq 1}$ with $\lim_{n\rightarrow\infty}b_{n}/n=c>0$ and any nondecreasing 
sequence $\{y_{n}\}_{n\geq 1}$, we have 
\begin{equation}
\label{cdf limit}
\lim_{n\rightarrow\infty} \exp\left((c/\sigma)\theta^{-y_{n}}\right) 
\left(1- \frac{\theta^{-y_{n}}}{\sigma n-(1-2p)^{-1}\theta^{-y_{n}}}\right)^{b_{n}} =1. 
\end{equation}
Indeed,
\begin{multline*}
\log \left[ \left(1- \frac{\theta^{-y_{n}}}{\sigma n-(1-2p)^{-1}\theta^{-y_{n}}}\right)^{b_{n}}
\exp\left((c/\sigma)\theta^{-y_{n}}\right) \right]\\
=\theta^{-y_{n}}(b_{n}/\theta^{-y_{n}})\log  \left(1- \frac{\theta^{-y_{n}}}{\sigma n-(1-2p)^{-1}\theta^{-y_{n}}}\right)
+(c/\sigma)\theta^{-y_{n}}\\
= \theta^{-y_{n}} \left[ \frac{c}{\sigma}
+\frac{\log \left(1-\frac{\theta^{-y_{n}}}{\sigma n - (1-2p)^{-1}\theta^{-y_{n}}} \right)}{\theta^{-y_{n}}/b_{n}} \right].
\end{multline*}
Since $\theta>1$ and $\{y_{n}\}_{n\geq 1}$ is nondecreasing, $\theta^{-y_{n}}$ is bounded. The claim follows since a Taylor 
expansion of the log term shows that
\[
\lim_{n\rightarrow\infty}\frac{\log \left(1-\frac{\theta^{-y_{n}}}{\sigma n - (1-2p)^{-1}\theta^{-y_{n}}} \right)}{\theta^{-y_{n}}/b_{n}}
= -\frac{c}{\sigma}.
\]

Now fix $\eps>0$ and a nondecreasing sequence $\{x_{n}\}_{n\geq 1}$. Recall that for any deterministic sequence of integers $\{b_{n}\}_{n\geq1}$, 
\begin{align}
\PP\left\{ h_{j:b_{n}}\leq x_{n}+\mu_{n}\right\} 
& = \sum_{k=0}^{j-1} \binom{b_{n}}{k}\left(1-\frac{1-2p}{\theta^{\lfloor x_{n}+\mu_{n} \rfloor+1}-1}\right)^{b_{n}-k}
\left(\frac{1-2p}{\theta^{\lfloor x_{n}+\mu_{n} \rfloor+1}-1}\right)^{k} \\
& = \left(1-\frac{1-2p}{\theta^{\lfloor x_{n}+\mu_{n} \rfloor+1}-1}\right)^{b_{n}} 
\, \sum_{k=0}^{j-1} b_{n}^{-k} \binom{b_{n}}{k} \left(1-\frac{1-2p}{\theta^{\lfloor x_{n}+\mu_{n} \rfloor+1}-1}\right)^{-k}
\left(\frac{(1-2p)b_{n}}{\theta^{\lfloor x_{n}+\mu_{n} \rfloor+1}-1}\right)^{k}
\end{align}
when $x_{n}+\mu_{n}\geq 0$. (Since $\{x_{n}\}_{n\geq 1}$ is nondecreasing and $\mu_{n}=\log_{\theta}\left((1-2p)\sigma n\right)\nearrow\infty$, 
this restriction is satisfied for all large $n$.)

Writing $\nu_{n}=(x_{n}+\mu_{n})-\lfloor x_{n} + \mu_{n} \rfloor$, we have 
\[
\frac{1-2p}{\theta^{\lfloor x_{n}+\mu_{n}\rfloor+1}-1}=\frac{1-2p}{\theta^{x_{n}+\mu_{n}}\theta^{1-\nu_{n}}-1}
=\frac{\theta^{-x_{n}}}{\sigma n\theta^{1-\nu_{n}}-(1-2p)^{-1}\theta^{-x_{n}}},
\]
so, since $\theta>1$ and $1-\nu_{n}\in (0,1]$, we see that 
\begin{equation}
\label{subcritical eq}
\frac{\theta^{-(x_{n}+1)}}{\sigma n-(1-2p)^{-1}\theta^{-(x_{n}+1)}} \leq \frac{1-2p}{\theta^{\lfloor x_{n}+\mu_{n}\rfloor+1}-1}
\leq \frac{\theta^{-x_{n}}}{\sigma n-(1-2p)^{-1}\theta^{-x_{n}}}.
\end{equation}
	
Set $b_{n}=\lfloor (\sigma-\eps)n \rfloor$ and note that $\lim_{n\rightarrow \infty} b_{n}^{-k} \binom{b_{n}}{k} = \frac{1}{k!}$.
Then the above estimates and show that for all sufficiently large $n$,
\begin{align}
\PP\left\{ h_{j:b_{n}}\leq x_{n}+\mu_{n}\right\} 
 &\leq  \left(1- \frac{\theta^{-(x_{n}+1)}}{\sigma n-(1-2p)^{-1}\theta^{-(x_{n}+1)}}\right)^{b_{n}}\,\sum_{k=0}^{j-1} \frac{1+\eps}{k!}
 \left( \frac{(\sigma+\eps)n\theta^{-x_{n}}}{\sigma n-(1-2p)^{-1}\theta^{-x_{n}}} \right)^{k}.
\end{align} 
Thus Equation \eqref{cdf limit} with $y_{n}=x_{n}+1$ and $b_{n}=\lfloor (\sigma-\eps)n \rfloor$ gives
\begin{align}
\limsup_{n\rightarrow \infty} \, \exp\left(\left(1-\frac{\eps}{\sigma}\right)\theta^{-(x_{n}+1)}\right) \left(\sum_{k=0}^{j-1} \frac{\theta^{-kx_{n}}}{k!} \right)^{-1}
\PP\left\{h_{j:b_{n}}\leq x_{n}+\mu_{n}\right\} \leq (1+\eps)\left(1+\frac{\eps}{\sigma} \right)^{k}.
\end{align}  
By taking $y_{n}=x_{n}$ and $b^{\prime}_{n}=\lceil (\sigma + \eps )n\rceil$, a similar argument shows that 
\begin{align}
\liminf_{n\rightarrow \infty} \, \exp\left(\left(1+\frac{\eps}{\sigma}\right)\theta^{-x_{n}}\right) \left(\sum_{k=0}^{j-1} \frac{\theta^{-k(x+1)}}{k!} \right)^{-1}
\PP\left\{h_{1}(n)\leq x_{n}+\mu_{n}\right\} \ge (1-\eps)\left(1-\frac{\eps}{\sigma} \right)^{k}.
\end{align}  
Letting $\eps\searrow 0$ and applying Proposition~\ref{prop:Mn_slln} completes the proof.
\end{proof}

\begin{remark}
\textup{Because we are taking the maximum of a random number of excursions, the sequence $\{h_{j}(n)-\mu_{n}\}_{n\geq 1}$ does not have a weak 
limit (and thus neither do the normalized subcritical soliton lengths). To see this, we first recall that 
\[
\PP\{h_{1:b_{n}}\leq x + \mu_{n}\}=\left(1-\frac{1-2p}{\theta^{\lfloor x+\mu_{n}\rfloor +1}-1}\right)\one\{x+\mu_{n}\geq 0\}
\]
for any real sequence $\{b_{n}\}_{n\geq 1}$ and any $x\in\R$. Now fix $x>\mu_{1}$, write $\nu_{n}=(x+\mu_{n})-\lfloor x+\mu_{n}\rfloor$, and 
choose subsequences $\{\nu_{n_{k}}\}_{k\geq 1}$ and $\{\nu_{n_{\ell}}\}_{\ell\geq 1}$ such that $\nu_{n_{k}}\leq\frac{1}{3}$ and 
$\nu_{n_{\ell}}\geq\frac{2}{3}$ for all $k,\ell\in\N$. This is possible since $\mu_{n}=\log_{\theta}\big((1-2p)\sigma\big)+\log_{\theta}(n)$ and the 
fractional part of $\log_{\theta}(n)$ is dense in $[0,1]$.}

\textup{Since
\[
\frac{1-2p}{\theta^{\lfloor x+\mu_{n}\rfloor +1}-1}=\frac{\theta^{-x}}{\sigma n\theta^{1-\nu_{n}}-(1-2p)^{-1}\theta^{-x}},
\]
$\theta>1$, $1-\nu_{\ell}\in [0,1/3]$, and $1-\nu_{k}\in [2/3,1]$, we have the following analogues of Equation \eqref{subcritical eq}:
\[
\frac{\theta^{-(x+1/3)}}{\sigma n_{\ell} - (1-2p)^{-1}\theta^{-(x+1/3)}}
\leq\frac{1-2p}{\theta^{\lfloor x+\mu_{n_{\ell}}\rfloor+1}-1}\leq\frac{\theta^{-x}}{\sigma n_{\ell}-(1-2p)^{-1}\theta^{-x}}
\]
and
\[
\frac{\theta^{-(x+1)}}{\sigma n_{k} - (1-2p)^{-1}\theta^{-(x+1)}}
 \leq\frac{1-2p}{\theta^{\lfloor x+\mu_{n_{k}}\rfloor+1}-1} \leq\frac{\theta^{-(x+2/3)}}{\sigma n_{k} - (1-2p)^{-1}\theta^{-(x+2/3)}}.
\]
}

\textup{Repeating the last part of the proof of Lemma \ref{lemma:limit_excursion_heights} (and restricting attention to $h_{1}(n)=\lambda_{1}(n)$ 
to simplify notation) shows that
\[
e^{x}\leq\liminf_{\ell\rightarrow\infty}\,\PP\big\{\lambda_{1}(n_{\ell})\leq x+\mu_{n_{\ell}}\big\}
\leq\limsup_{\ell\rightarrow\infty}\,\PP\big\{\lambda_{1}(n_{\ell})\leq x+\mu_{n_{\ell}}\big\}\leq e^{x+1/3}
\]
and
\[
e^{x+2/3}\leq\liminf_{k\rightarrow\infty}\,\PP\big\{\lambda_{1}(n_{k})\leq x+\mu_{n_{k}}\big\}
\leq\limsup_{k\rightarrow\infty}\,\PP\big\{\lambda_{1}(n_{k})\leq x+\mu_{n_{k}}\big\}\leq e^{x+1}.
\]
In particular, 
\begin{equation}
\label{subcritical subseq ineq}
\limsup_{\ell\rightarrow\infty}\,\PP\big\{\lambda_{1}(n_{\ell})\leq x+\mu_{n_{\ell}}\big\}
<\liminf_{k\rightarrow\infty}\,\PP\big\{\lambda_{1}(n_{k})\leq x+\mu_{n_{k}}\big\}.
\end{equation}
Since the sequence $\big\{\lambda_{1}(n)-\mu_{n}\big\}$ is tight by Lemma \ref{lemma:limit_excursion_heights}, both 
$\big\{\lambda_{1}(n_{k})-\mu_{n_{k}}\big\}$ and $\big\{\lambda_{1}(n_{\ell})-\mu_{n_{\ell}}\big\}$ have subsequential weak limits. As Inequality 
\eqref{subcritical subseq ineq} implies that the limiting distribution functions disagree at $x$, it follows that $\big\{\lambda_{1}(n)-\mu_{n}\big\}$ does 
not converge weakly.}
\end{remark}

\subsection{Subexcursions within an excursion}
\label{subsection:subsubexc}

Given an excursion $\gamma$ of $H$ with length $\varsigma$ and rightmost global maximum at $(m^{\ast},h)$, define 
$a_{\ell}=\max\{t\leq m^{\ast}\st \gamma(t)=\ell\}$ and $b_{\ell}=\min\{t\geq m^{\ast}\st \gamma(t)=\ell\}$ for $\ell=0,\ldots,h$. 
Write $\gamma_{a,\ell}=\gamma|_{[a_{\ell-1},a_{\ell}]}-\ell$ and $\gamma_{b,\ell}=\gamma|_{[b_{\ell},b_{\ell-1}]}-\ell$. 
These paths correspond to the portions of $\gamma$ which, moving away from $m^{\ast}$, begin at the point where $\gamma$ first 
descends to height $\ell$ and end where $\gamma$ first descends to height $\ell-1$, except that they are shifted down by $\ell$; 
see Figure \ref{pic:excursion intervals}. 

Set $\widetilde{\gamma}_{a,\ell}=\gamma_{a,\ell}\vee0$, $\widetilde{\gamma}_{b,\ell}=\gamma_{b,\ell}\vee0$ (which has the effect of changing 
the downstroke furthest from $m^{\ast}$ to an $h$-stroke) and define 
\[
\vartheta(k)=\begin{cases}
\max\widetilde{\gamma}_{a,k}, & k\leq h\\
\max\widetilde{\gamma}_{b,k-h}, & h<k\leq2h
\end{cases}.
\]
Then $\E(\gamma)$ is the concatenation of 
$\widetilde{\gamma}_{a,1},\ldots,\widetilde{\gamma}_{a,h},\widetilde{\gamma}_{b,h},\ldots,\widetilde{\gamma}_{b,1}$, so 
$\max\E(\gamma)=\max_{1\leq k\leq2h}\vartheta(k)$.
	
\begin{figure*}[h]
	\centering
	\includegraphics[width=1 \linewidth]{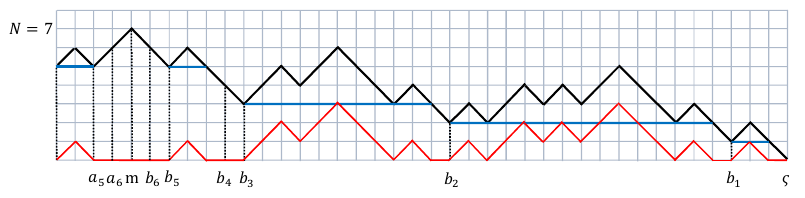}
	\vspace{-0.5cm}
	\caption{The black path is $\gamma$ and the red path is $\E(\gamma)$. $a_{\ell}$ and $b_{\ell}$ are the first locations to the left and right of $\texttt{m}$ 
		    that $\gamma$ is at height $\ell$ and the blue lines indicate paths between successive values. In these regions, $\gamma$ and $\E(\gamma)$ 
		    differ only by vertical translation.
		     }
	\label{pic:excursion intervals}
\end{figure*}
	
Now let $m_{\ast}$ denote the leftmost global maximum of $\gamma$ and set $c_{\ell}=\min\{t\geq m_{\ast}\st \gamma(t)=\ell\}$,  
$\gamma_{c,\ell}=\gamma|_{[c_{\ell},c_{\ell-1}]}-\ell$, $\widetilde{\gamma}_{c,\ell}=\gamma_{c,\ell}\vee0$. Define 
$\omega(k)=\max\widetilde{\gamma}_{c,k}=\max\gamma_{c,k}$ for $k=1,\ldots,h$. We first observe that 
$\max\{\omega(1),\ldots,\omega(h)\}\geq\max\{\vartheta(h+1),\ldots,\vartheta(2h)\}$. 
To see that this is so, let $j=\min_{m_{\ast}\leq t\leq m^{\ast}}\gamma(t)$. Then $b_{k}=c_{k}$ for all $k<j$, hence
$\omega(k)=\max_{c_{k}\leq t\leq c_{k-1}}\gamma(t)-k\geq \max_{b_{k}\leq t\leq b_{k-1}}\gamma(t)-k =\vartheta(h+k)$ for all $k\leq j$ 
because $c_{\ell}\leq b_{\ell}$ for all $\ell$. On the other hand, since $c_{j}\leq m^{\ast}<c_{j-1}$, $\omega(j)=h-j> h-k\geq\vartheta(h+k)$ for all 
$k>j$. It follows that $\max\E(\gamma)=\max\{\vartheta(1),\ldots,\vartheta(2h)\}\leq \max\{\vartheta(1),\ldots,\vartheta(h),\omega(1),\ldots,\omega(h)\}$.

Next we observe that $\gamma$ is symmetric about $\varsigma/2$ in distribution. This is because, conditional on the excursion length, the law of $\gamma$ 
depends only on the number of up and down steps. Accordingly, $\E\gamma|_{[0,m^{\ast}]}$ and $\E\gamma|_{[m_{\ast},\varsigma]}$ have the same 
distribution, so $\max\{\vartheta(1),\ldots,\vartheta(h)\}=_{d}\max\{\omega(1),\ldots,\omega(h)\}$, and thus 
\begin{align*}
\PP\big\{ \max\E(\gamma)>x\big\} & \leq \PP\big\{\max\{\vartheta(1),\ldots,\vartheta(h),\omega(1),\ldots,\omega(h)\} > x \big\}\\
 & \leq \PP\big\{\max\{\vartheta(1),\ldots,\vartheta(h)\} > x \big\}+\PP\big\{\max\{\omega(1),\ldots,\omega(h)\} > x \big\}\\
 & = 2\,\PP\big\{\max\{\omega(1),\ldots,\omega(h)\} > x \big\}.
\end{align*}

To treat the latter probability, note that given $\{h=r\}$, $m_{\ast},c_{r-1},\ldots,c_{0}$ are stopping times with respect to the natural filtration, so 
$\omega(1),\ldots,\omega(r)$ are independent by the strong Markov property. 
Also, each $\omega(k)$ is stochastically dominated by the random variable $Y$ which gives the maximum value taken by a simple random walker 
started at $0$ and moving right with probability $p$ before hitting $-1$ (as the path $\gamma_{c,k}$ is constrained to be at height at most $h-k$ 
whereas the random walker's path has no such restriction). We conclude that on the event $\{h\leq r\}$,
\begin{equation}
\label{subexcursion bound}
\PP\big\{ \max\E(\gamma)>x\big\}  \leq 2\,\PP\big\{\max\{\omega(1),\ldots,\omega(h)\} > x \big\} \leq 2\left(1-G(x)^{r}\right)
\end{equation}
where $G(x)=\PP\{ Y\leq x\} $ is the gambler's ruin probability \cite[Ch.\ 5.7]{Durrett}
\[
G(x)=\frac{\theta^{\left\lfloor x\right\rfloor +1}-1}{\theta^{\left\lfloor x\right\rfloor +1}-\theta^{-1}}\mathbf{1}_{[0,\infty)}(x).
\]

\begin{proof}[\textbf{Proof of Lemma} 
\ref{lemma:soliton_lengths_excursion_heights}]	

Fix $j\geq1$ and let $\eps>0$ be given. Lemma \ref{lemma:limit_excursion_heights} implies that there exist $\delta>0$, $N\in\N$ 
such that for each $n\geq N$, the event 
\[
E_{j,\delta,n}=\big\{ h_{1}(n)\leq n^{\delta},\,h_{j}(n)\geq 2\delta\log_{\theta}(n)\big\} 
\]
has probability at least $1-\eps$. Write $H^{(k,n)}$ for the $k^{\text{th}}$ highest excursion of $H|_{[0,n]}$, so that $h_{k}(n)=\max H^{(k,n)}$. 
As each application of the excursion operator affects only one excursion, $\lambda_{1}(n),\ldots,\lambda_{j}(n)$ are the $j$ largest values 
among $\E^{i-1}\left(H^{(k,n)}\right)$ as $i$ and $k$ range over $\{1,\ldots,j\}$. On $E_{j,\delta,n}$, these coincide 
with $h_{1}(n),\ldots,h_{j}(n)$ when $\E\big(H^{(k,n)}\big)\leq 2\delta\log_{\theta}(n)$ for $k=1,\ldots,j$. Since 
\[
G(x)=\frac{\theta^{\left\lfloor x\right\rfloor +1}-1}{\theta^{\left\lfloor x\right\rfloor +1}-\theta^{-1}}\geq1-\frac{\theta}{\theta^{x}-1}
\]
for $x>0$, Equation \eqref{subexcursion bound} implies 
\begin{align*}
\PP\big\{ \E\big(H^{(k,n)}\big)>2\delta\log_{\theta}(n);\,E_{j,\delta,n}\big\}  & \leq 2\left[1-G\big(2\delta\log_{\theta}(n)\big)^{n^{\delta}}\right]
 \leq 2\Bigg[1-\Bigg(1-\frac{\theta}{n^{2\delta}-1}\Bigg)^{n^{\delta}}\Bigg]\\
 & \leq 2\Bigg[1-\Bigg(1-n^{\delta}\frac{\theta}{n^{2\delta}-1}\Bigg)\Bigg] \leq \frac{4\theta}{n^{\delta}}.
\end{align*}
Consequently,
\[
\PP\big\{ \lambda_{k}(n)=h_{k}(n)\text{ for }k=1,\ldots,j\big\} \leq \eps+\frac{4j\theta}{n^{\delta}},
\]
and the claim follows since $\eps$ is arbitrary.
\end{proof}

\section{Top soliton lengths at criticality}
\label{section:Critical}

In this section we observe that when $p=1/2$, the (suitably scaled) Harris walk converges weakly to a reflected Brownian motion at the 
process level. In fact, this weak convergence can be strengthened to ``polynomial convergence'' by appealing to a result from 
Drmota \cite{drmota2004stochastic}. This enables us to deduce scaling limits for the top soliton lengths. 

Recall that $C([0,1])$ denotes the space of continuous functions $f:[0,1]\rightarrow \R$ equipped with the supremum norm. We say 
a continuous functional $F:C([0,1])\rightarrow \R$ is of \textit{polynomial growth} if there exists $r\geq 1$ such that 
$|F(\gamma)| \leq \lVert \gamma \rVert_{\infty}^{r}$ for all $\gamma\in C([0,1])$.

\begin{theorem}[Theorem 9 of \cite{drmota2004stochastic}]
\label{thm:drmota}
Suppose that a sequence of stochastic processes $x_{n}(t)$ defined on $C([0,1])$ converges weakly to $x(t)$. Furthermore suppose that there 
exists $s_{0}\in [0,1]$ such that for all $r\geq 0$, 
\[
\sup_{n\geq 0} \EE\big[ |x_{n}(s_{0})|^{r} \big] < \infty, 
\]
and that for every $\alpha>1$, there exists $\beta>0$ and $C>0$ with 
\[
\EE\big[|x_{n}(t)-x_{n}(s)|^{\beta}\big] < C|t-s|^{\alpha} \enspace \text{for all } s,t\in [0,1].
\]
If $F:C([0,1])\rightarrow \R$ is any continuous functional of polynomial growth, then 
\[
\lim_{n\rightarrow\infty} \EE\left[ F(x_{n}) \right] =\EE\left[ F(x) \right]. 
\]
\end{theorem}

\pagebreak

We show the following polynomial convergence of Harris walk to the reflected Brownian motion.

\begin{theorem} 
\label{brownian_thm_process}
Let $\{B(t) \st 0\leq t\leq 1\}$ be a standard Brownian motion and define $H^{n}(t)=H(nt)/\sqrt{n}$ for $0\leq t \leq 1$. Then for $p=1/2$,
\[
\big\{H^{n}(t) \st 0\leq t \leq 1\big\} \Rightarrow \big\{|B(t)| \st 0\leq t \leq 1\big\} \, \text{ in }\, C([0,1]).
\]  
Furthermore, if $F:C([0,1])\rightarrow \R$ is any continuous functional of polynomial growth, then 
\[
\lim_{n\rightarrow\infty} \EE[F(H^{n})] = \EE(F(|B|)). 
\]
\end{theorem}

\begin{proof}
Since the rescaled Harris walk $H^{n}(t)$ is uniformly bounded by $n^{-\frac{1}{2}}|S^{n}(t)|$, which has moments of all orders and satisfies the 
H\"{o}lder criterion in Theorem \ref{thm:drmota}, we only need to show that $H^{n}$ converges weakly to $|B|$. To this end, recall from Subsection 
\ref{subsection:Statement} that the linear interpolation of the $p=1/2$ Harris walk is given by $H(t)=\E_{0}(S)(t)=S(t)-\min_{0\leq r\leq t}S(r)$, 
where $S$ is the linear interpolation of symmetric simple random walk.

Donsker's Theorem shows that after scaling diffusively, $S(t)$ converges weakly to a standard Brownian motion in the space 
$C([0,1])$. That is, writing $S^{n}(t)=S(nt)/\sqrt{n}$, we have 
\[
\lim_{n\rightarrow\infty}\EE[F(S^{n})]=\EE[F(B)]
\]
for every bounded and continuous functional $F:C([0,1])\rightarrow\R$.

A direct computation shows that for any fixed $b\in [0,1]$, $\E_{b}$ is ($2$-Lipschitz) continuous and satisfies $\E_{b}(cf)=c\E_{b}(f)$ 
for all $b,c\geq 0$ (see Proposition \ref{prop:Lipschitz} (i) in Subsection \ref{subsection:regularity of column length functionals}), so 
for every bounded and continuous $G:C([0,1])\rightarrow\R$,
\[
\lim_{n\rightarrow\infty}\EE[G(H^{n})]=\lim_{n\rightarrow\infty}\EE\big[G\big(\E_{0}(S^{n})\big)\big]=\EE\big[G\big(\E_{0}(B)\big)\big],
\]
hence $H^{n}$ converges weakly to $\E_{0}(B)$. As  
\[
\E_{0}(B)(t)=B(t)-\min_{0\leq s\leq t}B(s)=_{d}-B(t)-\min_{0\leq s\leq t}\big(-B(s)\big)=\max_{0\leq s\leq t}B(s)-B(t), 
\]
L\'evy's $M-B$ theorem (see \cite[Ch.\ 2.3]{morters2010brownian}) implies $\E_{0}(B)=_{d} |B|$ and the proof is complete. 
\end{proof}

Now we can use the Lipschitz continuity of column length functionals $\max \E^{j-1}$ to obtain Theorem \ref{theorem:columns} (ii).

\begin{proof}
[\textbf{Proof of Theorem}  
\ref{theorem:columns} \textbf{\textup{(ii)}}.]
First recall that the Motzkin path $\Gamma=\Gamma(X^{n,1/2})$ agrees with the Harris walk $H$ on $[0,n]$, and has only 
downstrokes until it reaches height $0$ on $[n,\infty)$, hence all of its peaks are contained in $[0,n]$. Recall also that the excursion 
operator deletes the peak at the rightmost maximum and preserves all the other peaks. Thus by Lemma 
\ref{lemma:columnbyexcursionop}, we have 
\[
n^{-1/2}\lambda_{j}(n) =n^{-1/2}\max \E^{j-1}(\Gamma)  = n^{-1/2}\max_{[0,n]} \E^{j-1}(H|_{[0,n]})
=  \max_{0\leq t \leq 1} \E^{j-1}(H^{n}).
\]
Lemma \ref{lemma:Lipschitz} in Section \ref{section:Young Diagram} shows that the column length functionals 
$\max \E^{j-1}:C([0,1])\rightarrow \R$ are Lipschitz, so taking powers gives continuous functionals of polynomial growth, and 
the claimed convergence follows from Theorem \ref{brownian_thm_process}. A stronger version of the second part of the assertion 
(concerning orders of column lengths) is shown in Theorem \ref{thm_criticality_order} below. 
\end{proof}

To establish the order of the other top soliton lengths, we appeal to known results about the marginal densities of the ranked maxima of 
$|B|$ over all excursions. To state our conclusions precisely, note that the continuity of $B$ ensures that the random subset 
$\{t \st B(t) \neq 0\}$ of $[0,1]$ is a countable union of maximal disjoint intervals, called the \emph{excursion intervals} of $B$. 
We call an excursion interval $(a,b)$ \emph{complete} if $B(a)=B(b)=0$, and \emph{incomplete} otherwise. All of the excursion 
intervals are complete except possibly the last one $(g(t),1]$, where $g(t)=\sup\{0\leq t \leq 1 \st B(t)=0\}$ is the last zero of $B$. 

Let $\h_{1} \geq \h_{2} \geq \cdots > 0$ be the ranked sequence of values $\sup_{t\in (a,b)} |B_{t}|$ as $(a,b)$ ranges over all 
excursion intervals of $B$. The marginal distributions of the ranked heights over excursions in the reflected Brownian bridge were first 
obtained by Pitman and Yor \cite{pitman2001distribution}. Lagnoux, Mercier, and Vallois \cite{mercier2015probability} pointed out that the 
probability that the maximum of reflected Brownian motion is obtained during the last incomplete excursion is approximately $0.3069$. 
Csaki and Hu \cite{csaki2003lengths} obtained the following explicit expressions for the marginal densities of ranked maxima of 
reflected Brownian motion over all excursions, including the final meander: 

\begin{theorem}
For each $j\geq 1$ and $y>0$, 
\[
\PP\{\h_{j}\geq y\} = 2^{j+1}\sum_{k=0}^{\infty} (-1)^{k} \binom{k+j-1}{k}\left(1-\Phi\left((2k+2j-1)y\right)\right)
\]
where $\Phi(\cdot)$ is the standard normal distribution function.	
\end{theorem}

Accordingly, Theorem \ref{brownian_thm_process} and Lemma \ref{lemma:columnbyexcursionop} imply

\begin{theorem} 
\label{thm_criticality_order}
At criticality, we have that for each $x> 0$
\[
\lim_{n\rightarrow \infty} \PP\left\{ \lambda_{1}(n)\leq x\sqrt{n}  \right\} 
= 1 - 4\sum_{k=0}^{\infty} (-1)^{k} (1-\Phi[(2k+1)x]).
\]
Furthermore, 
\begin{align}\label{eq:critical_lambda_j}
\limsup_{n\rightarrow \infty}\PP\left\{ \lambda_{j}(n) \leq x\sqrt{n} \right\} 
\leq 1- 2^{j+1}\sum_{k=0}^{\infty} (-1)^{k} \binom{k+j-1}{k}(1-\Phi[(2k+2j-1)x]).
\end{align}
In particular, for any $j\geq 1$, $\lambda_{j}(n)=\Theta(\sqrt{n})$.	
\end{theorem}

\begin{remark}
\textup{One might wonder whether the top soliton lengths agree with the  top excursion heights as in the subcritical phase. This would imply that the 
right-hand side of \eqref{eq:critical_lambda_j} gives the limiting distribution of $\lambda_{j}(n)/\sqrt{n}$ for all $j\geq 1$. However, we conjecture that 
this is \emph{not} the case for $p=1/2$. This is because the random variable $Y$ appearing in the proof of Lemma 
\ref{lemma:soliton_lengths_excursion_heights} would then have distribution function $G(x)=1-1/(x+2)$ \cite[Ch.\ 4.1]{Durrett}, and one cannot find 
$x_{n}\in O(\sqrt{n})$, $r_{n}\in\Omega(\sqrt{n})$ with $1-\big(1-\frac{1}{x_{n}+2}\big)^{r_{n}}\rightarrow 0$.}
\end{remark}

\section{Top soliton lengths in the supercritical regime}
\label{section:Supercritical}
In this section, we fix $p\in (1/2,1)$ and prove Theorem \ref{theorem:columns} (iii). The intuition is the following. According to Proposition 
\ref{prop:GWF}, the top soliton lengths are encoded in the first $n$ nodes of a Galton-Watson forest 
$\F=(T_{i})_{i\geq 1}\sim \mathtt{GWF}(p)$. Since the offspring distribution has mean $p/(1-p)>1$ in the supercritical regime, 
the random index $I=\min\{i \st |T_{i}|=\infty\,\}$ is almost surely finite. For $n$ large, about $np$ nodes of the infinite component 
$T_{I}$ will be exposed by the Harris walk, which climbs up along the `leftmost' infinite branch in $T_{I}$. Hence $\lambda_{1}(n)$ should 
behave like the maximum of a random walk with positive drift, and $\lambda_{2}$ will be the maximum height of the first few finite components 
$T_{1},\ldots,T_{I-1}$ together with the `bushes' attached to the infinite branch in $T_{I}$. We prove the $\lambda_{1}(n)$ assertion by 
approximating $\{H_{k}\}$ by $\{S_{k}\}$. For subsequent soliton lengths, we appeal to a duality argument: A backward Harris walk started at the 
highest node will encounter a subcritical Galton-Watson forest, so $\lambda_{2}$ for density $p$ should behave as $\lambda_{1}$ for density $1-p$; 
see Figure \ref{pic:supercritical}. 

\begin{figure*}[h]
	\centering
	\includegraphics[width=1 \linewidth]{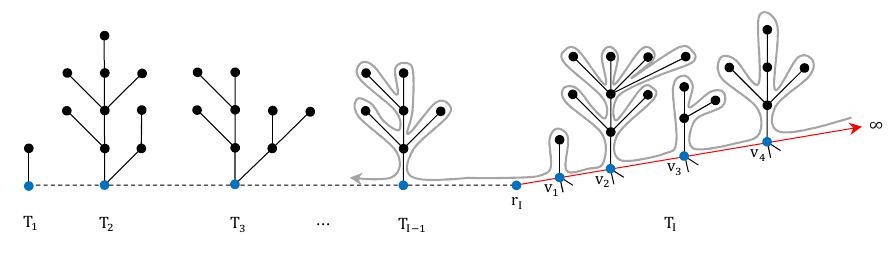}
	\vspace{-0.5cm}
	\caption{Supercritical Galton-Watson forest. $T_{I}$ is the first infinite component and the red ray is the leftmost infinite branch 
                     in $T_{I}$ on which the usual Harris walk climbs up. The grey contour is the backward Harris walk starting from the last 
                     vertex of level $N$, which encounters a subcritical Galton-Watson forest.}
	\label{pic:supercritical}
\end{figure*}

\subsection{Duality and proof of Theorem \ref{theorem:columns} (iii)}
To make the above sketch rigorous, we introduce the notion of flip and dual configurations, which will be used to provide a coupling between the random 
box-ball configurations $X^{n,p}$ and $X^{n,1-p}$. 

Given a random box-ball configuration $X^{n,p}=X^{p}\one_{[1,n]}$, define the associated box-ball configurations 
$\widetilde{X}^{n,p},\hat{X}^{n,p}:\N\rightarrow \{0,1\}$ (which we call the \emph{flip} and \emph{dual}) by
\begin{align}
\label{eq:associated_configs}
(\text{flip}) \qquad &\widetilde{X}^{n,p}(k) = 1-X^{n,p}(k), \\
(\text{dual}) \qquad&\widehat{X}^{n,p}(k) = \left(1-X^{n,p}(n-k+1)\right)\one\{1\leq k\leq n\}.
\end{align}
For each $j\geq 1$, denote $\lambda_{j}(n)=\lambda_{j}(X^{n,p})$ and $\widehat{\lambda}_{j}(n)=\lambda_{j}(\widehat{X}^{n,p})$.

\begin{figure*}[h]
	\centering
	\includegraphics[width=1 \linewidth]{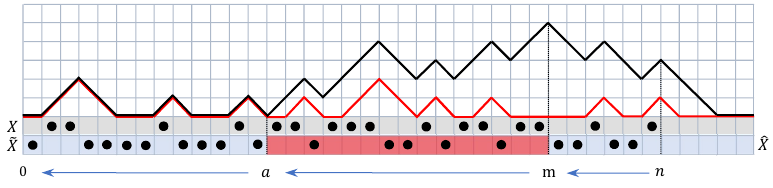}
	\caption{Box-ball configuration $X$, its flip $\widetilde{X}$, and its dual $\widehat{X}$, the latter being read from right to left. 
	             The black path is $\Gamma(X)$ and the red is $\E(\Gamma(X))$. Note that if one reads the Motzkin path $\E(\Gamma(X))$ from 
		     the righmost maximum $\mathtt{m}$ to the beginning of the associated excursion, $\mathtt{a}$, the corresponding box-ball
                     configuration coincides with the dual $\bar{X}$.}
	\label{fig:dual_config}
\end{figure*}

For $1/2<p<1$, it is easy to see from the postive drift that 
$\lambda_{1}(n)= \max_{1\leq k \leq n} H_{k} \approx \max_{1\leq k \leq n} S_{k} \approx S_{n}$, 
where $\{S_{k}\}_{k\geq 0}$ and $\{H_{k}\}_{k\geq 0}$ denote the random walk and Harris walk associated with $X^{p}$. 
For the subsequent soliton lengths, we establish a duality with corresponding soliton lengths in an appropriate subcritical configuration. 
 
\begin{lemma}
\label{lemma:weak_duality}
Fix $\eps>0$, $j\in\N$, and $\theta=(1-p)/p<1$. Then there exists a constant $c=c(p)>1$ such that for each $n,x\geq 1$, 
\[
\PP\left\{ \left| \lambda_{1}(n) - S_{n} \right| \geq x \right\} \leq c\theta^{x/2}
\]
and 
\[
\PP\left\{\left| \lambda_{j+1}(n) - \widehat{\lambda}_{j}(n) \right| \geq x\right\} \leq c\theta^{x/4}.
\]
\end{lemma}

It is straightforward to deduce Theorem \ref{theorem:columns} from the above lemma.

\begin{proof}[\textbf{Proof of Theorem} \ref{theorem:columns} \textbf{\textup{(iii)}}.]
First, we may write 
\[
\frac{\lambda_{1}(n) - (2p-1)n}{2\sqrt{p(1-p)n}}	=\frac{\lambda_{1}(n) - S_{n}}{2\sqrt{p(1-p)n}}+ \frac{S_{n}-(2p-1)n}{2\sqrt{p(1-p)n}}.
\]
The first term on the right-hand side converges in probability to zero by Lemma \ref{lemma:weak_duality}, and the second term 
converges in distribution to a standard normal by the usual central limit theorem, so the first part of the assertion follows from Slutsky's 
theorem. 
	
The concentration inequality for $\lambda_{1}(n)$ is a consequence of Lemma \ref{lemma:weak_duality} and Hoeffding's inequality applied to the 
associated random walk $S_{n}$: 
\begin{align}
\PP\big\{|\lambda_{1}(n)-(2p-1)n|\geq x\big\} & = \PP\big\{|\lambda_{1}(n)-(2p-1)n|\geq x,\,|\lambda_{1}(n)-S_{n}|\geq x/2\big\} \\
 & \qquad +\PP\big\{|\lambda_{1}(n)-(2p-1)n|\geq x,\,|\lambda_{1}(n)-S_{n}|< x/2\big\}\\
 & \leq \PP\big\{|\lambda_{1}(n)-S_{n}|\geq x/2\big\}+\PP\big\{|S_{n}-(2p-1)n|\geq x/2\big\}\\
 &  \leq c\theta^{x/4} + 2e^{-\frac{x^{2}}{8n}}\leq  Ce^{-\frac{x^{2}}{8n}}
\end{align}
for a suitable constant $C$.

Now let $\hat{\mu}_{n}=\log_{\theta^{-1}}\left(\frac{(1-2p)^{2}}{p}n \right)$. (This is the $\mu_{n}$ term from Section \ref{section:Subcritical} but 
with $p$ and  $1-p$ switched since we are now working in the supercritical regime.) Then for $j\geq 1$ fixed, Lemma \ref{lemma:weak_duality} implies
\begin{align}
\PP\big\{\lambda_{j+1}(n)\leq x+ \hat{\mu}_{n}\big\} 
 &= \PP\Big\{\lambda_{j+1}(n)\leq x+\hat{\mu}_{n},\,|\lambda_{j+1}(n)-\hat{\lambda}_{j}(n)|\leq x/2\Big\}\\
 & \qquad +\PP\Big\{\lambda_{j+1}(n)\leq x+\hat{\mu}_{n},\, |\lambda_{j+1}(n)-\hat{\lambda}_{j}(n)| > x/2\Big\}\\
 & \leq \PP\Big\{\hat{\lambda}_{j}(n)\leq 3x/2+\hat{\mu}_{n}\Big\} + c\theta^{x/8}.
\end{align}
The lower bound is established similarly:
\begin{align}
\PP\big\{\hat{\lambda}_{j}(n)\leq (x/2)+ \hat{\mu}_{n}\big\} &= \PP\Big\{\lambda_{j+1}(n)\leq (x/2)+\hat{\mu}_{n},\,  
|\lambda_{j+1}(n)-\hat{\lambda}_{j}(n)| 
\leq x/2\Big\}  \\
& \qquad +\PP\Big\{\hat{\lambda}_{j}(n)\leq x+\hat{\mu}_{n}, \,|\lambda_{j+1}(n)-\hat{\lambda}_{j}(n)| > x/2\big\} \\
&\leq \PP\Big\{\lambda_{j+1}(n)\leq x+\hat{\mu}_{n} \Big\}  + c\theta^{x/8},
\end{align}
and the assertion then follows from Theorem \ref{theorem:columns} (i).
\end{proof}

\subsection{Proof of Lemma \ref{lemma:weak_duality}}

We now prove Lemma \ref{lemma:weak_duality}, establishing a duality principle between the super- and sub-critical box ball systems. Positive drift ensures 
that $S$ and $H$ are not too different, so the first claim seems reasonable since $S$ should attain its maximum over $[0,n]$ near $n$. To explain why the 
second claim is true, let $\widehat{H}\in C_{0}^{+}(\R^{+})$ be the Harris walk for the dual configuration so that 
$\widehat{\lambda}_{1}(n)=\max \widehat{H}$. 
Now $H$ and $\widehat{H}$ are coupled in such a way that the latter is a time-reversal of $\E_{n}(S)$, which is approximated by $\E_{n}(H)$. Thus it 
all boils down to showing that the path $\E_{n}(H)$ pivoted at $n$ is close to $\E(H)=\E_{\m}(H)$, pivoted at the actual location $\m=\m(H)$ of the 
rightmost maximum of $H$. But again positive drift ensures that $H$ attains its maximum near the end. Continuity of the column length functionals can then 
be used to show that the two paths must be close to each other in an appropriate sense.

We begin by introducing the following random variable:
\begin{equation}
\label{eq:def_R}
R=\sup_{k\in \N}\bigg| \min_{1\leq i \leq k }S_{i}\bigg| = -\inf_{k\in \N}  S_{k}.
\end{equation}  
Also, let $\widetilde{S}_{k}$ and $\widetilde{H}_{k}$ be the random walk and Harris walk associated with the flip $\widetilde{X}^{n,p}$. Observe that 
$\widetilde{X}^{n,p}$ has the same law as $X^{n,1-p}$, and for each $1\leq k \leq n$, we have $\widetilde{S}_{k}=(-\xi_{1})+\cdots +(- \xi_{k})=-S_{k}$ 
and  
\begin{equation}
\label{eq:s_max_S_bd}
\widetilde{H}_{k}=\widetilde{S}_{k}-\min_{1\leq i \leq k}\widetilde{S}_{i}=\max_{1\leq i \leq k}S_{i}-S_{k}. 
\end{equation}	

In the following proposition, we show that the maximum of the Harris walk $\{H_{k}\}_{0\leq k\leq n}$ on the interval $[0,n]$ is exponentially concentrated 
around its last value $H_{n}$.  

\begin{prop} 
\label{prop:max_location}
Fix $1/2<p<1$ and let $\widehat{\theta}=p/(1-p)$. Then for any $n,x\geq 1$, 
\begin{equation} 
\label{p_n eq}
\PP\big\{ \max_{0\leq k \leq n}H_{k} - H_{n} \geq x  \big\} = \PP\big\{ \widetilde{H}_{n} \geq x \big\} 
\leq \frac{2p-1}{\widehat{\theta}^{\left\lfloor x\right\rfloor }-1}.
\end{equation}
\end{prop}

\begin{proof}
To show the first inequality, let $\mathtt{a}=\mathtt{a}(X^{n,p})$ be the location of the leftmost global minimum of the random walk 
$\{ S_{k} \}_{0\leq k \leq n}$. Then for any $k\geq \mathtt{a}$, 
\[
H_{k} = S_{k} - \min_{0\leq j \leq k} S_{j} = S_{k} - S_{\mathtt{a}}.
\]
It follows that 
\[
\max_{0\leq k \leq n} H_{k} - H_{n}  = \max_{0\leq k\leq n} \left( S_{k} - S_{\mathtt{a}} \right) - (S_{n}-S_{\mathtt{a}}) 
= \max_{0\leq k\leq n} S_{k} - S_{n}=\widetilde{H}_{n}. 
\]
Now $\widetilde{H}_{k}$ gives the height of the subcritical Harris walk which moves up with probability $1-p$, so writing $\mathtt{b}$ for the beginning 
of the excursion interval containing $n$, Equation \eqref{eq:CDF_height_sub_excursion} shows that 
\begin{align}
\PP\big\{ \max_{0\leq k \leq n}H_{k} - H_{n} \geq x  \big\} & = \PP\big\{ \widetilde{H}_{n} \geq x  \big\}
 \leq \PP\big\{ \max_{\mathtt{b}\leq k\leq n}\widetilde{H}_{k} > x-1 \big\}\\
 & \leq 1-\left( 1-\frac{1-2(1-p)}{\widehat{\theta}^{\left\lfloor x-1\right\rfloor +1}-1}\right)
 \leq \frac{2p-1}{\widehat{\theta}^{\left\lfloor x\right\rfloor }-1}.\qedhere
\end{align}
\end{proof}

\begin{prop}
\label{prop:exp_bound_R_H}
Fix $1/2<p<1$ and $\widehat{\theta}=p/(1-p)>1$. Let $R$ and $\widetilde{H}_{n}$ be as defined at \eqref{eq:def_R} and \eqref{eq:s_max_S_bd}. 
Then there exists a constant $c=c(p)>0$ such that for all $n,x\geq 1$,
\[
\PP\left\{ R+2\widetilde{H}_{n} \geq x  \right\}\leq c\theta^{x/2}.
\]
\end{prop}

\begin{proof}
Casing out according to the value of $\xi_{1}=\one\big\{X^{p}(1)=1\big\}-\one\big\{X^{p}(1)=0\big\}$ shows that for any integer $k\geq 1$, 
$\PP\big\{R\leq k\big\}=p\PP\big\{R\leq k+1\big\}+(1-p)\PP\big\{R\leq k-1\big\}$, hence 
\[
\PP\big\{R=k\big\}=\PP\big\{R\leq k\big\}-\PP\big\{R\leq k-1\big\}=p\big(\PP\big\{R\leq k+1\big\} - \PP\big\{R\leq k-1\big\}\big)
=p\big(\PP\big\{R=k\big\}+\PP\big\{R=k+1\big\}\big),
\]
so $\PP\big\{R=k+1\big\}=\widehat{\theta}^{-1}\PP\big\{R=k\big\}$. It follows that 
\[
\PP\big\{R\geq x\big\}=\sum_{k=x}^{\infty}\PP\big\{R=1\big\}\widehat{\theta}^{1-k}=\frac{\widehat{\theta}^{1-x}}{1-\widehat{\theta}^{-1}}
\PP\big\{R=1\big\}
\]
for each $x\in\N$. Thus Proposition \ref{prop:max_location} implies that there is a $c=c(p)>0$ such that
\[
\PP\left\{ R+2\widetilde{H}_{n} \geq x  \right\} \leq \PP\{R\geq x/2\}+\PP\{H\geq x/2\}
 \leq \frac{\widehat{\theta}^{1-\lfloor x/2 \rfloor}}{1-\widehat{\theta}^{-1}} + \frac{2p-1}{\widehat{\theta}^{\left\lfloor x/2\right\rfloor }-1}
\leq c\widehat{\theta}^{-x/2}
\]
for all $x\geq 2$.
\end{proof}

We are now ready to prove Lemma \ref{lemma:weak_duality}.

\begin{proof}[\textbf{Proof of Lemma} \ref{lemma:weak_duality}.]
Fix $n\geq j$ and let $R$ and $\widetilde{H}_{n}$ be as defined at \eqref{eq:def_R} and \eqref{eq:s_max_S_bd}, respectively. According to the 
exponential bound in Proposition \ref{prop:exp_bound_R_H}, it suffices to show the following inequalities: 
\begin{equation}
\label{RQ eq}
\bigg|\lambda_{1}(n) - S_{n}\bigg| \leq R+2\widetilde{H}_{n} \quad \text{and}\quad 
\bigg|\lambda_{j+1}(n)-\widehat{\lambda}_{j}(n)\bigg| \leq 2R+4\widetilde{H}_{n}.
\end{equation}
Note that the first inequality in \eqref{RQ eq} follows from Lemma \ref{lemma:columnbyexcursionop} and the triangle inequality upon observing that 
\[
\Big|\max_{1\leq k \leq n}H_{k} - \max_{1\leq k \leq n}S_{n}\Big| 
\leq \max_{1\leq k \leq n} \Big|H_{k} - S_{k}\Big|=\max_{1\leq k \leq n}\Big| \min_{1\leq i \leq k }S_{i}\Big|
\leq \sup_{k\in \N} \Big| \min_{1\leq i \leq k }S_{i}\Big|=R.
\]

To establish the second inequality, let $n^{\ast}:=\m(S \one_{[0,n]})$ denote the rightmost maximum of $S$ on $[0,n]$, and 
define the sequence of random variables $\{\check{S}_{k}\}_{0\leq k\leq n}$ by $\check{S}_{k}=S_{k}$ for all $k\neq n$ and 
$\check{S}_{n}=S_{n^{\ast}}$. As usual, let $\check{S}$ denote the linear interpolation of $\{\check{S}_{k}\}$. By construction,
$\lVert \check{S} - S \rVert_{\infty} =\widetilde{H}_{n}$. Also, observe that $\E_{n}(S)(n)=0=\E(\check{S})(n)$, and for $0\leq j <n$, writing 
$m_{j}=\min(S_{j},\ldots,S_{n-1})$, we have $\E_{n}(S)(j)=S_{j}-\min(m_{j},S_{n})$ and 
$\E(\check{S})(j)=S_{j}-\min(m_{j},S_{n^{\ast}})=S_{j}-m_{j}$. If $\min(m_{j},S_{n})=S_{n}$, then $m_{j}=S_{n}+\one\{S_{n}<m_{j}\}$. 
It follows that 
\[
\E_{n}(S)(j)=\E(\check{S})(j)+m_{j}-\min(m_{j},S_{n})=\E(\check{S})(j)+\one\{S_{n}<m_{j}\}. 
\]

Writing $\widehat{S}_{k}=-(S_{n}-S_{n-k})$ for the random walk associated with the dual configuration, we see that the Harris walk 
$\widehat{H}_{k}$ can be written as  
\[
\widehat{H}_{k} = (S_{n-k}-S_{n}) - \min_{0\leq j \leq k} (S_{n-j}-S_{n}) = S_{n-k}  -\min_{n-k\leq i \leq n}S_{i}
 = \E(\check{S})(n-k)+\one\{S_{n}<m_{n-k}\}
\]
for all $0\leq k \leq n$. As $S_{n}<m_{n-k}$ implies $\widetilde{H}_{n}=\lVert \check{S} - S \rVert_{\infty}\geq 1$, we have 
\[
\left| \widehat{H}_{k} - \E(\check{S})(n-k) \right| \leq \widetilde{H}_{n}.
\]
for all $k\geq 1$.
Since the functional $\max \E^{j-1}$ is invariant under time reversal, the above observation together with  Lemmas \ref{lemma:columnbyexcursionop} 
and \ref{lemma:Lipschitz} yields
\[
\bigg|\widehat{\lambda}_{j}(n) - \max \E^{j}(\check{S})\bigg|
 = \bigg|\max \E^{j-1}(\widehat{H}) - \max \E^{j}(\check{S})\bigg| \leq  2\widetilde{H}_{n}.
\]
Finally, the triangle inequality, Lemma \ref{lemma:columnbyexcursionop}, and Lemma \ref{lemma:Lipschitz} give 
\begin{align*}
\bigg|\lambda_{j+1}(n)-\widehat{\lambda}_{j}(n)\bigg| & \leq \left| \max \E^{j}(H) - \max \E^{j}(S)\right|
 + \left| \max \E^{j}(S) - \max \E^{j}(\check{S})\right|+2\widetilde{H}_{n} \\
 & \leq 2\lVert  H -  S\rVert_{\infty} + 2\lVert S-\check{S} \rVert_{\infty} +2\widetilde{H}_{n} \leq 2R+4\widetilde{H}_{n}.\qedhere
\end{align*}
\end{proof}

\section{Random 312-avoiding permutations}
\label{section:Permutation}

In this section, we discuss some relations between box-ball systems and 312-avoiding permutations and prove Theorem \ref{theorem:permutations}. 

Recall that for a given permutation $\sigma$, one can use the Robinson-Schensted algorithm (see \cite[Ch. 3.1]{Sagan}) to obtain a pair of standard Young 
tableaux with common shape $\mathtt{RS}(\sigma)$. Greene's theorem \cite{greene1982extension} relates the sum of the lengths of the first $k$ rows 
(resp. columns) of the Young diagram $\mathtt{RS}(\sigma)$ to the length of a longest subsequence in $\sigma$ which can be obtained by taking the union 
of $k$ increasing (resp. decreasing) subsequences. In Proposition \ref{prop:permutation_Greene}, we show that if $\sigma$ is $312$-avoiding, then a 
`naive' version of Greene's theorem holds: We can subsequently delete longest increasing/decreasing subsequences to obtain subsequent row/column 
lengths of $\mathtt{RS}(\sigma)$. Hence, roughly speaking, Theorem \ref{theorem:permutations} gives the asymptotics of the `$k^{\text{th}}$ longest' 
increasing/decreasing subsequences of a random $312$- (or $231$-) avoiding permutation. 

For a precise statement, we introduce some notation. Given two finite sequences $\alpha$, $\beta$ of positive integers, denote by $\alpha\setminus\beta$ 
the sequence obtained by deleting all elements in $\beta$ from $\alpha$. Denote by $\alpha_{+}$ (resp. $\alpha_{-}$) an arbitrary longest increasing 
(resp. decreasing) subsequence of $\alpha$. Furthermore, let $\alpha_{+}^{\ast}$ (resp. $\alpha_{-}^{\ast}$) be the unique longest increasing 
(resp. decreasing) subsequence in $\alpha$ such that the sum of all numbers used in $(\alpha_{+}^{\ast})^{-1}$ (resp. $(\alpha_{-}^{\ast})^{-1}$) is as 
small (resp. large) as possible. This ensures that $\sigma_{+}^{\ast}$ (resp. $\sigma_{-}^{\ast}$) is the `leftmost' (resp. `rightmost') longest increasing 
(resp. decreasing) subsequence in $\sigma$. For instance, if $\sigma = 146532$, then both $146$ and $145$ are longest increasing subsequences, where 
the former is $\sigma_{+}^{\ast}$. The following proposition is proved in Appendix \ref{subsection:statistics of 312-avoiding permutations}.

\begin{prop}
\label{prop:permutation_Greene}
Let $\tau$ be a 312-avoiding permutation and fix arbitrary $\tau_{-}$. Then $\mathtt{RS}(\tau\setminus \tau_{-})$ is obtained from $\mathtt{RS}(\tau)$ 
by deleting its first column. Moreover, $\mathtt{RS}(\tau\setminus \tau_{+}^{\ast})$ is obtained from $\mathtt{RS}(\tau)$ by deleting its first row.
\end{prop}

In order to prove Theorem \ref{theorem:permutations}, we begin by explaining (an equivalent version of) the construction of the time-invariant Young diagram 
introduced in \cite{torii1996combinatorial}, which was built upon a connection between box-ball configurations and $312$-avoiding permutations. 
The first step is to map a box-ball configuration $X_{0}$ of $m$ balls to a $312$-avoiding permutation $\sigma=\sigma(X_{0})\in \ess_{m}^{312}$ using 
the pushing and popping stack operations from \cite[Ch.\ 2.2.1]{Knuth}. To do so, label the balls $1,\ldots,m$ from left to right so that the $i^{\text{th}}$ 
ball gets label $i$. Then the one-line notation for $\sigma$ gives the left to right labels of the balls after a single update $X_{0}\mapsto X_{1}$. That is, 
we push the symbol $1$ onto an empty stack at the first ball and then, advancing to the right, pop the top of the stack off for storage at each empty box and 
push $k$ onto the stack upon encountering the $k^{\text{th}}$ ball. See Figure \ref{fig:stack} for an illustration. 

\begin{figure*}[h]
	\centering
	\includegraphics[width=0.95 \linewidth]{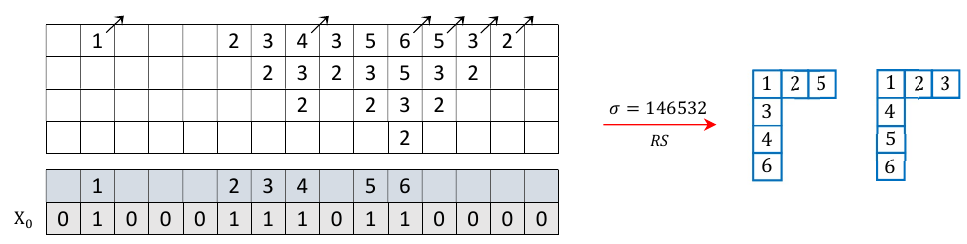}
	\caption{Construction of the $312$-avoiding permutation corresponding to the box-ball environment in the bottom row 
		     via push-pop operations. The second row from the bottom indicates the labels of the balls, and the columns in the upper 
		     table give the contents of the right-sweeping stack. The resulting permutation is $\sigma=1 \, 4 \, 6 \, 5 \, 3 \, 2$.}
	\label{fig:stack}
\end{figure*}

To get a Young diagram from this stack-representable permutation $\sigma(X_{0})$, one applies the Robinson-Schensted algorithm 
to obtain a pair of standard Young tableaux, and records their common shape as $\mathtt{RS}(\sigma(X_{0}))$. 
It was shown in \cite{torii1996combinatorial} that $\mathtt{RS}(\sigma(X_{s}))$ is invariant in $s\geq 0$ and its $j^{\text{th}}$ column length is 
the $j^{\text{th}}$ longest soliton length in the system. Thus, by Lemma \ref{lemma:invariance_Youngdiagram}, this construction gives the same 
Young diagram which was obtained by hill-flattening operations applied to the Motzkin path. 

\begin{prop}
\label{prop:permutation_Motzkin}
Let $X_{0}:\N_{0}\rightarrow \{0,1\}$ be a finitely supported box-ball configuration. Then
\[
\mathtt{RS}(\sigma(X_{0})) = \Lambda(X_{0})=  \Lambda(\Gamma(X_{0})).
\]
\end{prop}

The following proposition (proved in Appendix \ref{subsection:statistics of 312-avoiding permutations}) shows that there is a bijection between 
$312$-avoiding  permutations of length $n$ and Dyck paths of length $2n$ which `factors through' box-ball configurations in a natural way. Let 
$\ess_{n}^{312}$ be the set of all $312$-avoiding permutations of length $n$ and let $\textup{Dyck}_{2n}$ be the set of all Dyck paths of length 
$2n$---that is, lattice paths from $(0,0)$ to $(2n,0)$ consisting only of upstrokes and downstrokes and never dipping below the horizontal axis. 

\begin{prop}
\label{prop:bijection_312-Dyck}
${}$
\begin{description}[leftmargin=!,labelwidth=\widthof{\bfseries [(ii)]}]
	\item[(i)] There exists a bijection $\varphi:\textup{Dyck}_{2n}\rightarrow \ess_{n}^{312}$.
	\item[(ii)] For each $\tau\in \ess_{n}^{312}$ and $\digamma\in \textup{Dyck}_{2n}$ such that $\varphi(\digamma)=\tau$, there is a box-ball 
	configuration $X_{0}$ such that $\tau=\sigma(X_{0})$ and $\digamma=\Gamma(X_{0})$.
\end{description}
\end{prop}

We now prove Theorem \ref{theorem:permutations} using similar ideas from the proof of Theorem \ref{theorem:rows} 
together with some known results on random Dyck paths and random walk excursions.

\begin{proof}[\textbf{Proof of Theorem} \ref{theorem:permutations}]
Recall that $\sigma\in \ess_{n}^{312}$ if and only if $\sigma^{-1}\in \ess_{n}^{231}$, so the map $\sigma \mapsto \sigma^{-1}$ preserves the uniform 
distribution on the sets $\ess_{n}^{312}$ and $\ess_{n}^{231}$. Moreover, $\mathtt{RS}(\sigma)=\mathtt{RS}(\sigma^{-1})$. Hence it suffices to prove 
the assertion only for the 312-avoiding permutations. 	
		
Let $\digamma$ be a Dyck path of length $2n$ and let $\tau=\varphi(\digamma)$ be the corresponding $312$-avoiding permutation. Proposition 
\ref{prop:bijection_312-Dyck} enables us to choose a box-ball configuration $X_{0}$ such that $\tau=\sigma(X_{0})$ and $\digamma=\Gamma(X_{0})$, 
and Proposition \ref{prop:permutation_Motzkin} implies that $\mathtt{RS}(\tau)=\Lambda(\digamma)$. If we denote by $\Sigma^{n}$ and 
$\digamma^{n}$ uniformly random elements of $\ess^{312}_{n}$ and $\textup{Dyck}_{2n}$, this yields 
\begin{equation}
\label{eq:pf_corollary_permutation_1}
\mathtt{RS}(\Sigma^{n}) =_{d} \Lambda(\digamma^{n}).
\end{equation}
Now the contour process described in Subsection \ref{subsection:Trees} gives a bijection between Dyck paths of length $2n$ and rooted plane trees with 
$n+1$ nodes, so part (i) of Theorem \ref{theorem:permutations} follows from \eqref{eq:pf_corollary_permutation_1} and Proposition 
\ref{prop:Youngdiagram_forest}. 
	
Part (iii) also follows easily from known results. Indeed, it is well known that under diffusive scaling the random walk excursion converges weakly to a 
standard Brownian excursion \cite{aldous1993continuum}. Moreover, by Theorem \ref{thm:drmota}, the convergence is polynomial. 
Thus (iii) follows from \eqref{eq:pf_corollary_permutation_1} and Lemmas \ref{lemma:columnbyexcursionop} and \ref{lemma:Lipschitz}.	
	
Lastly, we establish the strong law for $\rho_{i}(\digamma^{n})$ stated in part (ii). To begin, fix $i\geq 1$, and let $\{S_{k}\}_{k\geq 0}$ be a 
simple symmetric random walk with $S_{0}=0$. We may view the uniformly random Dyck path $\digamma^{n}$ of length $2n$ as the trajectory 
of $S_{k}$ over the interval $[0,2n]$ conditioned to stay non-negative and satisfy $S_{2n}=0$. By (\ref{eq:pf_corollary_permutation_1}) 
and the hill-flattening procedure, $\rho_{i}(\digamma^{n})$ equals the number of subexcursions of $\digamma^{n}$ of height $i$. Recall the definitions of 
$\mu_{i}$ and $J_{\ell}^{i}$ given in Lemma \ref{lemma:chernoff} and above the same lemma, respectively. 
Let $N_{i}(n) = \sum_{\ell=0}^{n} J_{\ell}^{i}$. 
Then $N_{i}(2n) = \rho_{i}(\digamma^{n}) $, so for all $n\geq 1$ and $\eps<1/2n$, 
\begin{eqnarray*}
\PP\left\{\left|\frac{\rho_{i}(\digamma^{n})}{2n} - \mu_{i} \right| > \eps \right\} & \leq
 & \PP\left\{\left| \frac{N_{i}(2n)}{2n} - \mu_{i} \right|>\eps \,\bigg| \, \text{$S_{k}$ is a Dyck path over $[0,2n]$ } \right\}\\
 &\leq & \frac{\PP\{|N_{i}(2n)/2n - \mu_{i}|>\eps\} }{\PP\left\{ \text{$S_{k}$ is a Dyck path over $[0,2n]$ } \right\}}.
\end{eqnarray*}

It is well known that the number of Dyck paths of length $2n$ is the $n^{\text{th}}$ Catalan number $\frac{1}{n+1} \binom{2n}{n}$, so by Stirling's 
approximation, $\PP\left\{\text{$S_{k}$ is a Dyck path over $[0,2n]$ } \right\}\sim n^{-3/2}/\sqrt{\pi}$. Now by Lemma \ref{lemma:chernoff} with 
$m=4$ and $\eps=\eps(n)= 1/\log n \searrow 0$, we get 
\[
\PP\left\{\left| \frac{\rho_{i}(\digamma^{n})}{2n} - \mu_{i} \right| > \eps(n) \right\} = O((\log n)^{8} n^{-3/2}).
\]
In particular, these probabilities are summable, so the first Borel-Cantelli lemma implies $\rho_{i}(\digamma^{n})/2n \rightarrow \mu_{i}$ a.s. as 
$n\rightarrow \infty$. 
\end{proof}

\pagebreak

\appendix
\renewcommand{\appendixname}{}
\numberwithin{equation}{section}
\numberwithin{figure}{section}
\numberwithin{table}{section}

\section{Proofs of combinatorial lemmas}
\label{Section:proof of combinatorial Lemmas}

In this appendix, we provide proofs of Lemmas \ref{lemma:invariance_Youngdiagram}, \ref{lemma:columnbyexcursionop}, and 
\ref{lemma:Lipschitz}, and Propositions \ref{prop:permutation_Greene} and \ref{prop:bijection_312-Dyck}.

\subsection{Time invariance of the Young diagram} 
\label{Subsection:time invariance of the Young diagram}

Our proof of Lemma \ref{lemma:invariance_Youngdiagram} is similar to the argument from \cite{torii1996combinatorial}, 
which is formulated in terms of Dyck words intead of Motzkin paths. The argument is simplified by Proposition \ref{prop:readbackward}.

To begin, recall that given a box-ball configuration $X_{s}$ of finite support, the associated lattice path $\Gamma(X_{s})$ is 
constructed by reading $X_{s}$ from left to right: Starting at height $0$, increase by $1$ every time a $1$ is encountered, 
decrease by $1$ whenever a $0$ is encountered at positive height, and remain at height $0$ otherwise. A simple but useful 
observation is that reading $X_{s}$ from right to left produces the lattice path $\Gamma(X_{s-1})$. More precisely, let 
$(X_{s})_{s\geq 0}$ be a box-ball system started from a finitely supported configuration $X_{0}$. For each $s\geq 0$, let 
$r_{s}=\max\{k\geq 0 \st X_{s}(k)=1\}$ be the location of the rightmost 1 at time $s$. Construct a (backward) lattice path 
$\cev{\Gamma}(X_{s}):\N_{0}\rightarrow \N_{0}$ by $\cev{\Gamma}(X_{s})_{k}=0$ for $k\geq r_{s}$ and 
\[
\cev{\Gamma}(X_{s})_{k} = 
\begin{cases}
\cev{\Gamma}(X_{s})_{k+1} +1 & \text{if $X_{s}(k+1)=1$} \\ 
\cev{\Gamma}(X_{s})_{k+1} -1 & \text{if $X_{s}(k+1)=0$ and $\cev{\Gamma}(X_{s})_{k+1}\geq 1$} \\
0 & \text{if $\cev{\Gamma}(X_{s})_{k+1}=X_{s}(k+1)=0$} \\ 
\end{cases}
\]
for $0\leq k < r_{s}$. See Figure \ref{invariance_1_pic} for an illustration. 
In this appendix, we denote the ordinary lattice path $\Gamma$ by $\vec{\Gamma}$ to emphasize the reading direction. 

\begin{figure*}[h]
\centering
\includegraphics[width=.95 \linewidth]{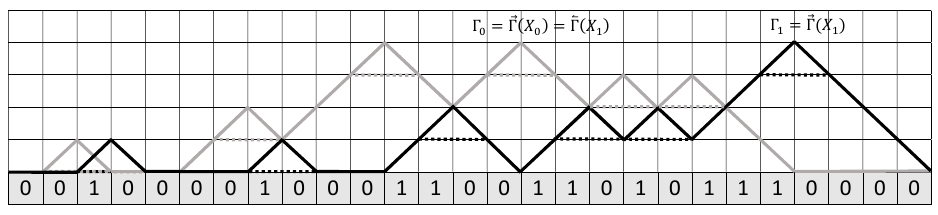}
	\caption{The environment is $X_{1}$ where $X_{0}$ is the environment given in Figure \ref{Young diagram}. 
	The black path is $\vec{\Gamma}(X_{1})$ and the grey path is $\cev{\Gamma}(X_{1})$. Notice that the latter 
	coincides with the black path in Figure \ref{Young diagram}.}
\label{invariance_1_pic}
\end{figure*}

\begin{prop}
\label{prop:readbackward}
$\displaystyle{\cev{\Gamma}(X_{s+1}) = \vec{\Gamma}(X_{s})}$ for all $s\geq 0$.
\end{prop}

\begin{proof}
Fix $s\geq 0$, and observe that both paths are $0$ on $[r_{s+1},\infty)$, so the assertion holds on this interval. 
Now suppose the paths agree on $[k+1,\infty)$ for some $k<r_{s+1}$. We must show that 
$\cev{\Gamma}(X_{s+1})_{k} = \vec{\Gamma}(X_{s})_{k}$. 
	
The definition of the box-ball dynamics shows that $X_{s+1}(k+1)=1$ if and only if 
$\vec{\Gamma}(X_{s})_{k}-1=\vec{\Gamma}(X_{s})_{k+1}$, hence  
\[
\vec{\Gamma}(X_{s})_{k} - \vec{\Gamma}(X_{s})_{k+1} = 1 
 \Longleftrightarrow  X_{s+1}(k+1) = 1
 \Longleftrightarrow \cev{\Gamma}(X_{s+1})_{k} - \cev{\Gamma}(X_{s+1})_{k+1} = 1.
\]
The inductive hypothesis implies
\begin{align*}
\vec{\Gamma}(X_{s})_{k} = \vec{\Gamma}(X_{s})_{k+1}  
& \Longleftrightarrow  \text{$X_{s+1}(k+1) = 0$ and $\vec{\Gamma}(X_{s})_{k+1}=0$} \\
& \Longleftrightarrow  \text{$X_{s+1}(k+1) = 0$ and $\cev{\Gamma}(X_{s+1})_{k+1}=0$} \\
& \Longleftrightarrow  \cev{\Gamma}(X_{s+1})_{k} = \cev{\Gamma}(X_{s+1})_{k+1}
\end{align*}
and 
\begin{align*}
\vec{\Gamma}(X_{s})_{k} - \vec{\Gamma}(X_{s})_{k+1}=-1  
& \Longleftrightarrow  \text{$X_{s+1}(k+1) = 0$ and $\vec{\Gamma}(X_{s})_{k+1}\geq 1$} \\
& \Longleftrightarrow  \text{$X_{s+1}(k+1) = 0$ and $\cev{\Gamma}(X_{s+1})_{k+1}\geq 1$} \\
& \Longleftrightarrow  \cev{\Gamma}(X_{s+1})_{k} - \cev{\Gamma}(X_{s+1})_{k+1}=-1.\qedhere
\end{align*}
\end{proof}

To facilitate the proof of Lemma  \ref{lemma:invariance_Youngdiagram}, it is convenient to reformulate the procedure for building 
Young diagrams row by row: Rather than flatten hills, we \emph{contract} peaks by deleting the upstroke-downstroke 
pair and then identifying the endpoints so that the path remains connected. The number of hills after flattening is the same as 
the number of peaks after contracting, so everything is exactly same as before. The advantage here is that if one begins with an 
$h$-restricted Motzkin path, then the hills are always peaks and the Motzkin paths are always $h$-restricted. Moreover, the 
contraction operation can be understood in terms of the environment as deleting $1 \, 0$ patterns.

\begin{proof}[\textbf{Proof of Lemma} \ref{lemma:invariance_Youngdiagram}.] 
The second part of the assertion clearly holds for all stable box-ball configurations $X_{0}:\N\rightarrow \{0,1\}$ of finite support. 
Since the system always stabilizes, the second part follows from the time invariance as stated in the first part. 
	
Now let $(X_{s})_{s\geq 0}$ be as before. To show the time invariance of $\Lambda(X_{s})$, recall that the construction 
of $\Lambda(X_{s})$ begins by counting the  number of peaks in the path corresponding $X_{s}=X_{s}^{(0)}$. This is equal to the 
number of $1 \, 0$ patterns, which is equal to the number of $1$-strings, which is equal to the number of $0 \, 1$ patterns. 
The length of the first row of $\Lambda(X_{s})$ is given by this number. The peaks are then contracted by deleting the $1 \, 0$ 
patterns from $X_{s}$ to obtain $X_{s}^{(1)}$ and the process is repeated with $\Gamma(X_{s}^{(1)})$. At each step, the 
$1$-strings are counted, the diagram is updated, and the $1 \, 0$ patterns are deleted, continuing until the path consists only of 
$h$-strokes.
	
The key insights are that the number of $1$ strings is the same regardless of whether the environment is read from left to right 
or conversely, and that the number of $1$-strings after $1 \, 0$ patterns are deleted is the same as the number of $1$ strings 
after $0 \, 1$ patterns are deleted. In the first case, each $1$ string either decreases in length by $1$ (possibly disappearing), 
or it merges with the string on its right. In the second, each string either decreases in length by $1$ or merges with the string 
on its left.
	
\begin{figure*}[h]
\centering
\includegraphics[width=0.7 \linewidth]{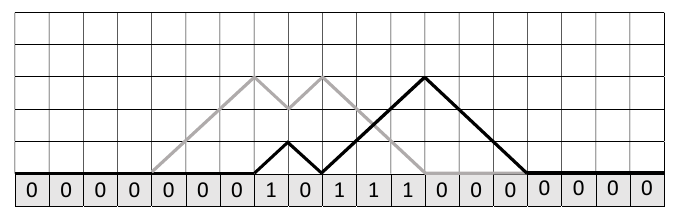}
	\caption{The environment is formed by deleting either $1 \, 0$ patterns or $0 \, 1$ patterns from the environment in 
	Figure \ref{invariance_1_pic}. The corresponding left-right (black) and right-left (grey) lattice paths have the same number 
	of hills as the flattened paths in Figure \ref{invariance_1_pic}.}
\label{invariance_2_pic}
\end{figure*}
\vspace{0.2cm}
	
Now for any fixed $s\geq 0$, $\vec{\Gamma}(X_{s})$ and $\vec{\Gamma}(X_{s+1})$ can be read off from $X_{s+1}$ by 
proceeding from right to left and from left to right, respectively. The update rule for the former is to count $1$-strings and then 
delete $1 \, 0$ patterns, and the update rule for the latter is to count $1$-strings and then delete $0 \, 1$ patterns. 
By the previous observations, both result in the same final Young diagram. 

At this point, it remains only to show that soliton lengths are given by the column lengths of the Young diagram 
$\Lambda(X_{0})$. To see that this is so, observe that the path $\Gamma(X_{\tau})$, which corresponds to the first 
stable configuration, consists of a series of single peaks of nondecreasing height, each as tall as the length of the associated soliton. 
As each flattening step reduces the height of the peaks by $1$, we see that the number of rows of $\Lambda(X_{\tau})$ 
having length at least $\ell$ corresponds to the number of solitons of length at least $\ell$. Therefore, the columns of 
$\Lambda(X_{\tau})$ encode the soliton lengths, so the same is true of $\Lambda(X_{0})$ by invariance.
\end{proof}

\subsection{Extracting column lengths with excursion operators}
\label{subsection:column lengths from excursion operators}

In this subsection, we prove Lemma \ref{lemma:columnbyexcursionop}. The key observation is that the hill-flattening and excursion operators 
commute on the space of Motzkin paths.

To begin, we need to establish a couple of technical results. First, for any interval $I\subseteq \R$, recall that $C_{0}^{+}(I)$ denotes the space of 
continuous functions $I\rightarrow [0,\infty)$ with compact support. For any $f\in C_{0}^{+}(I)$, we denote by $\textup{supp}^{+}(f)$ the open set 
$\{x\in I\,:\, f(x)>0 \}$, which is a finite disjoint union of open intervals. Accordingly, we may write 
$\textup{supp}^{+}(f)=\bigsqcup_{i=1}^{n} (c_{i},d_{i})$, where $d_{i}<c_{j}$ if $i<j$. 
We call $J_{i}:=(c_{i},d_{i})$ the $i^{\text{th}}$ \textit{excursion interval} of $f$. Recall that $\I(\Gamma)$ denotes the set of hill intervals 
of $\Gamma$ (see the beginning of Subsection \ref{subsection:Youngdiagram_construction}).

\begin{prop}
\label{prop:numberofhills}
Fix a Motzkin path $\Gamma$ and let $x\in\N$ be contained in a hill interval $I_{x}$ of $\,\Gamma$. 
Denote $\textup{supp}^{+}(\E_{x}(\Gamma))=\bigsqcup_{i=1}^{n} J_{i}$ as above. Then $\Gamma-\E_{x}(\Gamma)$ 
is constant on each $J_{i}$. In addition, $\I(\E_{x}(\Gamma))=\I(\Gamma) \setminus\{I_{x}\}$ and  
$\max\E^{j-1}(\Gamma)\geq 1$ for all $1\leq j \leq \rho(\Gamma)$. 
\end{prop}

\begin{proof}
To establish the first part, write $M=\Gamma_{x}\geq 0$, and define integers 
$a_{0} < a_{1}<\cdots<a_{M-1}<a_{M}= x =b_{M}<b_{M-1}<\cdots<b_{1}<b_{0}$ by 
\[
a_{i}=\max\{k\leq x \st \Gamma_{k}=i\},\quad b_{i}=\min\{k\geq x \st \Gamma_{k}=i\}
\]
for each $0\leq i \leq M$. In words, they are the first locations where $\Gamma$ has height $i$ when moving to the left and 
right from $x$; see Figure \ref{Young diagram:column}. To simplify notation, we set $a_{-1}=0$ and  $b_{-1}=\infty$. 
Now $\Gamma_{y}-\E_{x}(\Gamma)_{y}=\min\big\{\Gamma_{z} \st x\wedge y \leq z \leq x \vee y\big\}$, so on $\N_{0}$
\[
\Gamma - \E_{x}(\Gamma) = \sum_{i=0}^{M-1} i \big(\one_{(a_{i-1},a_{i}]} + \one_{[b_{i},b_{i-1})}\big) 
+ M \one_{(a_{M-1},b_{M-1})}.
\]
It follows that $\E_{x}(\Gamma)$ vanishes at the $a_{i}$'s and $b_{i}$'s, and differs from $\Gamma$ by a constant on 
$(a_{M-1},b_{M-1})$ and each interval of the form $(a_{k-1},a_{k}]$ or $[b_{k},b_{k-1})$, $0\leq k \leq M-1$. $J_{i}$ is the 
$i^{\text{th}}$ such interval (from left to right) where $\E_{x}(\Gamma)$ is not constant. This shows the first part of the assertion. 

The preceding argument also implies that $\I(\E_{x}(\Gamma))\subseteq \I(\Gamma)$. In addition, 
$\E_{x}(\Gamma)=0$ on $[a_{M-1},b_{M-1}]$ and $I_{x}\subseteq (a_{M-1},b_{M-1})$, so $I_{x}$ is not a hill interval of 
$\E_{x}(\Gamma)$. Finally, the definition of the $a$ and $b$ terms ensures that if $J\in \I(\Gamma)\setminus \{I_{x}\}$, 
then either $J\subseteq (a_{i-1},a_{i}]$ or $J\subseteq [b_{i},b_{i-1})$ for some $0\leq i \leq M-1$. Since $\E_{x}(\Gamma)$ is a 
vertical translate of $\Gamma$ on these intervals, $J$ must be a hill interval of $\E_{x}(\Gamma)$. This shows 
$\I(\E_{x}(\Gamma)) = \I(\Gamma)\setminus\{I_{x}\}$.
	
Lastly, taking $x=\m$ in the first part gives $\I(\E(\Gamma))= \I(\Gamma) \setminus\{I_{\m}\}$, and the second 
part of the second assertion follows from the first since each application of $\E$ removes a single hill interval and the height of a Motzkin 
path is at least one while hill intervals remain.
\end{proof}

\begin{prop}
\label{prop:pivotonconstantinterval}
For any interval $I\subseteq \R^{+}$, $f\in C_{0}^{+}(I)$, $x,y\in I$, if $f$ is constant on the interval $[x,y]\subseteq I$, then 
$\E_{x}(f)=\E_{y}(f)$. 
\end{prop}

\begin{proof}
Casing out according to whether $t<x$, $x\leq t\leq y$, or $t>y$ shows that 
\[
\min_{t\wedge x\leq s \leq t\vee x} f(s) = \min_{t\wedge y\leq s \leq t\vee y} f(s).\qedhere
\]
\end{proof}

\begin{prop}
\label{prop:hillandexcursioncommute}
For any Motzkin path $\Gamma$ and any $x\in\N$ contained in a hill interval of $\Gamma$, 
$\E_{x}\circ \HH (\Gamma)=\HH\circ \E_{x}(\Gamma)$. 
In particular, $\E\circ \HH (\Gamma)=\HH\circ \E(\Gamma)$.
\end{prop}

\begin{proof}
Let $\m=\m(\Gamma)$ and $\m^{\ast}=\m(\HH(\Gamma))$. 
Note that $\m< \m^{\ast}$ and that $\HH(\Gamma)$ is constant on $[\m,\m^{\ast}]$. This holds for any Motzkin path $\Gamma$. Thus by Proposition 
\ref{prop:pivotonconstantinterval} with $I=\R^{+}$, it suffices to prove the first part.
To this end, we first note that for any $k\in \N_{0}$, 
\[
\min_{k\wedge x \leq y \leq k\vee x }\Gamma_{y}-\min_{k\wedge x \leq y \leq k\vee x }\HH(\Gamma)_{y}  
= \one\{k\in I_{x}\},
\]
where $I_{x}$ denotes the hill interval of $\Gamma$ containing $x$. Indeed, $\HH(\Gamma)=\Gamma-1$ on $I_{x}$, so the left-hand side is $1$ for all 
$k\in I_{x}$. 
Now fix $k\notin I_{x}$, and let $x_{\ast}$ be the location of the leftmost minimum of $\Gamma$ over the interval $[k\wedge x , k\vee x]$. 
Then $x_{\ast}$ is an integer which is not contained in any hill interval of $\Gamma$, so 
$\HH(\Gamma)_{x_{\ast}}=\Gamma_{x_{\ast}}$. Moreover, $x_{\ast}$ minimizes $\HH(\Gamma)$ on $[k\wedge x , k\vee x]$ 
since the only integer points with $\HH(\Gamma)_{y}<\Gamma_{y}$ are those contained in a hill interval of $\Gamma$, in which case 
$\Gamma_{y}\geq \Gamma_{x_{\ast}}+1$. This shows that the left-hand side is $0$ for $k\notin I_{x}$ as desired.  

In conjunction with Proposition \ref{prop:numberofhills}, we have  
\begin{eqnarray*}
	\E_{x}(\HH(\Gamma))_{k}
	&=& \HH(\Gamma)_{k} - \min_{k\wedge x \leq y \leq k\vee x} \HH(\Gamma)_{y}\\
	&=& \HH(\Gamma)_{k} - \min_{k\wedge x \leq y \leq k\vee x} \Gamma_{y} + \one\{k\in I_{x}\}\\
	&=& \begin{cases}
		\E_{x}(\Gamma)_{k}-1 & \text{if $k\in \bigcup \I(\Gamma)\setminus \{I_{x}\}$ } \\
		\E_{x}(\Gamma)_{k} & \text{otherwise}
	\end{cases}\\
	&=& \HH(\E_{x}(\Gamma))_{k}.
\end{eqnarray*}
\end{proof}

Now we prove Lemma \ref{lemma:columnbyexcursionop}.

\begin{proof}
[\textbf{Proof of Lemma} \ref{lemma:columnbyexcursionop}]
Let $\Gamma$ be a Motzkin path and write $\lambda_{j}$ for the length of the $j^{\text{th}}$ column of 
$\Lambda(\Gamma)$ for each $1\leq j \leq \rho(\Gamma)$. We show 
\[
\lambda_{j} = \max \E^{j-1}(\Gamma)
\]
by induction on $\max \, \Gamma$. If the maximum is zero, then the assertion is trivial, so we may assume that 
it holds for all Motzkin paths with maximum less than $M\in \N$. Now fix a path $\Gamma$ with $\max \, \Gamma=M$. 
The inductive hypothesis implies that the assertion holds for $\HH(\Gamma)$ since it has maximum $M-1\geq 0$. 
Moreover, $\Lambda(\HH(\Gamma))$ is obtained by deleting the first row of 
$\Lambda(\Gamma)$. Thus by Proposition \ref{prop:hillandexcursioncommute}, we have  
\begin{align*}
\lambda_{j}-1 &= \max \E^{j-1}(\HH(\Gamma)) \\
&= \max \HH(\E^{j-1}(\Gamma))\\
&= \max \E^{j-1}(\Gamma) -1,
\end{align*}   
where the final equality used the second part of Proposition \ref{prop:numberofhills} to ensure 
$\max \E^{j-1}(\Gamma)\geq 1$ for any $1\leq j \leq \rho(\Gamma)$. 
\end{proof}

\begin{remark}\label{remark:lemma_excursion}
\textup{An easy modification of Proposition \ref{prop:hillandexcursioncommute} and applying the same proof of Lemma \ref{lemma:columnbyexcursionop} 
shows that the excursion operator $\E=\E_{\m}$ in the statement of Lemma \ref{lemma:columnbyexcursionop} could be replaced by $\E_{\m^{\ast}}$, 
where the pivot $\m^{\ast}=\m^{\ast}(\Gamma)$ is chosen to be an arbitrary element in the set $\{ x\geq 0\,:\, \Gamma(x) = \max \Gamma \}$ where 
the Motzkin path $\Gamma$ achieves its maximum. }
\end{remark}

\subsection{Regularity of the column length functionals}
\label{subsection:regularity of column length functionals}

In this subsection we prove Lemma \ref{lemma:Lipschitz}, establishing Lipschitz continuity of the `column length functionals' 
$\max \E^{j-1}(\cdot)$. The general strategy is to show that the column length functionals satisfy a Lipschitz condition on Motzkin 
paths and then extend the result to arbitrary functions in $C_{0}^{+}(\R^{+})$ by an approximation argument. We begin by establishing 
some preparatory results. 

\begin{prop}
\label{prop:Lipschitz}
$ $
\begin{description}[leftmargin=!,labelwidth=\widthof{\bfseries [(ii)]}]
	\item[(i)] Fix an interval $I\subseteq \R^{+}$, a point $b\in I$, and functions $f,g\in C_{0}^{+}(I)$. Then
	\[
	\lVert \E_{b}(f) - \E_{b}(g) \rVert_{\infty} \leq 2\lVert f-g \rVert_{\infty}.
	\]
	\item[(ii)] For any Motzkin paths $f,g\in C_{0}^{+}(\R^{+})$,
	\[
	\lVert \HH(f)-\HH(g) \rVert_{\infty}\leq \lVert f-g \rVert_{\infty}.
	\]
\end{description} 
\end{prop}

\begin{proof}
For (i), the triangle inequality gives  
\[
\lVert \E_{b}(f) - \E_{b}(g) \rVert_{\infty} 
\leq \lVert f-g \rVert_{\infty} + \sup_{t\in I}\bigg| \min_{[t\wedge b, t\vee b]}f -\min_{[t\wedge b, t\vee b]}g \bigg| 
\leq 2\lVert f-g \rVert_{\infty}
\]
since the minima of two functions over a given interval can differ by no more than their maximum difference over the interval.

For (ii), observe that the maximum distance between Motzkin paths is necessarily $\N_{0}$-valued and the claim is 
clearly true if $f=g$, so we may assume that $\left\Vert \HH(f)-\HH(g)\right\Vert _{\infty}\geq1$. Let 
\[
x^{\ast}=\max\big\{ x\in\N\st \big|\HH(f)_{x}-\HH(g)_{x}\big|=\big\Vert \HH(f)-\HH(g)\big\Vert _{\infty}\big\} ,
\]
and assume without loss of generality that $\HH(f)_{x^{\ast}}>\HH(g)_{x^{\ast}}$.
If $x^{\ast}$ is not in a hill interval of $g$, then $g(x^{\ast})=\HH(g)_{x^{\ast}}<\HH(f)_{x^{\ast}}\leq f(x^{\ast})$,
so 
\[
\left\Vert \HH(f)-\HH(g)\right\Vert _{\infty}=\left|\HH(f)_{x^{\ast}}-\HH(g)_{x^{\ast}}\right|
\leq\left|f(x^{\ast})-g(x^{\ast})\right|\leq\left\Vert f-g\right\Vert _{\infty}.
\]
If $x^{\ast}$ is in a hill interval of both $f$ and $g$, then 
\[
\left\Vert \HH(f)-\HH(g)\right\Vert _{\infty}=\left|\HH(f)_{x^{\ast}}-\HH(g)_{x^{\ast}}\right|
=\left|\left(f(x^{\ast})-1\right)-\left(g(x^{\ast})-1\right)\right|\leq\left\Vert f-g\right\Vert _{\infty}.
\]
Finally, suppose that $x^{\ast}$ is in a hill interval $[a,b]$ of $g$ but is not in any hill interval of $f$. Then $g$ is constant 
on $[a,b]$, so our choice of $x^{\ast}$ implies that $f(x^{\ast})\geq f(y)$ for all $y\in[a,b]$. By considering whether or not $
x^{\ast}<b$, we see that we must have $f(x^{\ast}+1)=f(x^{\ast})-1$. A similar consideration of whether $f(x^{\ast})=f(y)$ 
for all $a\leq y\leq x^{\ast}$ leads to the contradiction that $x^{\ast}$ is in a hill interval of $f$. 
\end{proof}	

To state our next result, we say that a function $\varphi:\R\rightarrow\R$ is an \emph{affine scaling} if $\varphi(x)=ax+b$ 
for some $a>0$, $b\in \R$. The set of all affine scalings forms a group under composition. Given $f\in C_{0}^{+}(\R)$ and 
an affine scaling $\varphi$, we write $\varphi^{\ast}(f)$ for the function $f\circ \varphi^{-1}$. 
A function $\Gamma:\R\rightarrow \R^{+}$ is an \emph{extended Motzkin path} if $\Gamma(n)=0$ for all 
$n\leq 0$ and $\Gamma|_{[0,\infty)}$ is a Motzkin path. 

\begin{prop}
\label{prop:Motzkin_approx}
For any $f_{1},f_{2}\in C_{0}^{+}(\R)$ which are not identically zero and any $\eps>0$, there exist affine scalings $\varphi,\psi$ 
and extended Motzkin paths $\Gamma_{1},\Gamma_{2}$ such that $\psi(0)=0$ and for $i=1,2$, the function 
$\bar{f_{i}} = \psi\circ \varphi^{\ast}(\Gamma_{i})\in C_{0}^{+}(\R)$ satisfies
\[
\lVert f_{i}- \bar{f}_{i} \rVert_{\infty}<\eps\quad \text{and}\quad \m(\bar{f}_{i})=\m(f_{i}).
\]
\end{prop}

\begin{proof}
By hypothesis, $\m(f_{1}),\m(f_{2})\in (0,\infty)$. Also, the $f_{i}$'s are uniformly continuous, so there is some 
$\delta>0$ such that $|x-y|<\delta$ implies $|f_{1}(x)-f_{1}(y)|, |f_{2}(x)-f_{2}(y)|<\eps/4$. Set 
$\mathtt{s}=|\m(f_{1})-\m(f_{2})|+\one\{\m(f_{1})=\m(f_{2})\}$ and choose $N$ large enough that 
$\Delta:=\mathtt{s}/2^{N}<\delta$. Define the lattice 
\[
\LL=\{\m(f_{1})+k\Delta\}_{k\in \Z}.
\]
Note that $\m(f_{1}),\m(f_{2})\in \LL$. Set $a=2\Delta/\eps$, $\LL^{+}=\LL\cap [0,\infty)$,  
and let $\ell_{0}$ denote the smallest element of $\LL^{+}$. Observe that $0\leq \ell_{0}<\Delta$ by construction.

For $i=1,2$, define the function $\gamma_{i}:\LL\rightarrow \LL^{+}$ by 
\[
\gamma_{i}(\ell) = 
\begin{cases}
\ell_{0} & \text{if $\ell\leq \ell_{0}$} \\
\Delta\lceil (af_{i}(\ell))/\Delta \rceil + \ell_{0} & \text{otherwise}.
\end{cases}
\]
Note that $af_{i}$ changes by no more than $\Delta/2$ when the argument changes by no more than $\Delta$. In conjunction with 
the fact that $f_{i}\equiv 0$ on $(-\infty,0]$, $f_{i}\geq 0$, and $\ell_{0}\in [0,\Delta)$, this implies that $\gamma_{i}$ is an 
extended Motzkin path on $\LL$. That is, $\gamma_{i}(\ell)=\ell_{0}$ for all $\ell\in \LL\cap (-\infty,\ell_{0}]$ 
and for each $\ell,\ell'\in \LL$ with $|\ell-\ell'|= \Delta$, we have $\gamma_{i}(\ell)\geq \ell_{0}$ and 
$|\gamma_{i}(\ell)-\gamma_{i}(\ell')|\in \{0,\Delta\}$. 

Let $\varphi(x)=\Delta \cdot x+\ell_{0}$. Then $\varphi$ is an affine scaling which maps $\Z$ bijectively to $\LL$. Also 
define the affine scaling $\sigma(x)=(x-\ell_{0})/a$. By a slight abuse of notation, we will henceforth let $\gamma_{i}$ denote its 
extension to $\R$ by linear interpolation. Let $\Gamma_{i}\in C_{0}^{+}(\R)$ be the extended Motzkin path defined by 
$\Gamma_{i}=\varphi^{-1}\circ \gamma_{i}\circ \varphi$. 
Now define
\[
\bar{f_{i}}=(\gamma_{i}-\ell_{0})/a = \sigma\circ \varphi\circ \Gamma_{i} \circ \varphi^{-1} = \psi\circ \varphi^{\ast}(\Gamma_{i})
\]
where $\psi(x)=\sigma\circ \varphi(x)=\frac{\eps}{2}x$. Then $\psi(0)=\sigma(\ell_{0})=0$ and $\m(\bar{f}_{i})=\m(\gamma_{i})=\m(f_{i})$. 
For $x\in\LL$, a direct computation gives $|f_{i}(x)-\bar{f}_{i}(x)|<\eps/2$. For $x\notin\LL$, writing $\ell_{x}$ 
for the nearest lattice point to $x$ gives
\[
|f_{i}(x)-\bar{f}_{i}(x)|\leq |f_{i}(x)-f_{i}(\ell_{x})|+|f_{i}(\ell_{x})-\bar{f}_{i}(\ell_{x})|+|\bar{f}_{i}(\ell_{x})-\bar{f}_{i}(x)|
<\frac{\eps}{4} + \frac{\eps}{2} + \frac{\Delta}{2a} = \eps,
\]
hence $\lVert f_{i}- \bar{f}_{i} \rVert_{\infty}<\eps$ as desired.
\end{proof}

We are now ready to prove Lemma \ref{lemma:Lipschitz}.

\begin{proof}[\textbf{Proof of Lemma \ref{lemma:Lipschitz}}.]
Fix $j\geq 1$. To begin, we observe that it is enough to show the assertion for $I=\R$. Indeed, for any $I\subseteq\R$ and any 
$h\in C_{0}^{+}(I)$, we can define a function $\tilde{h}\in C_{0}^{+}(\R)$ which equals $h$ on $I$ and drops linearly to zero on $[b,b+1]$ 
where $b$ is the rightmost boundary point of $I$. This construction ensures that $\max\E^{j-1}(h)=\max\E^{j-1}(\tilde{h})$ and 
$\lVert h_{1}-h_{2} \rVert_{\infty}=\lVert \tilde{h}_{1}-\tilde{h}_{2} \rVert_{\infty}$.

Next we show that the result holds if the graphs of $f$ and $g$ are (extended) Motzkin paths by induction on 
$m=\max \E^{j-1}(f)+\max \E^{j-1}(g)$. The assertion is trivial when $j=1$ or $m=0$. If $\max \E^{j-1}(f)\geq 1$ and 
$\max \E^{j-1}(g)=0$, write $\m_{j}:=\m(\E^{j-1}(f))$. Let $J=[a,b]$ be the excursion interval of $\E^{j-1}(\Gamma)$ 
which contains $\m_{j}$. By Proposition \ref{prop:numberofhills}, $\Gamma-\E^{j-1}(\Gamma)$ is 
constant on the excursion intervals of $\E^{j-1}(\Gamma)$. 
Hence we get 
\[
\max \E^{j-1}(f) = \E^{j-1}(f)(\m_{j}) - \E^{j-1}(f)(a) = f(\m_{j}) - f(a).
\]
As $f,g\geq 0$, consideration of whether or not $g(\m_{j})\geq g(a)$ shows that
\begin{align*}
\max \E^{j-1}(f) & =f(\m_{j}) - f(a) \leq   f(\m_{j}) - f(a) + \left|g(\m_{j})-g(a)\right| \\
&  \leq \left|f(\m_{j})-g(\m_{j})\right| +  \left|f(a)-g(a)\right| \leq 2\lVert f-g \rVert_{\infty},
\end{align*}
By symmetry, the result also holds when $m\geq 1$ and $\max \E^{j-1}(f)=0$, so we may assume that both 
$\max \E^{j-1}(f)$ and $\max \E^{j-1}(g)$ are at least 1. As the maxima are necessarily 
attained on hill intervals, Proposition \ref{prop:hillandexcursioncommute}, the inductive hypothesis, and part (ii) of Proposition 
\ref{prop:Lipschitz} imply 
\begin{align*}
\left|\max\E^{j-1}(f)-\max\E^{j-1}(g)\right| & =\left|\max\HH\circ\E^{j-1}(f)-\max\HH\circ\E^{j-1}(g)\right|\\
& =\left|\max\E^{j-1}\circ\HH(f)-\max\E^{j-1}\circ\HH(g)\right|\\
& \leq 2\left\Vert \HH(f)-\HH(g)\right\Vert _{\infty}\leq 2\left\Vert f-g\right\Vert _{\infty}.
\end{align*}
This completes the proof for Motzkin paths. 

Now we show the assertion for $f,g\in C_{0}^{+}(\R)$ by induction on $j\geq 1$. The base case is tautological. For the inductive step, 
choose $\psi,\varphi,\Gamma_{1},\Gamma_{2},\bar{f}$, $\bar{g}$ as in Proposition \ref{prop:Motzkin_approx} with 
$f_{1}=f,f_{2}=g$. Then by the choice of $\bar{f}$, Proposition \ref{prop:hillandexcursioncommute}, the inductive hypothesis, 
and Proposition \ref{prop:Lipschitz} (i), we have
\begin{align*}
\left| \max\E^{k}(f) - \max\E^{k}(\bar{f}) \right|
& = \left| \max\E^{k-1}\circ \E_{\m(f)}(f) - \max\E^{k-1}\circ \E_{\m(f)}(\bar{f}) \right| \\
& \leq \lVert\E_{\m(f)}(f)-\E_{\m(f)}(\bar{f})  \rVert_{\infty} \\
& \leq  2\lVert f-\bar{f}  \rVert_{\infty} < 2\eps,
\end{align*}
and similarly for $g$.
Also, since $\psi(0)=0$, the triangle inequality gives
\begin{align*}
\psi( \lVert \Gamma_{1} - \Gamma_{2}\rVert_{\infty})
& = \psi(\lVert \varphi^{\ast}\Gamma_{1}-\varphi^{\ast}\Gamma_{2}\rVert_{\infty}) \\
& = \lVert \psi\circ\varphi^{\ast}\Gamma_{1}-\psi\circ\varphi^{\ast}\Gamma_{2}\rVert_{\infty} \\
& =\lVert \bar{f}-\bar{g}\rVert_{\infty} <  4\eps + \lVert f-g \rVert_{\infty}.
\end{align*}

Lastly, observe that the functional $\max \E^{k}$ satisfies
\[
\max  \E^{k}(\bar{f}_{i}) = \psi\circ \max \E^{k}(\Gamma_{i}). 
\]
Thus in conjunction with the assertion for the Motzkin paths, we obtain 
\begin{align*}
\left| \max\E^{k}(f) - \max\E^{k}(g) \right| 
& < 4\eps +  \left| \max\E^{k}(\bar{f}) - \max\E^{k}(\bar{g}) \right|   \\
& \leq 4\eps +  \psi\Big( \Big| \max\E^{k}(\Gamma_{1}) - \max\E^{k}(\Gamma_{2}) \Big| \Big)  \\ 
& \leq 4\eps + \psi\left(\lVert \Gamma_{1} - \Gamma_{2} \rVert_{\infty}\right) < 8\eps + \lVert f - g \rVert_{\infty}.
\end{align*} 
Letting $\eps \searrow 0$ completes the inductive step and the proof.
\end{proof}

\subsection{Statistics of 312-avoiding permutations}
\label{subsection:statistics of 312-avoiding permutations}

In this subsection, we provide proofs of Propositions \ref{prop:bijection_312-Dyck} and \ref{prop:permutation_Greene}. 

Recall that for each $n\geq 1$ and permutation $\tau\in \ess_{3}$ of length 3, we denote by $\ess_{n}^{\tau}$ the set of all $\tau$-avoiding 
permutations of length $n$. Also recall that $\text{Dyck}_{2n}$ denotes the set of all Dyck paths of length $2n$. Note that a permutation 
$\sigma$ is $312$-avoiding iff its inverse $\sigma^{-1}$ is $231$-avoiding. Also, if we denote by $\cev{\sigma}$ the \textit{reversal} of 
$\sigma$ obtained by reading $\sigma$ from right to left, then $\sigma$ is $231$-avoiding iff its reversal $\cev{\sigma}$ is $132$-avoiding. 

There are a number of bijections between $\tau$-avoiding permutations and Dyck paths in the literature. For instance, Krattenthaler 
\cite{krattenthaler2001permutations} obtained a bijection $\ess_{n}^{132}\rightarrow \text{Dyck}_{2n}$, and later Hoffman, Rizzolo, and Silvken 
\cite{hoffman2015pattern} used a bijection $\text{Dyck}_{2n}\rightarrow \ess_{n}^{231}$ to study random $231$-avoiding permutations in terms 
of random walks and Brownian excursions. In fact, the inverse of the latter bijection is the conjugation of the former by reversals of permutations 
and Dyck paths, where the reversal of a Dyck path is its left-right mirror image. In the forthcoming proof of Proposition \ref{prop:bijection_312-Dyck}, 
we will make use of the bijection $\text{Dyck}_{2n}\rightarrow \ess_{n}^{231}$ mentioned above, which we give below in a slightly more general version. 

For a given $h$-restricted Motzkin path $\Gamma$, we define a permutation $\sigma(\Gamma)$ as follows: Let $v_{k}$ be the location of the 
$k^{\text{th}}$ upstroke of $\Gamma$. (Thus if $\Gamma=\Gamma(X_{0})$, then $v_{k}$ is the location of the $k^{\text{th}}$ ball in $X_{0}$). 
Then we define a $231$-avoiding permutation $\sigma(\Gamma)$ by 
\begin{equation}
\label{eq:permutation_Motzkin}
\begin{aligned}
\sigma(\Gamma)(k)&= k + \frac{1}{2}\inf\big\{ r\geq 0 \st \Gamma_{v_{k}+r}  = \Gamma_{v_{k}-1} \big\} - \Gamma_{v_{k}} \\
&= k + \frac{1}{2}\sup\big\{ r\geq 0 \st \Gamma_{v_{k}+r}  \geq \Gamma_{v_{k}} \big\} +1 - \Gamma_{v_{k}}. 
\end{aligned} 
\end{equation}
When restricted to Dyck paths, this map $\digamma\mapsto \sigma(\digamma)$ is shown to be a bijection between $\textup{Dyck}_{2n}$ and 
$\ess^{231}_{n}$ in \cite[Thm. 4.3]{hoffman2015pattern}. 

\begin{figure*}[h]
	\centering
	\includegraphics[width=.95 \linewidth]{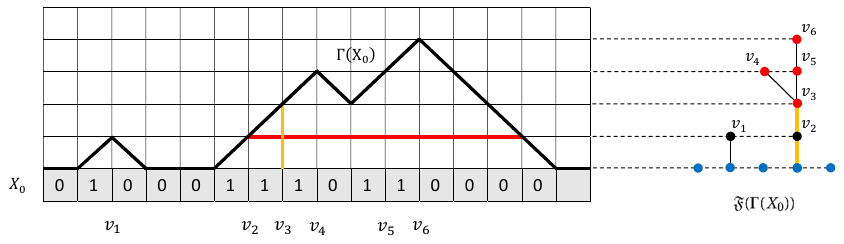}
	\caption{Construction of the inverse $1\,6\,5\,2\,4\,3$ of the 312-avoiding permutation $\sigma(X_{0})=1\,4\,6\,5\,3\,2$ directly from the 
		    corresponding Motzkin path $\Gamma(X_{0})$ and rooted forest $\F(\Gamma(X_{0}))$. On the left, the lengths of the red and orange 
		    paths correspond to the supremum and $\Gamma_{v_{k}}$ terms in \eqref{eq:permutation_Motzkin} for $k=3$. On the right, the subtree 
		    rooted at $v_{3}$ consists of the four red nodes and the level of $v_{3}$ is the number of edges in the orange path. Thus 
		    \eqref{eq:permutation_Motzkin} and \eqref{eq:permutation_Forest} each show that $\sigma^{-1}(3)=5$.}
	\label{fig:permutation}
\end{figure*}

\begin{remark}
\textup{
For a given rooted forest $\F$, a permutation $\sigma(\F)$ can be defined similarly: Let $v_{k}$ be the $k^{\text{th}}$ non-root node in $\F$ according 
to the depth-first order and define  
\begin{equation}
\label{eq:permutation_Forest}
\sigma(\F)(k) := k + \#{\{ \text{nodes in the subtree of $\F$ rooted at $v_{k}$} \} } - \text{(level of $v_{k}$ in $\F$)}.
\end{equation}
Note that the maps \eqref{eq:permutation_Forest} and \eqref{eq:permutation_Motzkin} yield the same permutation for corresponding rooted forest and 
its contour process. Namely, let $\Gamma$ be the $h$-restricted Motzkin path which is a contour process of $\F$. Then $\sigma(\Gamma) = \sigma(\F)$; 
see Figure \ref{fig:permutation} for an illustration. 
}
\end{remark}

\begin{proof}
[\textbf{Proof of Proposition \ref{prop:bijection_312-Dyck}}.] 
We define a map $\varphi:\textup{Dyck}_{2n}\rightarrow \ess^{312}_{n}$ by
\[
\digamma \mapsto \sigma(\digamma) \mapsto \sigma(\digamma)^{-1},
\]
where the first map is given by \eqref{eq:permutation_Motzkin}. As a composition of two bijections, $\varphi$ is a bijection from 
$\textup{Dyck}_{2n}$ to $\ess^{312}_{n}$. This shows (i). 

To show (ii), fix $\digamma\in \textup{Dyck}_{2n}$ and let $X_{0}$ be the box-ball configuration obtained from $\digamma$ by 
\[
X_{0}(i) = \one\big\{\digamma(i+1)-\digamma(i)=1\big\}
\]
for all $i\geq 0$. It then suffices to show that
\begin{equation}
\label{eq:pf_permutation_appendix}
\sigma(X_{0})= \sigma(\digamma)^{-1}. 
\end{equation} 	
To this end, label the balls $1,\ldots,n$ from left to right, and recall the push-pop stack construction $X_{0}\mapsto \sigma(X_{0})$ described in Section 
\ref{section:Permutation}. Fix a label $1\leq k \leq n$. We are going to track the trajectory of ball $k$ during the push-pop stack construction. 
Using the notation from Equation \eqref{eq:permutation_Motzkin}, let ball $k$ be at site $v_{k}$. Note that $\digamma_{v_{k}}$ equals the number of balls 
in the stack after ball $k$ is pushed onto it. Hence the number of balls which have been popped off in previous steps equals $k-\digamma_{v_{k}}$. 
Next, while the stack sweeps sites to the right of $v_{k}$, balls with larger labels will be pushed on and popped off until ball $k$ is finally deposited. This 
happens precisely when $\digamma$ first hits height $\digamma_{v_{k}}-1$ after location $v_{k}$. Accordingly, the number of balls that are deposited 
during the period when ball $k$ is in the stack equals the height of the subexcursion of $\digamma$ started at $v_{k}$, which equals to half of the duration 
of this excursion. Thus 
\[
\text{$\#$ balls popped out before ball $k$} = k - \digamma_{v_{k}} + \frac{1}{2}\sup\big\{ r\geq 0 \st \digamma_{v_{k}+r}  
\geq \digamma_{v_{k}} \big\}. 
\]
Therefore, $\sigma(\digamma)(k)$, which is one more than the above quantity, is the position of $k$ in $\sigma(X_{0})$ as desired. 
\end{proof} 

\begin{proof}[\textbf{Proof of Proposition \ref{prop:permutation_Greene}.}]	
Before we begin, recall the definition of the longest `leftmost'  increasing and `rightmost' decreasing subsequences $\tau_{+}^{\ast}$ and $\tau_{-}^{\ast}$ 
given above the statement of Proposition \ref{prop:permutation_Greene}.
	
We first show the assertion for $\tau_{+}^{\ast}$. By induction on the length of the permutation, we suppose that the assertion holds for all $312$-avoiding 
permutations of length less than $n$ for some $n\geq 3$, and fix a $312$-avoiding permutation $\tau$ of length $n$. (The result is true by inspection when 
$n=3$.) Using Proposition \ref{prop:bijection_312-Dyck}, choose a box-ball configuration $X_{0}$ and a Dyck path $\digamma$ such that $\tau=\sigma(X_{0})$ 
and $\digamma=\Gamma(X_{0})$.   

By Greene's theorem (\cite{greene1982extension}), we know that the length of the first row of $\mathtt{RS}(\tau)$ equals the length of any longest 
increasing subsequence in $\tau$. Since  $\mathtt{RS}(\tau)=\Lambda(\digamma)$, we see that the length of the longest increasing subsequence of 
$\tau$ equals the number of peaks in $\digamma$.    

Let $X_{0}'$ be the box-ball configuration obtained from $X_{0}$ by deleting all $1\,0$ patterns from $X_{0}$, as in the proof of Lemma 
\ref{lemma:invariance_Youngdiagram}, and let $\Gamma'=\Gamma(X_{0}')$ and $\tau'=\sigma(X_{0}')$ be the $h$-restricted Motzkin path and 
$312$-avoiding permutation constructed from $X_{0}'$ (see the commutative diagram \eqref{diagram}). It is easy to see that $\Gamma'$ can be directly 
obtained from $\Gamma$ by first applying the hill-flattening operator $\HH$ and then contracting new $h$-strokes which are not at height $0$. 

On the other hand, let $L$ be the number of $1\,0$ patterns in $X_{0}$, which is the same as the number of peaks in $\Gamma$. When reading $X_{0}$ 
from left to right, let $\ell_{i}$ be the label of the ball that corresponds to the `1' in the $i^{\text{th}}$ $1\,0$ pattern. 
Then $\tilde{\tau}:=\ell_{1}\ell_{2}\cdots\ell_{L}$ is an increasing subsequence in $\sigma$ satisfying $\tau'=\tau\setminus \tilde{\tau}$. Moreover, 
Greene's theorem shows that this is a longest increasing subsequence. 
By an easy induction argument, one sees that $\ell_{i+1}$ is the first number to the right in $\sigma$ that exceeds $\ell_{i}$. Thus by definition, 
$\tilde{\tau}=\tau_{+}^{\ast}$, is the `leftmost' longest increasing subsequence in $\sigma$. From the construction it is clear that 
$\tau' = \tau\setminus \tau^{\ast}_{+}$. 

\vspace{-0.2cm}
\begin{equation}
\label{diagram}
\begin{gathered}
\xymatrix{   
	\tau \ar@{..>}[d] & \ar[l] X_{0} \ar[r] \ar[d] & \Gamma \ar@{..>}[d] \\
	\tau' & \ar[l] X_{0}' \ar[r]  & \Gamma' 
}
\end{gathered}
\end{equation}	

To complete the argument, recall that $\Lambda(\Gamma')$ is obtained from $\Lambda(\Gamma)$ by deleting the first row. Since 
$\mathtt{RS}(\tau)=\Lambda(\Gamma)$ and $\mathtt{RS}(\tau')=\Lambda(\Gamma')$ by Proposition \ref{prop:permutation_Motzkin}, we have that 
$\mathtt{RS}(\tau')$ is obtained from $\mathtt{RS}(\tau)$ by deleting its first row. Since $\tau'$ can be obtained from $\tau$ by deleting a longest 
increasing subsequence, the inductive hypothesis applied to $\tau'$ completes the proof.  

Next, we show the assertion for the columns. Let $\tau$, $X_{0}$, $\Gamma$ be as before. To begin, observe that in the stack construction of $\tau$ 
from $X_{0}$, every decreasing subsequence in $\tau$ is generated by the balls that occupy the stack at the same time. (For instance, in Figure 
\ref{fig:stack}, the decreasing subsequence $432$ in $\sigma=14632$ is generated by the balls in the stack on top of the ball of label 4.) Thus, every 
longest decreasing subsequence in $\tau$ is generated by the balls in the stack where $\Gamma$ achieves its maximum. 

Let $\m^{\ast}$ be any location where $\Gamma$ attains its global maximum. During the stack operation to construct $\tau$ from $X_{0}$, let 
$\bar{\tau} = \ell_{1}\ell_{2}\cdots\ell_{M}$ be the decreasing sequence consisting of the numbers in the stack after pushing all the balls over the interval 
$[1,\m^{\ast}]$. This is a longest decreasing subsequence in $\tau$. Denote $\tau^{\dagger} = \tau\setminus \tau^{\ast}_{-}$. Let $X^{\dagger}_{0}$ 
be the box-ball configuration that is obtained by converting $1$'s that correspond to balls with labels in $\bar{\tau}$ to $0$'s. Then observe that 
$\sigma(X^{\dagger}_{0}) = \tau^{\dagger}$ and $\Gamma^{\dagger} = \E_{\m^{\ast}}(\Gamma)$, where $\E_{\m^{\ast}}$ is the excursion operator 
pivoted at location $\m^{\ast}$ instead of the rightmost one $\m$. According to Lemma \ref{lemma:columnbyexcursionop} and the following remark, 
$\Lambda(\E_{\m^{\ast}}(\Gamma))$ is obtained by deleting the first column of $\Lambda(\Gamma)$. 
Since $\Lambda(\E_{\m^{\ast}}(\Gamma)) = \mathtt{RS}(\tau)$ and $\Lambda(\Gamma)= \mathtt{RS}(\tau^{\dagger})$, the assertion follows.        
\end{proof}

\section*{Acknowledgments}

We thank Karthik Karnik, Thomas Lam, Yuval Peres, Pavlo Pylyavskyy, and Mikaeel Yunus for inspiring conversations.

\vspace{0.3cm}

\small{
	\bibliographystyle{plain}
	\bibliography{mybib}
}

\end{document}